\documentclass[10pt,a4paper]{amsart}
\usepackage[margin=16mm]{geometry}

\usepackage{epsfig}
\usepackage{amssymb,amsmath,amsthm,graphicx}
\let\amslrcorner\lrcorner
\usepackage{amscd,pb-diagram} 
\usepackage[all]{xypic}
\usepackage[numbers]{natbib}
\usepackage[colorlinks]{hyperref}
\usepackage{MnSymbol}

\let\lrcorner\amslrcorner

\newtheorem{theorem}{Theorem}[section]

\newtheorem{corollary}{Corollary}[section]

\newtheorem{example}{Example}[section]

\newtheorem{lemma}{Lemma}[section]
\newtheorem{notation}{Notation}[section]
\newtheorem*{warning}{Warning}

\newtheorem{proposition}{Proposition}[section]

\theoremstyle{definition}
\newtheorem{definition}{Definition}[section]
\newtheorem{remark}{Remark}[section]

\newcommand{\x}{\boldsymbol{x}}

\newcommand{\Z}{\mathbb{Z}}
\newcommand{\R}{\mathbb{R}}
\newcommand{\C}{\mathbb{C}}

\newcommand{\p}{\mathbb{P}}

\newcommand{\g}{\mathfrak g}
\newcommand{\gh}{\mathfrak h}

\newcommand{\X}{\mathfrak X}

\newcommand{\gl}{\mathfrak{gl}}

\newcommand{\sll}{\mathfrak{sl}}

\newcommand{\stab}{\mathfrak{stab}}

\newcommand{\Span}[1]{\left\langle#1\right\rangle}
\DeclareMathOperator{\LL}{LGr}
\DeclareMathOperator{\LG}{LG}
\DeclareMathOperator{\LFl}{LFl}

\DeclareMathOperator{\diag}{diag}
\DeclareMathOperator{\tr}{tr}

\DeclareMathOperator{\id}{id}
\DeclareMathOperator{\rank}{rank}
\DeclareMathOperator{\vol}{vol}

\DeclareMathOperator{\Graph}{graph}

\DeclareMathOperator{\SO}{\mathsf{SO}}

\DeclareMathOperator{\GL}{\mathsf{GL}}

\DeclareMathOperator{\SL}{\mathsf{SL}}
\DeclareMathOperator{\Sp}{\mathsf{Sp}}
\DeclareMathOperator{\CSp}{\mathsf{CSp}}

\DeclareMathOperator{\Gr}{Gr}
\DeclareMathOperator{\sign}{sign}

\DeclareMathOperator{\End}{End}

\DeclareMathOperator{\Hom}{Hom}

\DeclareMathOperator{\Stab}{Stab}

\DeclareMathOperator{\Sing}{Sing}
\DeclareMathOperator{\ind}{ind}
\newcommand{\OO}{\mathcal{O}}
\newcommand{\E}{\mathcal{E}}

\newcommand{\V}{\mathcal{V}}
\newcommand{\CC}{\mathcal{C}}

\newcommand{\NN}{\mathcal{N}}

\newcommand{\f}{\theta_1}
\newcommand{\ff}{\theta_2}
\newcommand{\fk}{\theta_k}

\newcommand{\Th}{^\textrm{th}}

\newcommand{\Nd}{^\textrm{nd}}
\newcommand{\ins}{\lrcorner}

\usepackage{mathrsfs}

\let\LG\LL

\def\sL{\mathcal{L}}
\def\sp{\mathfrak{sp}}
\let\rk\rank

\begin{document}

\title{The moment map on the space of symplectic 3D Monge--Amp\`ere equations}

\author{Jan Gutt}
\address{Center for Theoretical Physics of the Polish Academy of Sciences,
Al. Lotnik\'ow 32/46,
02-668 Warsaw,
Poland}
\email{jan.gutt@gmail.com}
\author{Gianni Manno}
\address{Dipartimento di Matematica ``G. L. Lagrange'', Politecnico di Torino, Corso Duca degli Abruzzi, 24, 10129 Torino, Italy.}
\email{giovanni.manno@polito.it}
\author{Giovanni Moreno}
\address{Department of Mathematical Methods in Physics,
Faculty of Physics, University of Warsaw,
ul. Pasteura 5, 02-093 Warsaw, Poland}
\email{giovanni.moreno@fuw.edu.pl}
\author{Robert Śmiech}
\address{Faculty of Mathematics, Informatics, and Mechanics, University of Warsaw,
Stefana Banacha 2, 02-097 Warsaw, Poland}
\email{r.smiech@mimuw.edu.pl}
\date{\today}

\maketitle

\begin{abstract}

For any $2\Nd$ order scalar PDE $\mathcal{E}$ in one unknown function, that we interpret as a hypersurface of a $2\Nd$ order jet space $J^2$, we construct,  by means of the characteristics of $\mathcal{E}$, a sub--bundle of the contact distribution of the underlying contact manifold $J^1$, consisting of conic  varieties. We call it the \emph{contact cone structure} associated with  $\mathcal{E}$. 
We then focus on symplectic Monge--Amp\`ere equations in 3 independent variables,  that are naturally parametrized by a 13--dimensional real projective space. If we pass to the field of complex numbers $\C$, this projective space turns out to be the projectivization of the 14--dimensional irreducible representation of the simple Lie group $\Sp(6,\C)$: the associated moment map allows to define a rational map $\varpi$ from the space of symplectic 3D Monge-Amp\`ere equations to the projectivization of the space of quadratic forms on a $6$--dimensional symplectic vector space. We study in details the relationship between the zero locus of the image of $\varpi$, herewith called the \emph{cocharacteristic variety}, and the contact cone structure of a 3D Monge-Amp\`ere equation $\E$: under the hypothesis of non--degenerate symbol, we prove that these two constructions coincide. A key tool in achieving such a result will be a complete list of mutually non--equivalent quadratic forms on a $6$--dimensional symplectic space, which has an interest on its own.

\end{abstract}
 
\noindent\textbf{MSC 2020:} 35A30; 58A20; 58J70


%
\setcounter{tocdepth}{1}

\tableofcontents

\section{Introduction}\label{sec:intro}

\subsection{Starting point: $2\Nd$ order PDEs and their symbol}\label{sec:starting}

A \emph{(scalar)
$2\Nd$ order PDE in one unknown function $u=u(x^1,\ldots,x^n)$ and $n$ independent variables $\x=(x^1,\ldots,x^n)$}, henceforth called $n$--dimensional PDEs ($n$D PDEs), can be written as
\begin{equation}\label{eq:equation}
\E:=\big\{\,F(\x,u,\nabla u,\nabla^2 u)=F(x^i,u,u_i,u_{ij})=0\big\}\, ,
\end{equation}
where $F$ is a real function on a domain of
\begin{equation}\label{eq:set.Omega}
\R^n\times\R\times\R^n\times \R^{\frac{n(n+1)}{2}} 
\end{equation}
and
\begin{equation}\label{eqDefDer}
    u_i:=\frac{\partial u}{\partial x^i}\,, \quad u_{ij}:=\frac{\partial^2 u}{\partial x^ix^j}\,.
\end{equation}
Equation \eqref{eq:equation} is \emph{elliptic} at a point 
$$
\ff=\big(\x_0,u(\x_0),\nabla u(\x_0),\nabla^2u(\x_0)\big)=(x^i_0,u_0,u^0_i,u^0_{ij})
$$ 
of the space \eqref{eq:set.Omega} lying in the subset $\E$ given by   \eqref{eq:equation}, if the matrix
\begin{equation}\label{eq:matrix.vertical.diff}
\frac{1}{2-\delta_{ij}}\left.\frac{\partial F}{\partial u_{ij}}\right|_{\ff}
\end{equation}
is definite (either positive or negative), i.e.,  if the quadratic form
\begin{equation}\label{eq:symbol.intro}
\sum_{i\leq j}\left.\frac{\partial F}{\partial u_{ij}}\right|_{\ff}\eta_i\eta_j\, ,
\end{equation}
that we call also the \emph{symbol} of equation \eqref{eq:equation} at the point $\ff\in\E$,
is either greater or less than zero for all vectors $(\eta_1,\dots,\eta_n) \in\R^n\smallsetminus\{0\}$.\par

A point $\ff\in\mathcal{E}$ is called \emph{regular} if the matrix \eqref{eq:matrix.vertical.diff} is not zero. 
In Sections \ref{sec:intro}--\ref{secBackRecepi}, points of $\mathcal{E}$ are always assumed to be regular, unless otherwise specified.\par

\subsection{Context}

It is well known that the notion of symbol of a $2\Nd$ order PDE $\mathcal{E}$, more precisely its rank, is closely related to the notion of characteristic of $\mathcal{E}$. 
There are PDEs that are completely characterized by the behavior of their characteristics: for instance, $2$D hyperbolic (resp., parabolic) Monge--Amp\`ere equations are those $2\Nd$ order PDEs whose characteristic lines are arranged in a couple of different (resp. coincident) $2$-dimensional linear subspaces (see, for instance, \cite{MR2985508,MR2503974} and reference therein). One can ask if a similar phenomenon occurs also in the multidimensional situation.
The oldest paper regarding a multidimensional generalization of Monge--Amp\`ere equations, to the authors' best knowledge, dates back to a 1899 work   by Goursat \cite{MR1504329}, where the  Monge--Amp\`ere equations with an arbitrary number of independent variables were introduced as those PDEs whose characteristic cones ``degenerate'' into linear subspaces. This phenomenon corresponds to the degeneration of the the symbol of the Monge--Amp\`ere equations. The equations obtained by Goursat are indeed of type
\begin{equation*}
\det\|u_{ij}-b_{ij}\|=0\,, \quad b_{ij}=b_{ij}(x^{1},\dots,x^{n},u,u_{1},\dots,u_{n})\, ,
\end{equation*}
and a straightforward computation shows that the rank of their symbols is less or equal to $2$: as such, these PDEs are a proper subclass of a larger class of Monge--Amp\`ere equations that were introduced later by Boillat in \cite{MR1139843}, as the only PDEs whose characteristic velocities behave in the ``completely exceptional" way 
in the sense of   P. Lax. \cite{MR0481580,BoillatDafermosLaxEtAl1996,MR3603758,MR0068093,MR0067317,Moreno2017}.




V.~Lychagin, who proposed studying $2\Nd$ order PDEs by means of the underlying contact geometry of the $(2n+1)$--dimensional first--order jet space $J^1$ (i.e., the space locally parametrized by $(x^i,u,u_i)$,   defined in \cite{MR525652} $n$--dimensional Monge--Amp\`ere equations in terms of certain differential $n$--forms on the $2n$--dimensional contact distribution $\CC$ of $J^1$, which he called \emph{effective $n$--forms} and whose set we will be denoting by $\Lambda^n_0(\CC^*)$; see also \cite{MR2352610}. Lychagin's approach leads directly to the general expression of an $n$--dimensional Monge--Amp\`ere equation:
\begin{equation*}
    M_n+M_{n-1}+\dots+M_0=0\, ,
\end{equation*}
where $M_k$ is a linear combination  of all $k\times k$ minors of the Hessian matrix $\|u_{ij}\|$, with coefficients in $C^\infty(J^1)$. Moreover, the equivalence problem for Monge--Amp\`ere equations, thanks to the aforementioned correspondence, can be recast in terms of effective forms, which is especially advantageous in the case of \emph{symplectic} Monge--Amp\`ere equations, i.e., Monge--Amp\`ere equations of type\footnote{Such equation are known also as ``Hirota type'', see, e.g.,  \cite{DoubrovFerapontov,Russo2019,MR2805306}.}
\begin{equation}\label{eq:sympl.hirota}
    F(u_{ij})=0\, ,
\end{equation}
with $F$ not depending neither on $x^i$, $u$, nor $u_i$. 
For instance, in the case $n=3$, the authors of \cite{MR2352610} associated with any effective $3$--form $\Phi\in\Lambda^3_0(\CC^*)$  the following symplectically invariant quadratic form on $\CC$:
\begin{equation}\label{eq:mappa.KLR}
 \mathrm{trace}(\omega^{-1}\circ \Phi)^2\,,
\end{equation}
where $\omega$ is a prescribed representative of the conformal symplectic structure of $\CC$, and employed it to obtain the normal forms of symplectic 3D Monge--Amp\`ere equations with non--degenerate symbol. A peculiar feature of the 3--dimensional case is that the 20--dimensional vector space $\Lambda^3(\CC_{\f}^*)$, where $\f\in J^1$, is equipped with a natural symplectic structure, which makes the natural $\Sp(\CC_{\f})$--action a Hamiltonian one;    the corresponding moment map, firstly studied by  N.~Hitchin  in \cite{Hitchin3Forms}, descends to a quadratic map between  the 14--dimensional vector space $\Lambda_0^3(\CC_{\f}^*)$ and the Lie algebra $\sp(\CC_{\f})$, which in turns identifies naturally with $S^2(\CC_{\f}^*)$, the space of quadratic forms on $\CC_{\f}$. It has been proved in  \cite{Banos2002}  that Hitchin's restricted moment map and the quadratic form \eqref{eq:mappa.KLR} lead to the same quadric in $\CC_{\f}$, see also \cite[Prop. 8.1.5]{MR2352610}.

Yet another possibility of seeing symplectic $n$--dimensional  Monge--Amp\`ere equations is as \emph{hyperplane sections} of the Lagrangian Grassmannian $\LL(n,2n)$: such is the perspective  adopted, among others, by E.~Ferapontov and his collaborators, who were mainly interested in the hydrodynamic integrability property  of  PDEs of the form  \eqref{eq:sympl.hirota}, see  \cite{MR2587572,Ferapontov2020}.  One of the results of \cite{MR2587572} is the existence of a ``master equation'' in the class of symplectic hydrodynamically integrable $2\Nd$ order PDEs in 3 independent variables: in terms of group actions, this means that the 21--dimensional group $\Sp(6,\R)$, that  acts naturally on the Lagrangian Grassmannian $\LG(3,6)$, has an open orbit in the space that parametrises  such PDEs. Another result is that the intersection of the latter  space with that parametrizing (symplectic 3D) Monge--Amp\`ere equations turns out to be the class of $\Sp(6,\R)$--linearizable Monge--Amp\`ere equations.


\subsection{Structure of the paper and description of the main results}

In Section \ref{secPRELIM}, after refreshing the basics concerning the characteristics and the symbol of a scalar $2\Nd$ order PDE in $n$ independent variables, we introduce the $k$--order jet spaces $J^k$ of (smooth) functions in $n$ independent variables  and then we interpret $2\Nd$ order PDEs as  hypersurfaces of $J^2$; the fibers of the latter are open dense subsets of the Lagrangian Grassmannian $\LL(n,2n)$. Next, we define a general (symplectic) Monge--Amp\`ere equation as a hyperplane section of a Lagrangian Grassmannian and, in the particular case  $n=3$, we define the \emph{cocharacteristic variety} of a (3D) Monge--Amp\`ere equation as the zero locus of the aforementioned map \eqref{eq:mappa.KLR}. We also introduce the notion of a \emph{contact cone structure on $J^1$}, that is  an assignment  $\f\in J^1\to\V_{\f}\subset\CC_{\f}$ of affine conic varieties, where $\CC$ is the contact distribution on $J^1$; \emph{cone structures} have been long known, both in the real--differentiable and in the complex--analytic contexts, and have recently risen to a certain attention as they cast a bridge between the two categories, e.g., in the works of J.--M.~Hwang  \cite{hwang_Mori_meets_Cartan}: a \emph{contact} cone structure can be then regarded as a cone structure that is compatible with a preexisting contact structure.\par 

In Section \ref{secForwardRecepi} we show that the notion of symbol, of a characteristic line, and  of a characteristic hyperplane, when properly framed in the theory of $2\Nd$ order PDEs based on  the contact manifold $J^1$, can be naturally combined to construct a contact cone structure on $J^1$, which  we call the \emph{contact cone structure associated with the considered PDE}. This supplies a common footing to both Goursat's idea of defining a Monge--Amp\`ere equation via a linear sub--distribution of the contact distribution, and the KLR invariant \eqref{eq:mappa.KLR}. \par

In Section \ref{secSezioneNuova} we narrow our attention to the class of (symplectic) 3D Monge--Amp\`ere equations and we compute both the contact cone structure and the cocharacteristic variety of four particular equations that represents almost all possible $\Sp(6,\R)$--equivalence classes of such PDEs---or all of them, without   ``almost'', once we will have passed to the field of complex numbers.  These computations will be used to prove that, for a Monge--Amp\`ere equation with non--degenerate symbol, the notions of contact cone structure and the notion of cocharacteristic variety are the same (in particular, they are quadric hypersurfaces of the contact plane of $J^1$), whereas for Monge--Amp\`ere equations with degenerate symbol, the cocharacteristic variety and the contact cone structure are quite different.
Such results are contained in Theorem \ref{thMain1}, which is reformulated in  Section \ref{secLinSec} over the field of complex numbers: Corollary \ref{corEquivalencyMomentCoChar} represents indeed a coarse proof of Theorem \ref{thMain1}.\par

In Section \ref{secBackRecepi} we give an answer to the natural question whether it is possible to revert the above procedure, i.e., to construct a 3D $2\Nd$ order PDE starting from a contact cone structure in dimension seven and to what extent the correspondence between a PDE and its contact cone structure is one--to--one. This leads to the study of the integral manifolds of certain vector distributions defined on the fiber of the $2\Nd$ order jet space $J^2$. As an example of computations we  consider the contact cone structures associated with the four particular Monge--Amp\`ere equations studied in Section \ref{secSezioneNuova} above.\par


The results of Sections \ref{secSezioneNuova} (in particular Theorem \ref{thMain1}) and Section  \ref{secBackRecepi} show that the contact cone structures of 3D Monge--Amp\`ere equations form a narrow sub--class of the class of all quadrics in the contact distribution $\CC$: even a dimensional inspection shows that the space of 3D symplectic Monge--Amp\`ere equations is 13--dimensional, whereas the space of all quadrics in $\CC_{\f}$ is 20--dimensional.
We also stress that, to obtain a PDE from a contact cone structure, as showed in Section \ref{secBackRecepi}, one generally has to pass through an integration procedure: such a procedure would be very difficult to carry out successfully without a complete list of normal  forms of quadratic forms on $\CC$ with respect to the action of the symplectic group $\Sp(\CC_{\f})=\Sp(6,\R)$.
This motivates looking for the normal forms of the quadratic forms on a $6$--dimensional symplectic space, up to symplectic equivalence.
This classification problem is classical, since, in view of the identification of $\sp(\CC_{\f})$ with $S^2(\CC^*_{\f})$, it coincides with the classification problem of adjoint orbits in a semisimple Lie algebra, about which there is a lot of excellent literature (besides the classical book \cite{Collingwood2017} we mention \cite{BURGOYNE1977339, DeBacker2004, Humphreys} and related works, such as \cite{Winternitz83,Nevins2011}), even though an explicit list of normal forms seems yet to be missing.

After fixing some notation and introducing some necessary tools in Section \ref{secMAESpace}, 
in Section  \ref{SecNormFormQadr} we indeed give, over the field of complex numbers, a complete list of mutually non--equivalent (up to symplectic transformations) quadratic forms on a $6$--dimensional symplectic space.

In Section \ref{secMomentMap} the Hitchin's moment map is reviewed 
and its equivalence with the symplectically invariant form \eqref{eq:mappa.KLR} is proved. Next equivalence, that is, the one with the cocharacteristic variety, is proved in Section \ref{secLinSec}, together with a review of the four geometries of the hyperplane sections of the Lagrangian Grassmanian $\LG(3,6)$, which is, in part, already contained in \cite[Proposition 2.5.1]{Iliev2005}. \par

\subsection*{Notation and conventions}
The Einstein convention for repeated indices will be used throughout the text, unless otherwise specified.
The linear span of vectors $v_1,\dots,v_k$ is denoted by $\Span{v_1,\dots,v_k}$. The cofactor matrix of a matrix $A$ is denoted by $A^\sharp$. The linear dual of a linear space (real or complex) $V$ is denoted by $V^*$, whereas $X^\vee$ stands for the \emph{projective dual} of the projective variety $X$. Symbol $X_{\textrm{sm}}$ stands for the subset of smooth points of a variety $X$. 
By a vector distribution $\mathcal{X}$ on a manifold $M$ we mean a smooth assignment $p\in M \to \mathcal{X}_p\subseteq T_pM$ (not necessarily of constant rank). We say that $\mathcal{X}$ is \emph{integrable} if it is such in the Frobenius sense, i.e., if  $[Y,Z]\in\mathcal{X}$ for any vector fields $Y,Z\in\mathcal{X}$, where $[Y,Z]$ is the Lie bracket of $Y$ and $Z$.

%
\subsection*{Acknowledgements} 
G.~Moreno is supported by the National Science Center, Poland, project ``Complex contact manifolds and geometry of secants'', 2017/26/E/ST1/00231.
R.~Śmiech is supported by the ``Kartezjusz'' program.
G. Manno gratefully acknowledges support by the project ``Connessioni proiettive, equazioni di Monge-Amp\`ere e sistemi integrabili'' (INdAM), ``MIUR grant Dipartimenti di Eccellenza 2018-2022 (E11G18000350001)'' and PRIN project 2017 ``Real and Complex Manifolds: Topology, Geometry and holomorphic dynamics'' (code 2017JZ2SW5). G. Manno is a member of GNSAGA of INdAM.

\noindent
The authors thank J.~Buczyński and A.~Weber for keen and useful remarks during the preparation of the manuscript.

\section{Preliminaries}\label{secPRELIM}

\subsection{Cauchy data and characteristics of $2\Nd$ order PDEs}


In the present section, as well as in the Sections \ref{secForwardRecepi}, \ref{secSezioneNuova} and \ref{secBackRecepi}, we mainly deal with PDEs \eqref{eq:equation} of \emph{non--elliptic type}, i.e., equations that are non--elliptic in the subset $\E$ of \eqref{eq:set.Omega}; indeed, non--elliptic PDEs (and, in particular, the hyperbolic ones) are abundant in \emph{real} characteristics: this is a choice of pure convenience, just to visually explain, via tangible examples, how to use the characteristics to construct our main object, that is the  quadric cone structure associated with a PDE. Starting from Section \ref{secMAESpace}  we switch to the field of \emph{complex} numbers and the very notion of ellipticity becomes meaningless.    

Let us consider the space consisting  of the first three factors appearing in \eqref{eq:set.Omega}, that is, the space with coordinates $(x^i,u,u_i)$: 
a \emph{Cauchy datum} is a particular $(n-1)$--dimensional submanifold of such a  space; one way to give it explicitly is via a    parametrization:
\begin{equation}\label{eq:Cauchy.datum}
\Phi(\mathbf{t})=\big(x^{1}(\mathbf{t}),\dots,x^n(\mathbf{t}),u(\mathbf{t}),u_{1}(\mathbf{t}),\dots,u_n(\mathbf{t})\big)\,,
\,\,\, \mathbf{t}=(t_1,\dots,t_{n-1})\in\mathbb{R}^{n-1}\, ,
\end{equation}
and this is the point of view we will adopt in this paper. By a ``particular submanifold" we meant that the functions $u_i(\mathbf{t})$ cannot be arbitrarily chosen: one should take into account that each $u_i$ must be the derivative of $u$ with respect to  $x^i$, cf. \eqref{eqDefDer}; this heuristic requirement will be made formal in Section \ref{secForwardRecepi}: as we shall see, the general condition for \eqref{eq:Cauchy.datum} to be a Cauchy datum can be given geometrically in terms of submanifolds of the (natural) contact structure the $(2n+1)$--dimensional $(x^i,u,u_i)$--space is equipped with.\par

Given a Cauchy datum \eqref{eq:Cauchy.datum}, we can formulate a  \emph{Cauchy problem}: this entails   finding a solution  $u=f(x^1,\dots,x^n)$ to \eqref{eq:equation} under the additional requirement that it satisfies the conditions
\begin{equation}\label{eq:Cauchy.problem}
f\big(x^1(\mathbf{t}),\dots,x^n(\mathbf{t})\big)=u(\mathbf{t})\,,\,\,\,
\frac{\partial f}{\partial x^{i}}
\big(x^1(\mathbf{t}),\dots,x^n(\mathbf{t})\big)=u_i(\mathbf{t})\,, \quad\forall \mathbf{t}\in\R^{n-1}\, .
\end{equation}
If the Cauchy datum \eqref{eq:Cauchy.datum} is \emph{non--characteristic} for the PDE \eqref{eq:equation},
then all the derivatives of arbitrary order of $u=f(x^1,\dots,x^n)$ are determined
by conditions \eqref{eq:Cauchy.problem} and \eqref{eq:equation}. This means that the full  Taylor expansion of $f$ is well determined, i.e., there exists a unique \emph{formal} solution; otherwise, the Cauchy datum \eqref{eq:Cauchy.datum} is called \emph{characteristic}. Of course the above definitions and reasonings  can be localized in a neighborhood of a considered point. The classical literature about geometry of PDEs and their characteristics  comprises, among others, \cite{MR1670044,Petrovsky1992}; see also the recent reviews   \cite{MR3760967,Vitagliano2014}.

\begin{example}\label{ex:introduction}
Consider the wave equation in one spacial  dimension:
\begin{equation}\label{eq:wave.2}
u_{12}=0\,.
\end{equation}
The Cauchy datum
\begin{equation}\label{eq:ex.char.wave.2}
u(x^1,0)=x^1\,,\,\,
u_1(x^1,0)=1\,,\,\,
u_2(x^1,0)=0\, ,\quad x^1\in\R\, ,
\end{equation}
can then be seen as a (parametric) curve $\Phi(x^1)$ in the $(x^1,x^2,u,u_1,u_2)$--space:
\begin{equation}\label{eq:CurvPhi.2}
\Phi(x^1)=\big(x^1,0,x^1,1,0\big)\, .
\end{equation}
It is characteristic for \eqref{eq:wave.2}, since we cannot determine all the derivatives of arbitrary order of a solution along the curve \eqref{eq:ex.char.wave.2}; indeed, the function $f(x^1,x^2)=x^1+k (x^2)^2$
is a solution to \eqref{eq:wave.2}, whose first jet (see below) along the curve $x^1\mapsto (x^1,0)$ coincides with \eqref{eq:CurvPhi.2} for each $k$. On the contrary, the Cauchy datum
\begin{equation*}
u(x^2,x^2)=x^2+(x^2)^2\,,\,\,u_1(x^2,x^2)=1\,,\,\,
u_2(x^2,x^2)=2x^2\, ,
\end{equation*}
whose corresponding parametric curve is
\begin{equation}\label{eq:CurvPsi.2}
\Psi(x^2)=\big(x^2,x^2,(x^2)^2+x^2,1,2x^2\big)\, ,
\end{equation}
is not characteristic for \eqref{eq:wave.2}: indeed, the function  $f(x^1,x^2)=x^1+(x^2)^2$ is the only solution to \eqref{eq:wave.2}, whose first jet   along the curve $x^2\mapsto (x^2,x^2)$ coincides with \eqref{eq:CurvPsi.2}.
\end{example}

\subsection{Spaces of $k$--jets  of functions in $n$ variables}\label{sec:k.jets}

Let $f:\Omega\subseteq\R^n\longrightarrow\mathbb{R}$ be a smooth function on a open domain $\Omega\subseteq \R^n$.
The \emph{$k$--jet }of $f=f(
\x)=f(x^1,\dots,x^n)$ at a point $\x_0=(x^1_0,\dots,x^n_0)\in\Omega$ is defined as the Taylor expansion of $f$ at $\x_0$ up to order $k$:
$$
j^k_{\x_0}f:=\sum_{h=0}^k\sum_{i_1\dots i_h}  \frac{1}{h!}\frac{\partial^h f}{\partial x^{i_1}\cdots\partial x^{i_h}}(\x_0) \, (x^{i_1}-x^{i_1}_0) \cdots (x^{i_h}-x^{i_h}_0) \,.
$$
The totality of such polynomials is denoted by
$$
J^k:=\{ j^k_{\x_0}f\mid f:\Omega \to\mathbb{R}\,,\,\, \x_0\in \Omega\}
$$
and it is called the \emph{space of $k$--jets of functions} on  $\R^n$. Note that $J^0=\Omega\times\R$, whereas \eqref{eq:set.Omega} is nothing but $J^2$ with $\Omega=\R^n$.
Since $j^k_{\x_0}f$ is unambiguously  defined by the coefficients of the above polynomial, we can unambiguously write
\begin{equation*}
j^k_{\x_0}f=\left( \x_0\,,\,f(\x_0)\,,\,\nabla f(\x_0)\,,\dots\,,\, \nabla^k f(\x_0) \right) = \left(\x_0\,,\,f(\x_0)\,,\,\frac{\partial f}{\partial x^i}(\x_0)\,,\dots\,, \frac{\partial^k f}{\partial x^{i_1}\cdots\partial x^{i_k}}(\x_0)\right)\, ,
\end{equation*}
where $1\leq i_1\leq i_2\leq\dots\leq i_k\leq n$, i.e.,
one can regard $j^k_{\x_0}f$ as the equivalence class $[f]^k_{\x_0}$ of functions having the same derivatives of $f$ at $\x_0$, up to order $k$.
The space $J^k$ admits a coordinate system
\begin{equation}\label{eq:jet.coordinates}
(x^i,u,u_i, u_{ij}, \ldots, u_{i_1 \cdots i_k} )\, ,\quad 1\leq i_1 \leq i_2\leq\dots\leq i_k\leq n\,,
\end{equation}
which can be thought of as an extension of a coordinate system
 $(x^1,\dots,x^n,u)$ on $\Omega\times\R$ in the following sense: each coordinate function\footnote{The $u_{i_1 \cdots i_k}$'s are symmetric in the lower indices.} $u_{i_1 \cdots i_h}$ on $J^k$, with $h\leq k$,  is unambiguously defined by
\begin{equation*}
u_{i_1 \cdots i_h}(j^k_{\x_0}f)= \frac{\partial^h f}{\partial x^{i_1}\cdots\partial x^{i_h}}(\x_0)\,.
\end{equation*}
To keep the notation light, from now on, a particular point $j^k_{\x_0}f\in J^k$ will be denoted by the symbol
$$
\fk=\left(x^i_0,u_0,u^0_i,u^0_{ij},\dots,u^0_{i_1\cdots i_k} \right)\, ,\quad 1\leq i_1 \leq i_2\leq\dots\leq i_k\leq n\,.
$$
The natural projections $\pi_{k,m}:J^k\to J^m$, $\fk\mapsto \theta_m$, $k>m$, define a tower of bundles
\begin{equation*}
\dots\longrightarrow\ J^k\longrightarrow J^{k-1}\longrightarrow\dots\longrightarrow J^1\longrightarrow J^0=\Omega\times \R\,.
\end{equation*}
We denote the fiber of $\pi_{k,k-1}$ over the point $\theta_{k-1}\in J^{k-1}$ by
\begin{equation}\label{eq:fiber}
J^k_{\theta_{k-1}}:=\pi^{-1}_{k,k-1}(\theta_{k-1})\, .
\end{equation}
The \emph{(truncated to order $k$) total derivative} operators are defined as follows:
\begin{equation}\label{eq:total.derivative}
D^{(k)}_i:=\partial_{x^i}+u_i\partial_u+u_{ij_1}\partial_{u_{j_1}} + \dots + \sum_{j_1\leq\cdots\leq j_{k-1}} u_{i j_1\cdots j_{k-1}}\partial_{u_{j_1\cdots j_{k-1}}}\,.
\end{equation}
\subsection{$2\Nd$ order PDEs via hypersurfaces in a Lagrangian Grassmanian}\label{sub2ordPDEhypSurf}

From now on, we will be considering only $2\Nd$ order PDEs, i.e., we set $k=2$ in the framework given in Section \ref{sec:k.jets}: the general machinery of jet spaces 
briefly sketched above gives way to the more specific, and yet equivalent, formalism based on  contact manifolds, see, e.g.,  \cite{MR2352610}. In view of such a choice, there will appear some terminology and gadgets typical of Exterior Differential Systems, such as \emph{integral elements} or \emph{Pfaffian systems}, see \cite{MR1083148, McKay2019}.\par

The \emph{integral element} associated with $\ff\in J^2$,  denoted by $L_{\ff}$, is, by definition, the $n$--dimensional vector subspace of $T_{\f}J^1$ spanned by the operators \eqref{eq:total.derivative} (with $k=2$) evaluated at
\begin{equation}\label{eq:theta.2.bello}
\ff=\big(x^i_0,u_0,u_i^0,u_{ij}^0\big)=\big(\f,u_{ij}^0\big)\,.
\end{equation}
More precisely, 
\begin{equation}\label{eq:TangSpaToJKMU}
L_{\ff}:=\Span{\left.D_i^{(2)}\right|_{\ff}}_{i=1,2,\ldots, n}=\Span{ \partial_{x^i}\big|_{\f}+u_i^0\partial_u\big|_{\f}+\sum_{i\leq j}u_{ij}^0\partial_{u_j}\big|_{\f} }_{i=1,2,\ldots, n}\,.
\end{equation}
The key remark of this section is that the $(2n+1)$--dimensional jet space $J^1$ is endowed with a natural \emph{contact structure}, i.e., a  (completely non--integrable) vector distribution of rank $2n$, which we denote by $\CC$ and   can locally be described as the kernel of the \emph{contact form}
\begin{equation*}
\theta = du -  u_i dx^i.
\end{equation*}
It turns out that
\begin{equation}\label{eq:piano.contatto}
    \CC =\Span{D^{(1)}_i\,, \partial_{u_i}}_{i=1,2,\ldots,n}\, .
\end{equation}
Note that $(d\theta)_{\f}=(dx^i)_{\f}\wedge (du_i)_{\f}$ is a symplectic form on $\CC_{\f}$ for any $\f\in J^1$. A (local) diffeomorphism  $g:J^1\longrightarrow J^1$ preserving the contact distribution is called a \emph{contactmorphism} and induces a \emph{symplectomorphism} $g_\ast:\CC_{\f}\longrightarrow\CC_{g(\f)}$ between the corresponding contact spaces. The contactomorphisms that leave the point $\f$ invariant constitute a group isomorphic to the conformal symplectic group $\CSp(2n)$. Any contactomorphism can be prolonged to $J^2$ by taking the corresponding tangent map.
\begin{example}\label{ex:Legendre}
Let $m\leq n$. The (partial) Legendre transformation
\begin{equation}\label{eq:legendre.partial}
(x^i,u,u_i)\to \left(u_1,\dots,u_m,x^{m+1},\dots,x^n,u-\sum_{i=1}^mx^iu_i,-x^1,\dots,-x^m,u_{m+1},\dots,u_n\right) 
=\left(\tilde{x}^i,\tilde{u},\tilde{u}_i\right)
\end{equation}
is a contactomorphism. If $m=n$, then we have a \emph{total} Legendre transformation.
\end{example}
A \emph{Lagrangian} subspace is an $n$--dimensional and isotropic (with respect to the symplectic form $d\theta$) subspace of the symplectic space $\CC_{\f}$. The set
\begin{equation}\label{eq:Lagr.Grass}
\LL(n,\CC_{\f}):=\{\text{Lagrangian subspaces of } \CC_{\f}\}
\end{equation}
of all such subspaces is the \emph{Lagrangian Grassmannian variety} of $\CC_{\f}$; since all the Lagrangian Grassmannian varieties of $2n$--dimensional symplectic spaces are isomorphic, one can use the collective symbol $\LL(n,2n)$, when there is no need to stress the base point $\f$. It is worth stressing that in $\LL(n,\CC_{\f})$ there are, in particular, the integral elements $L_{\ff}$ corresponding to the points $\ff$ of the fiber $J^2_{\f}$, cf. \eqref{eq:fiber}: these integral elements do not, however, fill out the whole of \eqref{eq:Lagr.Grass}, but just an open subset of it.\par

A $2\Nd$ order PDE $\E$ can be then seen as a hypersurface of $J^2$ 
whose local expression, in a system of coordinates \eqref{eq:jet.coordinates}, is  \eqref{eq:equation}: put differently, a $2\Nd$ order PDE $\E$ is a sub--bundle of $J^2$ over $\pi_{2,1}(\E)$, where the latter subset of $J^1$ can be always assumed, save for a few exceptional cases, to be coinciding with $J^1$ itself. 
It is called \emph{symplectic} (or \emph{dispersionless Hirota type}, see \cite{Ferapontov2020}) if the function $F$ of \eqref{eq:equation} does not depend neither on coordinates $x^i$, nor on  $u$ and its derivatives $u_i$. In other words, a symplectic PDE $\E$ has the structure of a product of the $(x^i,u,u_i)$--space, by a fiber 
\begin{equation}\label{eq:fibre.PDE}
\E_{\f}:=\E\cap J^2_{\f}\,.
\end{equation} 
Accordingly,  the equivalence problem for symplectic PDEs becomes a problem of (conformal) symplectic equivalence.
In fact, the study of symplectic PDEs up to contactomorphisms coincides with the study of codimension--one subsets of the fiber $J^2_{\f}$ up to $\CSp(2n)$, 
which in turn boils down to studying codimension--one subsets of $\LL(n,2n)$, see \cite{Gutt2019} for  more details.\par

\subsection{3D symplectic Monge--Amp\`ere equations as hyperplane sections of $\LL(3,6)$}\label{sub3DSMAE}
The main concern of this paper are 3D Monge--Amp\`ere equations, i.e., Monge--Amp\`ere equations  with $3$ independent variables: we set then $n=3$. In the framework we have just outlined, a general 3D Monge--Amp\`ere equation, regarded as a hypersurface of $J^2$, can be given in terms of integral elements $L_{\ff}$ as follows
\begin{equation}\label{eq:gen.MAE}
\mathcal{E}:=\{\ff\in J^2\,\,|\,\,\Phi|_{L_{\ff}}=0\}\,,
\end{equation}
where 
\begin{equation}\label{eq:Phi.per.MAE}
\Phi=\Phi_{ijk}dy^i\wedge dy^j\wedge dy^k\,,\,\,(y^1,\dots,y^6)=(x^1,x^2,x^3,u_1,u_2,u_3)\,,\,\,\Phi_{ijk}\in C^\infty(J^1)\,,
\end{equation}
is a 3--differential form.
By using the system of coordinates \eqref{eq:jet.coordinates}, it suffices to substitute 
\begin{equation}\label{eq:substitution.hor}
u_i\to u_{ij}dx^j 
\end{equation} 
in \eqref{eq:Phi.per.MAE} in order to obtain a local coordinate description of the equation  \eqref{eq:gen.MAE}: indeed, substitution \eqref{eq:substitution.hor} gives us a multiple of the ``volume form'' $dx^1\wedge dx^2\wedge dx^3$, whose coefficient,  equated to zero, locally represents $\E$.
\begin{example}\label{ex:oriz.and.equival}
Replacement \eqref{eq:substitution.hor} above, performed on   the differential 3--form  
$$
\Phi=du_1\wedge du_2\wedge du_3-kdx^1\wedge dx^2\wedge dx^3\,,\quad k\in\mathbb{R}\,,
$$ 
yields 
$\big(\det\|u_{ij}\|-k\big)dx^1\wedge dx^2\wedge dx^3$, whose coefficient, equated to zero, is  the Monge--Amp\`ere equation 
$$
\det\|u_{ij}\|=k\,.
$$ 
After a total Legendre transform \eqref{eq:legendre.partial} ($m=n=3$),  
$\Phi$ reads $\widetilde{\Phi}=kd\widetilde{u}_1\wedge d\widetilde{u}_2\wedge d\widetilde{u}_3 + d\widetilde{x}^1\wedge d\widetilde{x}^2\wedge d\widetilde{x}^3$, whose  associated Monge--Amp\`ere equation    is $k\det\|\widetilde{u}_{ij}\|+1=0$. Had we  considered the transformation \eqref{eq:legendre.partial} with $m=1$ and $n=3$, then $\Phi$ would have become $ \widetilde{\Phi}=kd\widetilde{u}_1\wedge d\widetilde{x}^2\wedge d\widetilde{x}^3 + d\widetilde{x}^1\wedge d\widetilde{u}_2 \wedge d\widetilde{u}_3$, whose associated   Monge--Amp\`ere equation     would be   
$k\widetilde{u}_{11}+\widetilde{u}^\sharp_{11}=k\widetilde{u}_{11}+\widetilde{u}_{22}\widetilde{u}_{33}-\widetilde{u}_{23}^2=0$.
\end{example}
It turns out that a Monge--Amp\`ere equation 
is described by \eqref{eq:equation},  where $F$ is  a linear combination of the minors of the Hessian matrix $\|u_{ij}\|$ with coefficients in $C^\infty(J^1)$.

\begin{definition}\label{defMAEq}
An equation \eqref{eq:equation}, where $F$ is a linear combination of the minors of the Hessian matrix $\|u_{ij}\|$ with coefficients in $C^\infty(J^1)$, is called a  \emph{Monge--Amp\`ere equation}; if the coefficients are  constant, then it is called \emph{symplectic}. 
\end{definition}
A general Monge--Amp\`ere equation $\mathcal{E}$ with 3 independent variables, in view of Definition \ref{defMAEq}, can be then written down as
\begin{equation}\label{eq.MAE}
\mathcal{E}:=\,\left\{\,A\det\|u_{ij}\| + \sum_{i\leq j} B_{ij}u^\sharp_{ij} + \sum_{i\leq j}C^{ij}u_{ij} + D=0\,\right\}\,,
\end{equation}
where we recall that $\|u^\sharp_{ij}\|$ is the cofactor matrix of $\|u_{ij}\|$ and $A,B_{ij},C^{ij},D\in C^\infty(J^1)$. Thus, 3D symplectic Monge--Amp\`ere equations are subsets \eqref{eq.MAE} with $A,B_{ij},C^{ij},D\in \mathbb{R}$.
Note that, as codimension--one subsets of the Lagrangian Grassmanian $\LL(3,6)$, cf. \eqref{eq:Lagr.Grass}, symplectic Monge--Amp\`ere equations are  hypersurfaces  of the simplest kind, that is, \emph{hyperplane sections} of $\LL(3,6)$; in the last Section \ref{secLinSec} we show that, at least over the field of complex number, there are only four possible geometries for  such hyperplane sections.\par

In order to obtain a faithful parametrization, it is convenient to consider a special subclass of differential $3$--forms \eqref{eq:Phi.per.MAE}, namely the linear subspace consisting of those forms $\Phi$,  such that $\omega^{ij}\Phi_{ijk}=0$, where $\omega=dx^i\wedge du_i$ locally represents (the conformal class of) the natural  symplectic form on $\CC$. Such forms, introduced in \cite{MR2352610}, are called \emph{effective} and they are in one--to--one correspondence with 3D  Monge--Amp\`ere equations,  up to a nowhere zero factor: this is the reason why,  in Section \ref{secMAESpace}, we will be considering  the projectivization of the space of effective $3$--forms (at a point $\f$); not only it leads to a strict one--to--one parametrization (of \emph{symplectic} Monge--Amp\`ere equations), but it also allows to work with the symplectic group $\Sp(\CC_{\f})=\Sp(6,\R)$, rather than its conformal counterpart. From Section \ref{secMAESpace} on,  we will be working on $\C$ and then  group $\Sp(\CC_{\f})=\Sp(6,\C)$ will be a simple one, thus making it possible to deploy the whole machinery of representation theory.\par
In the last Section \ref{secLinSec} the reader will find another definition of a 3D symplectic Monge--Amp\`ere equation, see Definition \ref{defMAEComplex}: it is formally analogous to Definition \ref{defMAEq} above, only over the field of complex numbers.

\subsection{The Kushner--Lychagin--Rubtsov (KLR) invariant and the cocharacteristic variety}\label{secKLRform}
In \cite{MR2352610} the authors have showed that with any    Monge-Amp\`ere equation \eqref{eq.MAE} (which has been defined by means of the 3--form \eqref{eq:Phi.per.MAE}) one can associate a quadratic form via
\begin{equation}\label{eq.JGG}
\omega^{i_1j_1}\omega^{i_2j_2}\Phi_{a\, i_1 i_2}\Phi_{b\, j_1 j_2}d y^a d y^b\, ,
\end{equation}
where 
\begin{multline*}
\Phi_{456}=A\,,\,\,
\Phi_{156}=B_{11}\,,\,\,
\Phi_{146}=-\Phi_{256}=-B_{12}\,,\,\,
\Phi_{145}=\Phi_{356}=B_{13} \,,\,\,
\Phi_{246}=-B_{22} \,,\,\,
\Phi_{346}=-\Phi_{245}=-B_{23}  \,,\,\,
\\
\Phi_{345}=B_{33} \,,\,\,
\Phi_{234}=C^{11} \,,\,\,
\Phi_{235}=-\Phi_{134}=C^{12}\,,\,\,
\Phi_{124}=\Phi_{236}=C^{13} \,,\,\,
\Phi_{135}=-C^{22} \,,\,\,
\Phi_{125}= -\Phi_{136}=C^{23}\,,
\\
\Phi_{126}=C^{33} \,,\,\,
\Phi_{123}=D \,.
\end{multline*}
and $(y^1,\dots,y^6)$ are the same as in \eqref{eq:Phi.per.MAE}.
\begin{remark}
The quadratic form \eqref{eq.JGG} above can be given without employing any coordinates, see  \eqref{eq:mappa.KLR}; from now on, both of them will be referred to as the \emph{KLR invariant} of the Monge-Amp\`ere equation \eqref{eq.MAE}.
\end{remark}
Theorem \ref{thFirstOriginalResult} below shows that the KLR invariant   \eqref{eq.JGG} is equivalent to the Hitchin moment map  on the space parametrizing symplectic Monge--Amp\`ere equations, and it will be thoroughly reviewed and deepened  in Section \ref{subsMomMap}; moreover,  
Corollary \ref{corEquivalencyMomentCoChar} shows that, if the invariant \eqref{eq.JGG} is equated to zero, one obtains a quadratic hypersurface in $\CC$, which is naturally linked with the characteristics of the Monge--Amp\`ere equation at hand via \emph{projective duality}. 
Such a duality is the rationale behind  the choice of the prefix ``co" in the next definition.

\begin{definition}\label{def.cochar.var}
The zero locus of the homogeneous $2^{\Nd}$ order polynomial \eqref{eq.JGG} is called the \emph{cocharacteristic variety} of the 3D Monge--Amp\`ere equation \eqref{eq.MAE}.
\end{definition}

We stress again that later on in Section \ref{secLinSec}, above Definition \ref{def.cochar.var} will be reformulated in the complex setting for \emph{symplectic} 3D Monge--Amp\`ere equations, see Definition \ref{defCoCarComp}: to avoid uninteresting complications, the proofs of the main results of this paper will be given over the field of complex numbers.\par


 
%
%

\subsection{
Monge--Amp\`ere equations of Goursat--type and a non--linear generalization of his idea: contact cone structures}\label{sec:Goursat}
%

At the very beginning we mentioned Goursat pioneering paper \cite{MR1504329}: there, he proposed  
the following way of obtaining  Monge-Amp\`ere equations:  substituting  $du_{1}=u_{11}dx^{1}+u_{12}dx^{2}$ and $du_{2}=u_{12}dx^{1}+u_{22}dx^{2}$ into the  Pfaffian system
\[
\left\{
\begin{array}{ll}
du_{1}-b_{11}dx^{1}-b_{12}dx^{2}=0\, , &
\\
& \qquad b_{ij}=b_{ij}(x^{1},x^{2},u,u_{1},u_{2})\, ,
\\
du_{2}-b_{21}dx^{1}-b_{22}dx^{2}=0\, , &
\end{array}
\right.
\]
and then requiring  its (non trivial) compatibility. Such a procedure was then generalized by Goursat  to any number $n$ of independent variables, by considering 
the system
\begin{equation}\label{eq:pfaffian.Goursat.general}
\alpha_i:=du_{i}-\sum_{j=1}^{n}b_{ij}dx^{j}=0\,,\,\,\,i=1,\dots,n\,,\,\,\,b_{ij}%
=b_{ij}(x^{1},\dots,x^{n},u,u_{1},\dots,u_{n})\,,
\end{equation}
thus getting the equation
\begin{equation}\label{eq.MAEs.Goursat.type}
\det\|u_{ij}-b_{ij}\|=0\,.
\end{equation}
\begin{definition}\label{def:Goursat.MAE}
An equation \eqref{eq:equation}, where $F$ is given by the left--hand side of  \eqref{eq.MAEs.Goursat.type}, is called a  Monge--Amp\`ere equation of \emph{Goursat type}.
\end{definition} 
The same equation is obtained if a replacement $b_{ij}\to b_{ji}$ is performed in \eqref{eq:pfaffian.Goursat.general}, i.e., if one starts from the Pfaffian system
\begin{equation*}
\widetilde{\alpha}_i:=du_{i}-\sum_{j=1}^{n}b_{ji}dx^{j}=0\,,\,\,\,i=1,\dots,n\, .
\end{equation*}
Equation \eqref{eq.MAEs.Goursat.type} belongs to a special class of Monge-Amp\`ere equations, i.e., those whose symbol has rank less or equal to 2; see \cite{MR2985508} for more details.\par

An equivalent approach, leading to the same equation \eqref{eq.MAEs.Goursat.type},  goes as follows.\par
Let us consider the jet space $J^1$ with coordinates $(x^i,u,u_i)$ and, within its canonical contact distribution $\CC$, let us single out two  $n$--dimensional sub--distributions:
\begin{equation}\label{eq:D}
\mathcal{D}:=\langle \partial_{x^i} + u_i\partial_u+b_{ij}\partial_{u_j}\rangle_{i=1,\ldots n}=\{\alpha_i=0\, , \theta=0 \}
\,,\quad \mathcal{D}^\perp=\langle \partial_{x^i} + u_i\partial_u+b_{ji}\partial_{u_j}\rangle_{i=1,\ldots n}=\{\widetilde{\alpha}_i=0\, , \theta=0 \}\, .
\end{equation}
Then, the equation \eqref{eq.MAEs.Goursat.type} is obtained by requiring that a general integral element  $\langle \partial_{x^i} + u_i\partial_u+u_{ij}\partial_{u_j}\rangle_{i=1,\ldots n}$ of $\CC$ nontrivially intersects $\mathcal{D}\cup\mathcal{D}^\perp$. In the above cited paper \cite{MR2985508} it is proved that any hyperplane containing a line of $\mathcal{D}\cup\mathcal{D}^\perp$ is a characteristic hyperplane for the equation \eqref{eq.MAEs.Goursat.type}, and vice versa.
%
%

\smallskip
Our idea for generalizing Goursat approach stems from a simple observation: the union $\V:=\mathcal{D}\cup\mathcal{D}^\perp$ is a  distribution $\f\in J^1\to\V_{\f}\subset\CC_{\f}$ of (very degenerate) $n$--codimensional quadric affine conic varieties; therefore, if from $\V$ one gets a Monge--Amp\`ere equation whose characteristic hyperplanes lie in $\V$, which PDEs would one obtain,  starting from a more general quadric (for instance, non--degenerate and not necessarily of codimension $n$)  in $\CC$?
\par
This question is   the paper's backbone and it leads, in particular, to the problem of classifying all quadric hypersurfaces in a 6--dimensional symplectic space, up to symplectomorphisms, which will be done in Section \ref{SecNormFormQadr}.\par
In Sections \ref{secForwardRecepi}, \ref{secSezioneNuova} and \ref{secBackRecepi} below we will be rather focusing  on some practical examples of how to pass from a 3D  $2\Nd$ order PDE to a (quadratic) contact cone structure, accordingly to the following definition.
 \begin{definition}\label{defCCS}
A \emph{contact cone structure} $\V$ on $J^1$ is a smooth assignment of an affine conic variety $\V_{\f}\subset\CC_{\f}$ or, equivalently, of a projective variety $\p\V_{\f}\subset\p\CC_{\f}$, in each point $\f\in J^1$.
\end{definition}
The cocharacteristic variety associated with a Monge--Amp\`ere equation (see Definition \ref{def.cochar.var}) is our first example of a \emph{quadratic} generalization of a \emph{linear} sub--distribution of a contact distribution, henceforth called a \emph{quadric contact cone structure}. It turns out that there is more than one way to associate a quadric contact cone structure with a $2\Nd$ order PDE, see Section \ref{secForwardRecepi} below: nevertheless, at least for 3D Monge--Amp\`ere equations with \emph{non--degenerate symbol}, the cocharacteristic variety is the only natural one (this will follows from Corollary \ref{corEquivalencyMomentCoChar} later on).

\section{The contact cone structure associated with a $2\Nd$ order PDE}\label{secForwardRecepi}
In order to associate a contact cone structure with a $2\Nd$ order PDE and, in particular, with a Monge--Amp\`ere equation, we need to pass through the notion of a characteristic which is, in turn, related to that of a rank--one vector. To this purpose, we need the following definition.
\begin{definition}\label{def:prolong}
Any $\kappa$--dimensional subspace $H_{\f}$ of the contact space $\CC_{\f}$, with $\kappa\leq n$, can be prolonged to a submanifold $H^{(1)}_{\f}$ of $J^2_{\f}$ defined by
$$
H^{(1)}_{\f}:=\{\ff\in J^2_{\f}\,\,\mid\,\, L_{\ff}\supseteq H_{\f}\}\,.
$$
\end{definition}
\begin{remark}\label{rem:prol.hyper}
It is easy to see that if $H_{\f}$ is a hyperplane of $L_{\ff}$, then $H^{(1)}_{\f}$ is a curve in $J^2_{\f}$ passing through  $\ff$.
\end{remark}

\begin{remark}\label{rem.regular.points}
As we said in Section \ref{sec:starting}, we work with regular points of $2\Nd$ order PDEs, so that any point $\ff\in \mathcal{E}$ of a $2\Nd$ order PDE $\mathcal{E}$ is assumed regular.
For the sake of simplicity, and to not overload the notation, the set of regular points $\mathcal{E}_{\mathrm{reg}}$ of a $2\Nd$ order PDE $\mathcal{E}$ will be denoted by the same symbol we use for the PDE, i.e., by $\mathcal{E}$. For the same reason, the projection $\check{J}^1:=\pi_{2,1}(\mathcal{E}_{\mathrm{reg}})\subseteq J^1$ of $\mathcal{E}_{\mathrm{reg}}$ onto $J^1$ will be denoted by $J^1$.
\end{remark}
From the point of view of contact geometry, the notions that have been introduced in Section \ref{secPRELIM} above can be recast, more geometrically, as follows.
\begin{definition}\label{def:char.hyper}
A \emph{Cauchy datum} for a $2\Nd$ order PDE $\E\subset J^2$ is an $(n-1)$--dimensional integral submanifold $\Sigma$ of the contact
distribution $\CC$ on $J^1$. It is \emph{characteristic} at a point $\ff\in\mathcal{E}$
if the prolongation $(T_{\f}\Sigma)^{(1)}$ is tangent to  $\E_{\f}$ at the point $\ff$.
In this case, $T_{\f}\Sigma$ is also called a \emph{characteristic} hyperplane (at $\f=\pi_{2,1}(\ff)$).
\end{definition}

\subsection{Rank of vertical vectors and characteristics}\label{sec:rank}

\begin{definition}
The \emph{rank} of a vector
\begin{equation}\label{eq:vertical.vector}
\nu_{\ff}=\sum_{i\leq j} \nu_{ij}\partial_{u_{ij}}\big|_{\ff}\,,\quad \nu_{ij}\in\R\, ,
\end{equation}
of $T_{\ff}J^2_{\f}$ is the rank of the matrix
\begin{equation}\label{eq:matrix.nu}
\begin{pmatrix}
\nu_{11}&\dots&\nu_{1n}
\\
\vdots & \ddots&\vdots
\\
\nu_{1n}&\dots&\nu_{nn}
\end{pmatrix}\,.
\end{equation}
The rank of a line $\ell_{\ff}\subset T_{\ff}J^2_{\f}$ is the rank of any vector $v_{\ff}$,  such that $\ell_{\ff}=\langle v_{\ff}\rangle$.
\end{definition}
A direct computation shows  that  the rank of a line $\ell_{\ff}$ in  $T_{\f}J^2_{\f}$ is invariant with respect to  contactomorphisms.

\smallskip\noindent
Below we shall clarify the nature of such an invariant. Let us fix a point $\ff\in J^2$ and a vector \eqref{eq:vertical.vector}.
Take a curve $\ff(t)$ in $J^2_{\f}$ such that $\ff(0)=\ff$ and $\ff'(0)=\nu_{\ff}$: it will be given locally by
\begin{equation}\label{eq:fft}
\ff(t)=\big(\f,u_{ij}(t)\big)\,, 
\end{equation}
with $u'_{ij}(0)=\nu_{ij}$ and it will correspond to a $1$--parametric family of integral elements (cf. \eqref{eq:TangSpaToJKMU}):
$$
L_{\ff(t)}
=
\Span{
\partial_{x^i}\big|_{\f}+u_i^0\partial_u\big|_{\f} + u_{ij}(t)\partial_{u_{j}}\big|_{\f}
}_{i=1,\dots,n}=\Span{D^{(1)}_i\big|_{\f} + u_{ij}(t)\partial_{u_{j}}\big|_{\f}}_{i=1,\dots,n}\, .
$$
Since  $L_{\ff(t)}$ and  $L_{\ff}$  are two
$n$--dimensional subspaces of the $2n$--dimensional vector space $\mathcal{C}_{\theta_{1}}$, their intersection  $L_{\ff(t)}\cap L_{\ff}$ is, generically, zero--dimensional. However, there are curves \eqref{eq:fft} for which $L_{\ff(t)}\cap L_{\ff}$ is $\kappa$--codimensional $\forall\,t$, i.e., $L_{\ff(t)}$ ``rotates'' around an $(n-\kappa)$--dimensional subspace of $L_{\ff}$. From an infinitesimal viewpoint, this means that there are tangent directions \eqref{eq:vertical.vector}
along which $L_{\ff(t)}$, with $\ff(t)$ given by \eqref{eq:fft}, moves away  from $L_{\ff}$ by retaining some ``common piece''. This motivates the following definition.
\begin{definition}\label{def:deviation}
We say that $L_{\ff(t)}$ \emph{has a deviation of order $\kappa$ from $L_{\ff}$} if $\dim\big(L_{\widetilde{\ff}(t)}\cap L_{\ff}\big)=n-\kappa$, for small nonzero $t$ and $\widetilde{\ff}(t)$ is the linear approximation of $\ff(t)$ at  $\ff$. In other words, if
$$
\dim\big(L_{\ff+\ff'(0)t}\cap L_{\ff}\big)=n-\kappa\,,\,\,\forall\,t\in (-\epsilon,\epsilon)\smallsetminus \{ 0\}\,.
$$
\end{definition}
\begin{proposition}\label{prop:CorrVecDefInfinNonGen}
The vector \eqref{eq:vertical.vector} has rank $\kappa$ if and only if $L_{\ff(t)}$, with $\ff(t)$ given by \eqref{eq:fft}, has a deviation of order $\kappa$ from $L_{\ff}$.
\end{proposition}
\begin{proof}
Without any loss of  generality, we can assume  that   $\ff$ has  all $2\Nd$ order jet  coordinates  equal to 0, i.e., that  $u_{ij}(\ff)=0$ $\forall i,j$. Then
$$
L_{\ff+\ff'(0)t}=\Span{D^{(1)}_i\big|_{\f} + u_{ij}'(0)t\,\partial_{u_{j}}\big|_{\f}}_{i=1,\ldots,n}
$$
and the intersection
\begin{equation*}
L_{\ff+\ff'(0)t}\cap L_{\ff}
\end{equation*}
has dimension $n-\kappa$ if and only if the block matrix
\begin{equation*}
\left(
\begin{array}{c|c}
\text{Id} & u_{ij}'(0)t \\
\hline
\text{Id} & 0
\end{array}
\right)\,,
\end{equation*}
has rank $n+\kappa$, which in turn is equivalent to the upper--right block having rank $\kappa$: but, in view of  \eqref{eq:fft}, this is the same as the rank of the matrix \eqref{eq:matrix.nu}.
\end{proof}
Putting together Definition \ref{def:deviation} and Proposition \ref{prop:CorrVecDefInfinNonGen} we obtain that to each line  $\ell_{\ff}\subset T_{\ff}J^2_{\f}$ of rank $\kappa$ we can associate an $(n-\kappa)$--dimensional subspace $H(\ell_{\ff})$ of the integral element $L_{\ff}$:
\begin{equation}\label{eq:correspondence.general}
\text{tangent lines of rank $\kappa$   in $T_{\ff}J^2_{\f}$ $\Rightarrow$ $(n-\kappa)$--dimensional subspaces of $L_{\ff}$}\,.
\end{equation}
The implication \eqref{eq:correspondence.general} is one--to--one only in the case of lines  of rank $1$. Indeed, the dimension of the prolongation $H^{(1)}$ of a subspace $H\subseteq L_{\ff}$ of codimension $\kappa\geq 2$ is greater than $1$, whereas it is equal to $1$ in the case $\kappa=1$ (see Remark \ref{rem:prol.hyper}).\par

We now analyse in more detail the case $\kappa=1$, i.e., rank--one lines, since they  are closely related with characteristic Cauchy data. In this case \eqref{eq:correspondence.general} becomes
\begin{equation}\label{eq:correspondence}
\text{tangent lines $\ell_{\ff}=\Span{\nu_{\theta_2}}\subset T_{\ff}J^2_{\f}$ of rank $1$  $\Leftrightarrow$ hyperplanes $H(\ell_{\ff})\subset L_{\ff}$}\,.
\end{equation}
Taking into account Definition \ref{def:char.hyper} and correspondence \eqref{eq:correspondence}, it is easy to realize that with  any rank--one line  $\ell_{\ff}$   is associated the hyperplane $H(\ell_{\ff})$ of $L_{\ff}$ and that $H(\ell_{\ff})$ is characteristic if the direction $\ell_{\ff}$ is included in the tangent space of the considered PDE.
\begin{definition}\label{def.char}
A line  $\ell_{\ff}\subset T_{\ff}J^2_{\f}$ of rank $1$ is \emph{characteristic} for a PDE \eqref{eq:equation} in $\ff\in\E$ if $\ell_{\ff}\subset T_{\ff}\E$.
\end{definition}
It is well known that matrices of rank $1$ have the $(i,j)$--entry equal to $\eta_i\eta_j$, where $\eta=(\eta_1,\eta_2,\ldots,\eta_n)$ is a vector in $\R^n$. Therefore, a $\rank$--one vector \eqref{eq:vertical.vector} has the form
\begin{equation}\label{eq:local.rank.1} 
\nu_{\ff}=\sum_{i\leq j}\eta_i\eta_j\partial_{u_{ij}}\big|_{\ff}\,,\quad\eta_i\in\R\,.
\end{equation}
Then the  hyperplane  $H(\ell_{\ff})$ corresponding to  $L_{\ff}$ via \eqref{eq:correspondence} is locally described by
\begin{equation}\label{eq:H}
H(\ell_{\ff})=\left\{\xi^iD^{(2)}_i\big|_{\ff} \mid \xi^i \eta_i=0\,, \quad \xi^i\in\R \right\} 
\end{equation}
or equivalently, as
\begin{equation}\label{eq:AccaComeNucleo}
H(\ell_{\ff})=\ker\eta\, ,\quad \eta=\eta_idx^i\in L^*_{\ff}\, .
\end{equation}
The claim \eqref{eq:AccaComeNucleo} is a direct consequence of an elementary property of symmetric rank--one $n\times n$ matrices:
\begin{equation*}
\ker (\eta_i\eta_j)=\Span{(\xi^1,\ldots, \xi^n)\in\R^n\mid \xi^i \eta_i=0}\, .
\end{equation*}
The converse reads as follows: given 
a hyperplane $H\subset L_{\theta_2}$, there is a covector $\eta\in T^*_{\x_0}\R^n$ (defined up to a nonzero factor) by an analogous formula to \eqref{eq:AccaComeNucleo}; the vector $\nu_{\theta_2}$, such that $H(\ell_{\ff})=H$ is then constructed by means of formula \eqref{eq:local.rank.1}, by using the components of $\eta$.
If we  look at Definition \ref{def.char} in local coordinates \eqref{eq:jet.coordinates}, we see that a line $\ell_{\ff}$ is of rank--one if it is spanned by a vector of type \eqref{eq:local.rank.1}. Therefore, it is characteristic for PDE \eqref{eq:equation} at $\ff\in\E$ if
\begin{equation}\label{eq:char.general.local}
\sum_{i\leq j}\left.\frac{\partial F}{\partial u_{ij}}\right|_{\ff}\,\eta_i\eta_j=0\,,
\end{equation}
which  coincides with the equation of characteristics (see, for instance, \cite{Petrovsky1992}). In other words, the covector $\eta_idx^i$ annihilates the symbol of the equation (cf. \eqref{eq:symbol.intro}).
\begin{example}\label{ex:u12.characteristics}
In the case $n=2$, i.e., with $2$ independent variables, in view of correspondence \eqref{eq:correspondence}, to the line $\ell_{\ff}$ it corresponds a line $H(\ell_{\ff})$ in $L_{\ff}$ via   \eqref{eq:correspondence}.
On account of \eqref{eq:H}, \eqref{eq:char.general.local} reads as
\begin{equation}\label{eq:char.class}
\left.\frac{\partial F}{\partial u_{11}}\right|_{\ff} ({\xi^1})^2 - \left.\frac{\partial F}{\partial u_{12}}\right|_{\ff}{\xi^1}{\xi^2} + \left.\frac{\partial F}{\partial u_{22}}\right|_{\ff} ({\xi^2})^2=0\,.
\end{equation}
For instance, if we consider the equation $\E:=\{u_{12}=0\}$ from  Example \ref{ex:introduction}, the characteristic equation \eqref{eq:char.general.local} reads
$
\eta_1\eta_2=0
$
or, in the form \eqref{eq:char.class},
$
\xi^1\xi^2=0\,.
$
So, characteristic lines  $\ell_{\ff}$, where $\ff=(x^1_0,x^2_0,u^0,u^0_1,u^0_2,u^0_{11},0,u^0_{22})\in\E$, are spanned by \eqref{eq:local.rank.1} where either $\eta_1$ or $\eta_2$ is equal to zero, i.e., lines 
$$
\ell_{\ff}^+=\Span{\partial_{u_{11}}\big|_{\ff}}\quad\text{or}\quad \ell_{\ff}^-=\Span{\partial_{u_{22}}\big|_{\ff}}\, , 
$$
and the corresponding characteristic hyperplanes (that,  in this case,  are lines of the integral element $L_{\ff}$) are
$$
H(\ell_{\ff}^+)=\Span{D^{(2)}_2\big|_{\ff}}=\Span{\partial_{x^2}\big|_{\ff}+u^0_2\partial_{u}\big|_{\ff} +u^0_{22}\partial_{u_2}\big|_{\ff} }\,,\quad
H(\ell_{\ff}^-)=\Span{D^{(2)}_1\big|_{\ff}} =\Span{\partial_{x^1}\big|_{\ff}+u^0_1\partial_{u}\big|_{\ff} +u^0_{11}\partial_{u_1}\big|_{\ff} } \,.
$$
In fact, in Example \ref{ex:introduction}, according to the definitions given in the present section, the curve $\Sigma=\{\Phi(x^1)\mid x^1\in\R\}$ is a Cauchy datum as it is a $1$--dimensional integral submanifold of the contact distribution $\CC$ (cf. Definition \ref{def:char.hyper}). Its prolongation  (cf. Definition \ref{rem:prol.hyper}) is  $\Sigma^{(1)}_{\f}=\{(x^1_0,0,x^1_0,1,0,1,0,t)\mid t\in\mathbb{R}\}$,  and $T_{\ff}\Sigma^{(1)}$ is a characteristic hyperplane, since  $T_{\ff}\Sigma^{(1)}=\ell_{\ff}^-\subset T_{\ff}\mathcal{E}_{\f}$.
%
\end{example}

\subsection{Contact cone structure associated with a $2\Nd$ order PDE in 3 independent variables of non--elliptic type}\label{sec.andata}

From now on, unless specified otherwise, we will be considering  only PDEs in $3$ independent variables, i.e.,
\begin{equation}\label{eq:equation.3.indip.var}
\E:=\,\big\{\,F(x^1,x^2,x^3,u,u_1,u_2,u_3,u_{11},\dots,u_{33})=0\,\big\}
\end{equation}
that are non--elliptic.
Below we will  see how to construct, starting from a PDE \eqref{eq:equation.3.indip.var}, a contact cone structure  $\V: \f\in J^1\to \V_{\f}\subset \CC_{\f}$, see Definition \ref{defCCS}.\par 
Such a construction breaks down into 4 steps.
%
\begin{enumerate}
\item Fix a point $\ff=\big(x^i_0,u_0,u_i^0,u_{ij}^0\big)\in\E$, i.e., such that
$
F\big(x^i_0,u_0,u_i^0,u_{ij}^0\big)=0\,.
$
\item Consider the set of rank--one lines at $\ff$ that are  tangent to $\E$. Such a set is given by vectors of type \eqref{eq:local.rank.1} such that $(\eta_1,\dots,\eta_n)$ satisfies  \eqref{eq:char.general.local}. Note that such vectors are organized either in two distinct families $\nu^+$ and $\nu^-$,  if the  PDE \eqref{eq:equation.3.indip.var} is hyperbolic at  $\ff$, or in one single family $\nu^+=\nu^-=\nu$ if it is parabolic at  $\ff$
In both cases, such families are $2$-parametric. This implies that the corresponding families of rank--one \emph{lines} are $1$-parametric. We denote such lines by $\ell^+_{\ff}(t)$ and $\ell^-_{\ff}(t)$, where $t$ is the aforementioned parameter.
\item Let us consider the line $\ell^+_{\ff}(t)$ only, since  the same reasoning works for  $\ell^-_{\ff}(t)$ as well. To each line  $\ell^+_{\ff}(t)$ we   associate  the hyperplane $H(\ell^+_{\ff}(t))\subset L_{\ff}$, see \eqref{eq:H}. Then, by varying the parameter $t$, the hyperplanes $H(\ell^+_{\ff}(t))$ sweep a cone of $L_{\ff}$, that we denote by $\V_{\ff}$, in the following sense: the generatrix $v(\ell^+_{\ff}(t))$ of $\V_{\ff}$, which is a line passing through  $\f=\big(x^i_0,u_0,u_i^0\big)$, will be given as an infinitesimal intersection
    $$
    v(\ell^{+}_{\ff}(t)):=\lim_{\epsilon\to 0}\,H(\ell^+_{\ff}(t))\cap H(\ell^+_{\ff}(t+\epsilon))\,.
    $$
Summing up, to any point $\ff\in\E$ we can associate two  cones in $L_{\ff}$:
    $$
    \V^+_{\ff}:=\bigcup_t\, v(\ell^+_{\ff}(t))\,,\quad \V^-_{\ff}:=\bigcup_t\, v(\ell^-_{\ff}(t))
    $$
\item If now we let vary the point $\ff\in\E$ over $\f$ we obtain a \emph{conic  variety} $\V_{\f}\subseteq \CC_{\f}$: 
    \begin{equation}\label{eq:CCC}
    \V_{\f}:=\V_{\f}^+\cup \V_{\f}^-=\bigcup_{\ff\in \E_{\f}}\V_{\ff}\, ,\quad  \V_{\f}^\pm:=\bigcup_{\ff\in \E_{\f}}\V^\pm_{\ff}\, ,\quad \V_{\ff}:=\V^+_{\ff}\cup\V^-_{\ff}\, .
    \end{equation}
\end{enumerate}
%
\begin{definition}
The conic sub--distribution $\V:\f\in J^1\to\V_{\f}\subset\CC_{\f}$, with  $\V_{\f}$  given by \eqref{eq:CCC}, is called the \emph{contact cone structure} of the  PDE \eqref{eq:equation.3.indip.var}.
\end{definition}

\section{Quadric contact cone structures associated with 3D Monge--Amp\`ere equations}\label{secSezioneNuova}

In this section we work with  3D Monge--Amp\`ere equations, i.e., Monge--Amp\`ere equations in $3$ independent variables.

As we will see in Section \ref{subsVarpiAndOrbits}, up to symplectic equivalence and signature, there are only four types of symplectic 3D Monge--Amp\`ere equations: below we pick a representative for each type and compute its contact cone structure, that will turn out to be a \emph{quadric} contact cone structure.  At the end of each subsection, we quickly comment on the relationship between the quadric contact cone structure and the cocharacteristic variety of each considered PDE.

\begin{notation}\label{notZetaQu}
From now on, the coordinates on $\CC_{\f}$ dual to the (truncated) total derivatives $\left.D_i^{(1)}\right|_{\f}$ and vectors $\partial_{u_i}|_{\f}$ will be denoted by $z^i$ and $q_i$, respectively:  such a choice is dictated by a purely aesthetic concern. 
\end{notation} 

\subsection{The quadric contact cone structure of the equation $\det\|u_{ij}\|=1$}\label{sec:var.car.det.1}

Let us consider the equation
\begin{equation}\label{eq:det.1}
\E:=\{\det\|u_{ij}\|=1\}\, 
\end{equation}
and apply to it, step by step, the scheme given at the beginning of Section \ref{sec.andata}. 
Recall that  $u_{ij}^\sharp$ denotes  the $(i,j)$-entry of the cofactor matrix of $\|u_{ij}\|$.
\begin{enumerate}
\item Let us fix a point $\ff\in\E$,  
$\ff=\big(x^i_0,u_0,u_i^0,u_{ij}^0\big)$
with
$$
u_{11}^0= \frac{(u_{12}^0)^2u_{33}^0-2u_{12}^0u_{13}^0u_{23}^0+(u_{13}^0)^2u_{22}^0+1}{u_{11}^{0\sharp}}\, ,
$$
assuming $u_{11}^{0\sharp}\neq 0$. We also assume  that $\ff$ is not an elliptic point for the equation \eqref{eq:det.1}.
\item Equation \eqref{eq:char.general.local}, which  reads now 
$\sum_{i,j}u_{ij}^\sharp\eta_i\eta_j=\sum_{i\leq j}(2-\delta_{ij})u_{ij}^\sharp\eta_i\eta_j=0$, can be   solved   with respect to  $\eta_1$,   obtaining 
\begin{equation}\label{eq:eta.1.+}
\eta_1^{\pm}(\eta_2,\eta_3)= \frac{-u_{12}^{0\sharp}\,\eta_2-u_{12}^{0\sharp}\,\eta_3 \pm \sqrt{B}}{u_{11}^{0\sharp}}\, ,
\end{equation}
with
$$
B=B(\eta_2,\eta_3):=-\eta_2^2 u_{33}^0+2 \eta_2 \eta_3 u_{23}^0-\eta_3^2 u_{22}^0\,.
$$
All the rank--1 vectors $\nu_{\ff}$ of the PDE \eqref{eq:det.1} at the (non--elliptic) point $\ff$,
in view of formula \eqref{eq:local.rank.1}, are described by the two $2$--parametric families
\begin{equation}\label{eq:nu.appoggio}
\nu_{\ff}^{\pm}(\eta_2,\eta_3)=(\eta_1^{\pm})^2\partial_{u_{11}}\big|_{\ff} + \eta_1^{\pm}\eta_2\partial_{u_{12}}\big|_{\ff} + \eta_1^{\pm}\eta_3\partial_{u_{13}}\big|_{\ff} +\eta_2^2\partial_{u_{22}}\big|_{\ff} + \eta_2\eta_3\partial_{u_{23}}\big|_{\ff} + \eta_3^2\partial_{u_{33}}\big|_{\ff}\, ,
\end{equation}
with $\eta_1^{\pm}$  given by \eqref{eq:eta.1.+}.
We focus now only on the family $\nu^{+}_{\ff}(\eta_2,\eta_3)$, since, for the other one, computations and reasonings are   the same. Since  we are interested in the line spanned by $\nu^+_{\ff}(\eta_2,\eta_3)$, we shall substitute $\eta_2=t$ and $\eta_3=1$ in \eqref{eq:nu.appoggio}, thus obtaining  $\ell^+_{\ff}(t):=\Span{\nu^+_{\ff}(t,1)}$. Accordingly, we set  $\eta_1^+(t):=\eta_1^+(t,1)$, $B(t)=B(t,1)$.
\item  With the line $\ell^+_{\ff}(t)$   is associated the hyperplane $H(\ell^+_{\ff}(t)\subset L_{\ff}$, which, in view of \eqref{eq:H}, is given by
\begin{equation*}
H(\ell^+_{\ff}(t))
=\Span{   -t D^{(2)}_1(\ff) + \eta_1^+(t)\,D^{(2)}_2(\ff) \,\,,\,\, D^{(2)}_2(\ff) -tD^{(2)}_3(\ff)  }\,.
\end{equation*}
In order to study $\lim_{\epsilon\to 0}\,H(\ell^+_{\ff}(t))\cap H(\ell^+_{\ff}(t+\epsilon))$,    we consider the system
\begin{equation}\label{eq:system.inf.det.1}
\left\{
\begin{array}{l}
\eta_1^+(t)\,\xi^1+t\xi^2+\xi^3=0\, ,
\\
\\
\eta_1^+(t+\epsilon)\,\xi^1+(t+\epsilon)\xi^2+\xi^3=0\, .
\end{array}
\right.
\end{equation}
By solving system \eqref{eq:system.inf.det.1} with respect to  $\xi^2=\xi^2(\xi^1,t,\epsilon)$, $\xi^3=\xi^3(\xi^1,t,\epsilon)$, and then by computing   $\lim_{\epsilon\to 0}\xi^2$ and $\lim_{\epsilon\to 0}\xi^3$, we obtain  
$$
\xi^2= -\frac{\xi_1\left(-\sqrt{B(t)}u_{12}^{0\sharp}-u_{33}^0t+u_{23}^0\right)}
{\sqrt{B(t)}u_{11}^{0\sharp}}\,,\quad
\xi^3=\frac{\xi_1\left(\sqrt{B(t)}u_{13}^{0\sharp}-u_{23}^0t+u_{22}^0\right)}
{\sqrt{B(t)}u_{11}^{0\sharp}}\, ,
$$
so that
$$
v(\ell^{+}_{\ff}(t))=\Span{\, \sqrt{B(t)}u_{11}^{0\sharp}  \,  D^{(2)}_1(\ff)   +\left(\sqrt{B(t)}u_{12}^{0\sharp}+u_{33}^0t-u_{23}^0\right)  D^{(2)}_2(\ff)
+ \xi_1\left(\sqrt{B(t)}u_{13}^{0\sharp}-u_{23}^0t+u_{22}^0\right) D^{(2)}_3(\ff)}\subset L_{\ff}\,.
$$
\item Had  we  considered the family $\nu^{-}_{\ff}(\eta_2,\eta_3)$, we would have come to the line $v(\ell^{-}_{\ff}(t))$.
If we let vary the parameter $t$ and the point $\ff$ on the fiber $\E_{\f}$, the above--found lines $v(\ell^{\pm}_{\ff}(t))$   give a conic variety inside the contact hyperplane $\CC_{\f}$ of $T_{\f}J^1$:
\begin{equation}\label{eq:cone.structure.det.1}
\V_{\f}:z^1q_1+z^2q_2+z^3q_3=0\,.
\end{equation}
\end{enumerate}
By computing the cocharacteristic variety of the same equation \eqref{eq:det.1}, according to Definition \ref{def.cochar.var}, we obtain again \eqref{eq:cone.structure.det.1}.

\begin{remark}\label{rem:partial.Legendre}
The partial Legendre transformation  
$$
(x^i,u,u_i)\to (u_1,x^2,x^3,u-x^1u_1,-x^1,u_2,u_3)=\left(\tilde{x}^1,\tilde{x}^2,\tilde{x}^3,\tilde{u}_1,\tilde{u}_2,\tilde{u}_3\right)\, ,
$$
cf. \eqref{eq:legendre.partial}, 
transforms  equation \eqref{eq:det.1} into (up to a renaming of coordinates)
\begin{equation}\label{eq.u11.minore}
u_{11}+u_{22}u_{33}-u_{23}^2=0
\end{equation}
and the conic  variety \eqref{eq:cone.structure.det.1} into
\begin{equation}\label{eq:cone.structure.u11.minore}
z^1q_1-z^2q_2-z^3q_3=0\,.
\end{equation}
In fact, the conic variety \eqref{eq:cone.structure.u11.minore} is the quadric contact cone structure of   equation \eqref{eq.u11.minore}.
\end{remark}

\subsection{The quadric contact cone structure of the equation $u_{11}-u_{22}-u_{33}=0$}

Let us consider the wave equation
\begin{equation}\label{eq:wave.3d.bis}
\E:=\{u_{11}=u_{22}+u_{33}\}\,.
\end{equation}
As we did in Section \ref{sec:var.car.det.1},
we apply  the same scheme to the equation \eqref{eq:wave.3d.bis}.
\begin{enumerate}
\item Let us fix a point 
$$
\ff=\big(x^i_0,u_0,u_i^0,u_{22}^0+u_{33}^0,u_{12}^0,\dots,u_{33}^0\big)\in\E\,.
$$
\item Equation \eqref{eq:char.general.local} reads 
$
\eta_1^2=\eta_2^2+\eta_3^2\,,
$
so that the rank--1 vectors $\nu_{\ff}$ of the PDE \eqref{eq:wave.3d.bis} at point $\ff$,
in view of formula \eqref{eq:local.rank.1}, are described by the two $2$--parametric families:
$$
\nu^+_{\ff}(\eta_2,\eta_3)=(\eta_2^2+\eta_3^2) \partial_{u_{11}}\big|_{\ff} \pm  \sqrt{\eta_2^2+\eta_3^2}\,\eta_2   \partial_{u_{12}}\big|_{\ff} \pm \sqrt{\eta_2^2+\eta_3^2}\,\eta_3  \partial_{u_{13}}\big|_{\ff} +   \eta_2^2  \partial_{u_{22}}\big|_{\ff} +   \eta_2\eta_3  \partial_{u_{23}}\big|_{\ff} +  \eta_3^2 \partial_{u_{33}}\big|_{\ff}\, .
$$
Once again consider only the first family. 
Since  we are interested in the line spanned by $\nu^+_{\ff}(\eta_2,\eta_3)$, we set $\eta_2=t$, $\eta_3=1$, thus  obtaining 
$$
\ell^+_{\ff}(t)=\Span{\nu^+_{\ff}(t,1)}=\Span{(1+t^2) \partial_{u_{11}}\big|_{\ff} +  \sqrt{1+t^2}\,t   \partial_{u_{12}}\big|_{\ff} + \sqrt{1+t^2}  \partial_{u_{13}}\big|_{\ff} +  t^2 \partial_{u_{22}}\big|_{\ff} +   t  \partial_{u_{23}}\big|_{\ff} +  \partial_{u_{33}}\big|_{\ff}}\,.
$$
\item
With the line $\ell^+_{\ff}(t)$  is associated the hyperplane
\begin{equation*}
H(\ell^+_{\ff}(t))
=\Span{   -t D^{(2)}_1(\ff) + \sqrt{1+t^2}\,D^{(2)}_2(\ff) \,\,,\,\, D^{(2)}_2(\ff) -tD^{(2)}_3(\ff)  }\subset L_{\ff}\,.
\end{equation*}
In order to study the intersection of $H(\ell^+_{\ff}(t))$ with the plane $H(\ell^+_{\ff}(t+\epsilon))$ at the limit $\epsilon\to 0$, 
  we consider the system
\begin{equation}\label{eq:system.inf}
\left\{
\begin{array}{l}
\sqrt{1+t^2}\,\xi^1+t\xi^2+\xi^3=0\, ,
\\
\\
\sqrt{1+(t+\epsilon)^2}\,\xi^1+(t+\epsilon)\xi^2+\xi^3=0\, .
\end{array}
\right.
\end{equation}
By solving system \eqref{eq:system.inf} with respect to  $\xi^2=\xi^2(\xi^1,t,\epsilon)$, $\xi^3=\xi^3(\xi^1,t,\epsilon)$, and then by computing   $\lim_{\epsilon\to 0}\xi^2$ and $\lim_{\epsilon\to 0}\xi^3$, we obtain  
$$
\xi^2=-\frac{t\xi^1}{\sqrt{1+t^2}}\,,\quad \xi^3=-\frac{\xi^1}{\sqrt{1+t^2}}\,,
$$
i.e., the limit solution of \eqref{eq:system.inf} is
\begin{equation*}
v(\ell^{+}_{\ff}(t))=\Span{\sqrt{1+t^2}\,D^{(2)}_1(\ff) - tD^{(2)}_2(\ff) -D^{(2)}_3(\ff)}\subset L_{\ff}\,.
\end{equation*}
\item Had  we   considered the family $\nu^{-}_{\ff}(\eta_2,\eta_3)$, we would have come to the line $v(\ell^{-}_{\ff}(t))$. 
As in the step (4) of the previous Section \ref{sec:var.car.det.1},  the lines $v(\ell^{\pm}_{\ff}(t))$ 
describe the conic variety
\begin{equation}\label{eq:cone.structure.wave}
\V_{\f}:(z^1)^2-(z^2)^2-(z^3)^2=0\,.
\end{equation}
\end{enumerate} 

By computing the cocharacteristic variety of the same equation \eqref{eq:wave.3d.bis}, according to Definition \ref{def.cochar.var}, we obtain again \eqref{eq:cone.structure.wave}.  This is the last case, when the two objects coincide.

\subsection{The quadric contact cone structure of the equation $u_{12}=0$}\label{sec:cone.structure.u11}
The equation
 \begin{equation}\label{eq:Goursat.hyp}
\E:\{u_{12}=0\}\,, 
\end{equation}
considered in this section,  is degenerate in the sense that such is its symbol. Indeed, according to formula \eqref{eq:symbol.intro}, the symbol of \eqref{eq:Goursat.hyp} is equal to $\eta_1\eta_2$, that is a degenerate quadratic form. Let us compute the contact cone structure of such an equation by following the steps described in Section \ref{sec.andata}.
\begin{enumerate}
    \item  Let us fix a point 
    $\ff=\big(x^i_0,u_0,u_i^0,u^0_{11},0,u^0_{13},u^0_{22},u^0_{23},u^0_{33}\big)\in\E$.
    \item Equation \eqref{eq:char.general.local} reads   $\eta_1\eta_2=0$, which  gives either $\eta_1=0$ or $\eta_2=0$. Below we work out the case $\eta_2=0$ as the case $\eta_1=0$ can be treated in the same way. 
    We are going to use ``$+$'' to indicate the case when $\eta_2=0$ and ``$-$'' to indicate the case when $\eta_1=0$. The rank--one directions at $\ff$ that are tangent to $\mathcal{E}$ are 
    $$
    \eta_1^2\partial_{u_{11}}\big|_{\ff} + \eta_1\eta_3\partial_{u_{13}}\big|_{\ff} + \eta_3^2 \partial_{u_{33}}\big|_{\ff}\, ,
    $$
    so that, by letting  $t=\eta_3/\eta_1$, we have  
    $$
    \ell^+_{\ff}(t)=\partial_{u_{11}}\big|_{\ff} + t \partial_{u_{13}}\big|_{\ff} + t^2 \partial_{u_{33}}\big|_{\ff}\,.
    $$
    \item With the line $\ell^+_{\ff}(t)$   is associated the hyperplane 
    $$
    H(\ell^+_{\ff}(t))=\Span{ -tD_1^{(2)}(\ff)+D^{(2)}_3(\ff)\,,\, D^{(2)}_2(\ff) }\subset L_{\ff}\, ,
    $$
    i.e., all planes $H(\ell^+_{\ff}(t))$ contain the line $\Span{D^{(2)}_2(\ff)}$. Thus, the line $v(\ell^{+}_{\ff}(t))$ is independent of $t$:
    \begin{equation}\label{eq:line.J1.Goursat}
    v(\ell^{+}_{\ff}(t))=v(\ell^{+}_{\ff})=\Span{D^{(2)}_2(\ff)}\,.
    \end{equation}
    \item If we let vary the point $\ff$ on the fiber $J^2_{\f}$, the line \eqref{eq:line.J1.Goursat} describes a  $3$-dimensional subspace of the contact hyperplane $\CC_{\f}$:
    \begin{equation*}
    \mathcal{W}_{\f}:=\V^+_{\f}=\Span{D^{(1)}_2(\f)\,,\,\partial_{u_2}\big|_{\f}\,,\,\partial_{u_3}\big|_{\f}}\, .
    \end{equation*}
\end{enumerate}
Had we considered, in the above step $(2)$,  the case $\eta_1=0$,   we would have gotten the subspace $\mathcal{W}^\perp_{\f}$ of $\CC_{\f}$, that is  the symplectic orthogonal to $\mathcal{W}_{\f}$:
$$
\mathcal{W}^\perp_{\f}:=\V^-_{\f}=\Span{D^{(1)}_1(\f)\,,\,\partial_{u_1}\big|_{\f}\,,\,\partial_{u_3}\big|_{\f}}\, .
$$
We come to the conclusion  that the cone structure associated with PDE \eqref{eq:Goursat.hyp} is a pair of mutually symplectic--orthogonal distributions:
\begin{equation}\label{eq:cone.structure.u12}
\V:\f\in J^1\to\V_{\f}=\mathcal{W}_{\f}\cup\mathcal{W}^\perp_{\f}\,.
\end{equation}

Now the notion of a  cocharacteristic variety and that of a contact cone structure begin to diverge: indeed, the cocharacteristic variety of the equation  \eqref{eq:Goursat.hyp}, according to Definition \ref{def.cochar.var}, is the degenerate quadric hypersurface
\begin{equation*}
    \{ (z^3)^2=0\} = \{z^3=0\}\, ,
\end{equation*}
which turns out to be the linear span of  $\mathcal{W}_{\f}\cup\mathcal{W}^\perp_{\f}$.

\subsection{The quadric contact cone structure of the equation $u_{11}=0$}
The reasoning to get to the contact cone structure associated with the equation $u_{11}=0$ is the same as the one we employed in   Section \ref{sec:cone.structure.u11} and then we omit it: the result 
is the Lagrangian distribution
\begin{equation}\label{eq:cone.structure.u11}
\V:\f\in J^1\to\V_{\f}=\Span{D^{(1)}_1(\f)\,,\,\partial_{u_2}\big|_{\f}\,,\,\partial_{u_3}\big|_{\f}}\,.
\end{equation}
It is worth adding that, in this case  the cocharacteristic variety becomes the most degenerate as possible: indeed, the cocharacteristic variety of the equation  $u_{11}=0$, according to Definition \ref{def.cochar.var}, is the whole contact space $\CC_{\f}$.

\subsection{Relation between the contact cone structure and the cocharacteristic variety of a Monge--Amp\`ere equation}

All the examples worked out in the previous four subsections point towards the existence of a natural relationship between the proposed construction of a quadric cone structure associated with a Monge--Amp\`ere equation, cf. \eqref{eq.MAE},  and their cocharacteristic variety (see Definition \ref{def.cochar.var}): this is captured by the following theorem. 

\begin{theorem}\label{thMain1}
Let $\mathcal{E}$ be a Monge--Amp\`ere equation \eqref{eq.MAE} and let $\ff\in \mathcal{E}_{\f}$, cf. \eqref{eq:fibre.PDE}, be an its regular point, see also Remark \ref{rem.regular.points}. Then the following are true.
\begin{enumerate}
\item If the symbol of $\mathcal{E}$ is not degenerate at $\ff$, then the contact cone structure of $\mathcal{E}$ at $\f$ is the cocharacteristic variety of $\mathcal{E}$ at $\f$.
\item 
If the symbol of $\mathcal{E}$ has rank $2$ at $\ff$ and it is hyperbolic in this point, then its contact cone structure  at $\f$ is the union $D\cup D^\perp$ of two symplectic-orthogonal $3$-dimensional subspaces of $\CC_{\f}$ whereas the cocharacteristic variety of $\mathcal{E}$ at $\f$ describes the smallest linear subspace of $\CC_{\f}$ containing $D\cup D^\perp$.
\item If the symbol of $\mathcal{E}$ has rank $1$ at $\ff$, then 
the contact cone structure of $\E$ at $\f$ is a Lagrangian subspace of $\CC_{\f}$ and the cocharacteristic variety of $\mathcal{E}$ at  $\f$ is trivial.
\end{enumerate}
\end{theorem}

In the Section \ref{secLinSec}  we reformulate  Theorem \ref{thMain1}  over the field of complex numbers: the so--obtained  Corollary \ref{corEquivalencyMomentCoChar} represents then a coarse proof of Theorem \ref{thMain1}; indeed, in the complex case, the four examples above exhaust all possible isomorphism types of Monge--Amp\`ere equations. A finer and as such complete proof can be easily obtained by modifying the signature in the given examples.

\section{Reconstructing a $2\Nd$ order PDE from a contact cone structure}\label{secBackRecepi}
Now we try to reverse the above recipe, i.e., starting from an arbitrary contact cone structure $\V$, we propose two different methods of associating a $2\Nd$ order PDE with $\V$: the reader will immediately recognize in the second a ``degenerate version" of the first. Since there is plenty of contactomorphism types of contact cone structures (even considering only the quadratic ones), in the face of only four contactomorphism types of (symplectic) Monge--Amp\`ere equations, a general ``inverse recipe" would necessarily exceed the class of PDE under consideration. This is why we propose below only two versions: they will be just enough to reconstruct all Monge--Amp\`ere equations.  


\subsection{The case of a $5$-dimensional contact cone structure}\label{sec:ritorno.non.degenere}

Let $\V$ be a contact cone structure on $J^1$ and 
let us assume that $\dim(\V_{\f})=5\,\,\forall\f\in J^1$. Starting from $\V_{\f}$, we will be constructing a distribution (not necessarily of constant rank) on $J^2_{\f}$ by working out  the following steps, 
that represent a sort of inverse procedure to the one described in Section \ref{sec.andata}.

\begin{enumerate}
\item Let us consider $\ff\equiv L_{\ff}\in J^2_{\f}$ and set  $\V_{\ff}:=L_{\ff}\cap \V_{\f}$: then, generically,  $\dim \V_{\ff}=2$.
\item If $\dim \V_{\ff}=2$, then to point $\ff$ we can associate the set
$$
H_{\ff}:=\{\text{Hyperplanes of $L_{\ff}$ tangent to $\V_{\ff}$ along its generatrices}\}\,.
$$
The set $H_{\ff}$ depends on one parameter, i.e., we have a 1--parametric family 
$H_{\ff}(t)$ 
of hyperplanes.
\item Let $\ell_{\ff}(t)$ be the line of rank $1$ corresponding to $H_{\ff}(t)$ via \eqref{eq:correspondence}.
\item Let $\mathcal{D}_{\ff}$ be the smallest linear subspace containing $\ell_{\ff}(t)$ $\forall\,t$. Then the correspondence $\ff\to\mathcal{D}_{\ff}$ defines
    a distribution $\mathcal{D}$ on $J^2_{\f}$: its integral submanifolds  will be submanifolds of $J^2_{\f}$, i.e., fibers of PDEs. 
\end{enumerate}

By starting from the contact cone structure associated with a symplectic Monge--Amp\`ere equation, the above procedure leads  to a foliation of  $J^2$ and then, each leaf of it will be a PDE.

\subsubsection{Foliation of PDEs associated with  the contact cone structure \eqref{eq:cone.structure.u11.minore}}\label{sec:ritorno.u11.minore}

At the end of Section \ref{sec:var.car.det.1} we have seen that equations $\det\|u_{ij}\|=1$ and $u_{11}+u_{22}u_{33}-u_{23}^2=0$ are contactomorphic; therefore, we can consider them, as well as their contact cone structures, as equivalent. 
In particular, in this section, 
we will be working with the contact cone structure \eqref{eq:cone.structure.u11.minore} because the computations are easier.\par

Let us apply the scheme explained above at the beginning of   Section \ref{sec:ritorno.non.degenere}, in order to construct the $2\Nd$ order PDEs associated with the contact cone structure $\V_{\f}$ given by \eqref{eq:cone.structure.u11.minore}. 
We employ the coordinates $(z^i,q_i)$ on $\CC_{\f}$  introduced in Notation \ref{notZetaQu}.

\begin{enumerate}
\item Let us fix $\ff=\big(x^i_0,u_0,u_i^0,u_{ij}^0\big)=\big(\f,u_{ij}^0\big)\in J^2_{\f}$, cf. \eqref{eq:theta.2.bello}, or, equivalently, $L_{\ff}=\langle D^{(1)}_i(\f)+u^0_{ij}\partial_{u_j}\big|_{\f} \rangle_{i=1,2,3}
=\Span{D_i^{(2)}(\ff)}_{i=1,2,3}$. Since $L_{\ff}$, as a vector subspace of $\CC_{\f}$, is locally described by $q_i-u_{ij}^0z^j=0$,
$$
\V_{\ff}=\{ q_i-u_{ij}^0z^j=0\,,\, z^1q_1-z^2q_2-z^3q_3=0  \}\,.
$$
\item Then we have
\begin{equation}\label{H.theta2.ritorno.u11.minore}
H_{\ff}=\left\{ \xi^i D_i^{(2)}(\ff)=\xi^i D^{(1)}_i(\f)+\xi^i u^0_{ij}\partial_{u_j}\big|_{\f} \right\}
\end{equation}
with $\xi^i$ satisfying
\begin{equation}\label{eq:system.ritorno.u11.minore}
\bar{q}_i-u_{ij}^0\bar{z}^j=0\,,\,\, \bar{z}^1\bar{q}_1-\bar{z}^2\bar{q}_2-\bar{z}^3\bar{q}_3=0\,,\,\, \bar{q}_1\xi^1 - \bar{q}_2\xi^2 -\bar{q}_3\xi^3 + \bar{z}^1\xi^i u^0_{i1} - \bar{z}^2\xi^i u^0_{i2} -\bar{z}^3\xi^i u^0_{i3}=0\,.
\end{equation}
By a direct computation, from the first four equations of the system \eqref{eq:system.ritorno.u11.minore} we obtain 
\begin{equation}\label{eq:z3.ritorno.u11.minore}
\bar{z}^3=\frac{-u^0_{23}\bar{z}^2\pm \sqrt{u^0_{11}u^0_{33}(\bar{z}^1)^2-u^0_{22}u^0_{33}(\bar{z}^2)^2 + (u^0_{23})^2(\bar{z}^2)^2} }{u^0_{33}}=: \frac{-u^0_{23}\bar{z}^2\pm \sqrt{A(\bar{z}^1,\bar{z}^2)} }{u^0_{33}}
\end{equation}
and, assuming  $u^0_{33}\neq 0$,  the last equation of the system \eqref{eq:system.ritorno.u11.minore} yields 
$$
u^0_{11}u^0_{33}\bar{z}^1\xi^1 - u^0_{22}u^0_{33}\bar{z}^2\xi^2 + (u^0_{23})^2\bar{z}^2\xi^2  \mp \sqrt{A(\bar{z}^1,\bar{z}^2)}\, u^0_{23}\xi^2\mp\sqrt{A(\bar{z}^1,\bar{z}^2)}\,u^0_{33}\xi^3=0\,.
$$
Let us consider the case with the plus sign in \eqref{eq:z3.ritorno.u11.minore}. By setting $\bar{z}^1=t$ and $\bar{z}^2=1$,   equation above  becomes
$$
u^0_{11}u^0_{33}t\,\xi^1 - u^0_{22}u^0_{33}\xi^2 + (u^0_{23})^2\xi^2  - \sqrt{A(t)}\, u^0_{23}\xi^2 -\sqrt{A(t)}\,u^0_{33}\xi^3=0\,,\quad A(t):=A(t,1)\, ,
$$
whose solution $(\xi^1(t),\xi^2(t),\xi^3(t))$, substituted in \eqref{H.theta2.ritorno.u11.minore}, gives $H_{\ff}(t)$ that we were looking for.

\item
By looking at \eqref{eq:local.rank.1}--\eqref{eq:H}, the rank--one line $\ell_{\ff}(t)$ corresponding to $H_{\ff}(t)$ is 
$$
\ell_{\ff}(t)=\Span{\sum_{i\leq j}\eta_i\eta_j\partial_{u_{ij}}\big|_{\ff}}\, ,
$$
where
$$
\eta_1=u^0_{11}u^0_{33}t\,, \quad \eta_2= -u^0_{22}u^0_{33} + (u^0_{23})^2  - \sqrt{A(t)}\, u^0_{23} \,, \quad \eta_3=-\sqrt{A(t)}\,u^0_{33}\, .
$$
\item The smallest linear subspace $\mathcal{D}_{\ff}$ containing $\ell_{\ff}(t)$ for any $t$ is $\mathcal{D}_{\ff}=\Span{X_1|_{\ff},X_2|_{\ff},X_3|_{\ff},X_4|_{\ff},X_5|_{\ff}}$ where the  vector fields $X_i$ on $J^2_{\f}$  are:
\begin{eqnarray*}
X_1&=&u_{11}u_{23}u_{33}\frac{\partial}{\partial u_{12}}+u_{11}u_{33}^2\frac{\partial}{\partial u_{13}}\, ,\\
X_2&=&-2u_{22}u_{23}u_{33}\frac{\partial}{\partial u_{22}}+2u_{23}^3\frac{\partial}{\partial u_{22}}-u_{22}u_{33}^2\frac{\partial}{\partial u_{23}}+u_{23}^2u_{33}\frac{\partial}{\partial u_{23}}\, ,\\
X_3&=&u_{11}^2u_{33}^2\frac{\partial}{\partial u_{11}}+u_{11}u_{23}^2u_{33}\frac{\partial}{\partial u_{22}}+u_{11}u_{23}u_{33}^2\frac{\partial}{\partial u_{23}}+u_{11}u_{33}^3\frac{\partial}{\partial u_{33}}\, ,\\
X_4&=&-u_{11}u_{22}u_{33}^2\frac{\partial}{\partial u_{12}}+u_{11}u_{23}^2u_{33}\frac{\partial}{\partial u_{12}}\, ,\\
X_5&=&u_{22}^2u_{33}^2\frac{\partial}{\partial u_{22}}-3u_{22}u_{23}^2u_{33}\frac{\partial}{\partial u_{22}}+2u_{23}^4\frac{\partial}{\partial u_{22}}-u_{22}u_{23}u_{33}^2\frac{\partial}{\partial u_{23}}
+u_{23}^3u_{33}\frac{\partial}{\partial u_{23}}
-u_{22}u_{33}^3\frac{\partial}{\partial u_{33}}
+u_{23}^2u_{33}^2\frac{\partial}{\partial u_{33}}\, .
\end{eqnarray*}
A direct computation shows that the vector  distribution $\ff\in J^2_{\f}\to\mathcal{D}_{\ff}\subseteq T_{\ff}J^2_{\f}$ is integrable, so it admits a $1$-parametric family of integral submanifolds locally given by $f=0$, where $f$ is the   general solution to the system $\{X_1(f)=X_2(f)=X_3(f)=X_4(f)=X_5(f)=0\}$, which is
$$
f=K_1u_{11}+K_2(u_{22}u_{33}-u_{23}^2)\,,\quad K_i\in\mathbb{R}\,,
$$
so that the PDEs we were looking for are given by  $K_1u_{11}+K_2(u_{22}u_{33}-u_{23}^2)=0$. 

\end{enumerate}

\subsubsection{Foliation of PDEs associated with  the contact cone structure \eqref{eq:cone.structure.wave}}
We consider now 
the cone structure  \eqref{eq:cone.structure.wave} and we perform the same steps as 
in Section \ref{sec:ritorno.u11.minore} above. We report only the final result of the computations, i.e., the distribution $\mathcal{D}$ on $J^2_{\f}$ constructed starting from \eqref{eq:cone.structure.wave}:
$$
\mathcal{D}=\langle  \partial_{u_{11}} + \partial_{u_{22}}\,,\, \partial_{u_{11}} + \partial_{u_{33}}\,,\, \partial_{u_{12}}\,,\,\partial_{u_{13}} \,,\, \partial_{u_{23}}   \rangle\, .
$$
The integral manifolds of $\mathcal{D}$ are described by
$$
u_{11}-u_{22}-u_{33}=K\,,\quad K\in\mathbb{R}\,.
$$

\subsection{The case of a degenerate  $3$-dimensional contact cone structure}\label{sec:ritorno.degenere}

In step (3) of the recipe at the beginning of Section \ref{sec:ritorno.non.degenere} we have seen  that the hyperplanes $H_{\ff}(t)$ are tangent to $\V_{\ff}$ along its generatrices. If we do not assume particular properties of $\V$, one can have a unique generatrix of $\V_{\ff}$. This happens, for instance, in the case when  $\V$ is a $3$-dimensional vector distribution on $J^1$, i.e., a particular degenerate contact cone structure on $J^1$, that we study in details below. We 
give here a similar scheme to the one given in Section \ref{sec:ritorno.non.degenere}, that allows us   to define a (non-constant rank) distribution $\mathcal{D}$ on $J^2_{\f}$ starting from $\V_{\f}$.


\begin{enumerate}
\item Let us consider $\ff\equiv L_{\ff}\in J^2_{\f}$ and let  $\V_{\ff}:=L_{\ff}\cap \V_{\f}$. 
%
%
\item 
To the point $\ff$ we associate the set
$$
H_{\ff}:=\{\text{Hyperplanes of $L_{\ff}$ containing $\V_{\ff}$}\}\,.
$$
\item To each element $h\in H_{\ff}$ we associate the rank--one line $\ell_{\ff}(h)$ via \eqref{eq:correspondence}.
\item Let $\mathcal{D}_{\ff}$ be the smallest linear subspace containing $\ell_{\ff}(h)$ $\forall\,h\in H_{\ff}$. Then the correspondence $\ff\to\mathcal{D}_{\ff}$ defines
    a (non-constant rank) distribution $\mathcal{D}$ on $J^2_{\f}$: its integral submanifolds will be submanifolds of $J^2_{\f}$, i.e., fibers of PDEs.
\end{enumerate}

\subsubsection{PDEs associated with  the contact cone structure \eqref{eq:cone.structure.u12}}\label{sec.rec.u12}

We work out the above steps  
in the case of the contact cone structure \eqref{eq:cone.structure.u12}. A key remark  is that $L_{\ff}$ intersects $\mathcal{W}_{\f}$ if and only if it intersects also $\mathcal{W}^\perp_{\f}$: in particular, it is enough to   study the intersection $\V^+_{\ff}=L_{\ff}\cap\mathcal{W}_{\f}$, since for $\V^-_{\ff}=L_{\ff}\cap\mathcal{W}^\perp_{\f}$ the reasonings are  the same.
\begin{enumerate}
    \item 
Let us fix $\ff=\big(x^i_0,u_0,u_i^0,u_{ij}^0\big)=\big(\f,u_{ij}^0\big)\in J^2_{\f}$, cf.  \eqref{eq:theta.2.bello}, or, equivalently, $L_{\ff}=\langle D^{(1)}_i(\f)+u^0_{ij}\partial_{u_j}\big|_{\f} \rangle_{i=1,2,3}
=\Span{D_i^{(2)}(\ff)}_{i=1,2,3}$. The dimension of $\V^+_{\ff}=L_{\ff}\cap\mathcal{W}_{\f}$ is the corank of the matrix
\begin{equation}\label{eq:matrix.u12}
    \left(
    \begin{array}{cccccc}
         1 & 0 & 0 & u^0_{11} & u^0_{12} & u^0_{13} \\
         0 & 1 & 0 & u^0_{12} & u^0_{22} & u^0_{23} \\
         0 & 0 & 1 & u^0_{13} & u^0_{23} & u^0_{33} \\
         0 & 1 & 0 & 0 & 0 & 0 \\
         0 & 0 & 0 & 0 & 1 & 0 \\
         0 & 0 & 0 & 0 & 0 & 1 \\
    \end{array}
    \right)\, .
\end{equation}
The rank of   matrix \eqref{eq:matrix.u12} is $6$ if and only if $u^0_{12}\neq 0$ and it is $5$ otherwise. No other cases can occur.
\item If the rank of matrix \eqref{eq:matrix.u12} is $6$, then $\dim(L_{\ff}\cap\mathcal{W}_{\f})=0=\dim(L_{\ff}\cap\mathcal{W}^\perp_{\f})$ and $H_{\ff}$ consists of all the hyperplanes of $L_{\ff}$. If the rank of matrix \eqref{eq:matrix.u12} is $5$,  then $\dim(L_{\ff}\cap\mathcal{W}_{\f})=1$ and $H^+_{\ff}$ consists of the hyperplanes of $L_{\ff}$ containing the line $\V^+_{\ff}=L_{\ff}\cap\mathcal{W}_{\f}$.
\item In the case $\dim(L_{\ff}\cap\mathcal{W}_{\f})=0$ ($=\dim(L_{\ff}\cap\mathcal{W}^\perp_{\f})$), the set $\{\ell_{\ff}(h)\}$ consists of all the rank--one lines of $J^2_{\f}$. 

\noindent
In the case $\dim(L_{\ff}\cap\mathcal{W}_{\f})=1$ we have that $\V^+_{\ff}=\Span{D_2^{(2)}(\ff)}=\Span{D_2^{(1)}(\ff)+u^0_{22}\partial_{u_2}\big|_{\ff}  +u^0_{23}\partial_{u_3}\big|_{\ff}   }$. In view of the previous point, we have that $H^+_{\ff}=\{H^+_{\ff}(t)\}$, where 
$$
 H^+_{\ff}(t)=\Span{-tD_1^{(2)}(\ff)+D^{(2)}_3(\ff)\,,\,D_2^{(2)}(\ff) }\, ,
$$
so that the set $\{\ell_{\ff}(h)\}$ equals $\{ \ell^+_{\ff}(t)\mid t\in\mathbb{R}\}$, where 
\begin{equation}\label{eq:appoggio+}
\ell^+_{\ff}(t)=\Span{\partial_{u_{11}}\big|_{\ff} + t \partial_{u_{13}}\big|_{\ff} + t^2 \partial_{u_{33}}\big|_{\ff} }\, ,
\end{equation}
cf.   \eqref{eq:H}.  Had we considered the case $\dim(L_{\ff}\cap\mathcal{W}^\perp_{\f})=1$, we would have gotten 
\begin{equation}\label{eq:appoggio-}
\ell^-_{\ff}(t)=\Span{\partial_{u_{22}}\big|_{\ff} + t \partial_{u_{23}}\big|_{\ff} + t^2 \partial_{u_{33}}\big|_{\ff} }\,.
\end{equation}

\item Taking into account that the smallest linear space containing \eqref{eq:appoggio+} and \eqref{eq:appoggio-}, for any $t\in\mathbb{R}$ is 
$$
\Span{\partial_{u_{11}}|_{\ff},\partial_{u_{13}}|_{\ff},\partial_{u_{22}}|_{\ff}, 
\partial_{u_{23}}|_{\ff},\partial_{u_{33}}|_{\ff}}\,,
$$
combining above points $(1)-(3)$, the distribution $\mathcal{D}$ on $J^2_{\f}$ turns out to be
\begin{equation*}
    \mathcal{D}=\Span{\partial_{u_{11}}\,,\, u_{12}\partial_{u_{12}}\,,\, \partial_{u_{13}}\,,\, \partial_{u_{22}}\,,\, \partial_{u_{23}}\,,\, \partial_{u_{33}} }\, .
\end{equation*}
The only $5$--dimensional integral submanifold of $\mathcal{D}$ is described by $u_{12}=0$.
\end{enumerate}

\subsubsection{PDEs associated with the contact cone structure \eqref{eq:cone.structure.u11}}

In the case of \eqref{eq:cone.structure.u11}, the computations to get to the distribution $\mathcal{D}$ of step (4) of the scheme given at the beginning of Section \ref{sec:ritorno.degenere} closely follow those of Section \ref{sec.rec.u12}, so we omit them. The result is the  distribution 
\begin{equation*}
    \mathcal{D}=\Span{u_{11}\partial_{u_{11}}\,,\, u_{11}\partial_{u_{12}}\,,\, u_{11}\partial_{u_{13}}\,,\, \partial_{u_{22}}\,,\, \partial_{u_{23}}\,,\, \partial_{u_{33}} }\, .
\end{equation*}
Note that $\dim{D_{\ff}}=6$ if $u_{11}\neq 0$ and that the only $5$--dimensional integral submanifold of $\mathcal{D}$ is described by $u_{11}=0$.

\section{The space $\p\Lambda_0^3(\CC)$ of symplectic 3D Monge--Amp\`ere equations}\label{secMAESpace}
\begin{warning}
From now on we will be working over the field of complex numbers; we retain the symbol $\CC$ for a 6--dimensional (complex) linear symplectic space but we no longer make a distinction between ``total derivatives" and ``vertical vectors", see \eqref{eq:piano.contatto}: we will have a generic bi--Lagrangian splitting of $\CC$ instead.
\end{warning}
By regarding $\CC$ as the contact plane $\CC_{\f}$ at a generic point $\f\in J^1$, and by replacing $\C$ with $\R$, the reader will immediately see how the constructions obtained below mirror analogous results in the real--differentiable  case of symplectic Monge--Amp\`ere equations and contact cone structures; with only one major caveat: structures that are equivalent over $\C$ need not to be equivalent over $\R$. 

\subsection{The symplectic space $\CC$}\label{subCCstar}


We define $\CC$ by fixing a subspace $V\subset\CC$, such that 
\begin{equation*}
    \CC:=V\oplus V^*\, ,
\end{equation*}
and the symplectic form $\omega$ corresponds to $(0,\id_V,0)$ in the splitting
\begin{equation*}
    \Lambda^2(V\oplus V^*)=\Lambda^2(V)\oplus\End(V)\oplus\Lambda^2(   V^*)\, .
\end{equation*}
A choice of a basis of $V$, and its dual in $V^*$, for instance,
 \begin{equation}\label{eqBasiViViDuale}
V=\Span{e_1,e_2,e_3}\, ,\quad V^*=\Span{\epsilon^1,\epsilon^2,\epsilon^3}\, ,\quad  \epsilon^i(e_j)=\delta_j^i\, ,
\end{equation}
leads to the basis 
\begin{equation}\label{eqCoordinatesOnC}
e_1,\, e_2,\, e_3,\, e_4:=\epsilon^1,\, e_5:=\epsilon^2,\, e_6:=\epsilon^3\, 
\end{equation}
of $\CC$, such that   the symplectic form $\omega$ looks like
\begin{equation*}
\omega=\epsilon^i\wedge e_i\in \Lambda^2(\CC^*)\, .
\end{equation*}
At risk of sounding redundant, we set
\begin{equation}\label{eqCoordinatesXonC}
 x^1:=\epsilon^1,\, x^2:=\epsilon^2,\, x^3:=\epsilon^3\, ,x^4:=e_1\, ,x^5:=e_2\, ,x^6:=e_3\, 
\end{equation}
and we regard these $x^i$'s as linear functions on $\CC$, 
that is, as basis elements of $\CC^*$; the usefulness of such a choice will become clearer in the sequel. Observe that, by construction,  $x^j(e_i)=\delta_i^j$ and 
\begin{equation}\label{eq:symplectic.1}
\omega=x^1\wedge x^4+x^2\wedge x^5+x^3\wedge x^6\,.
\end{equation}
The isomorphism $\CC\simeq\CC^*$ given by \eqref{eq:symplectic.1}, acts on the basis elements $e_1,\ldots, e_6$ of $\CC$ as follows:
\begin{equation}\label{eq:C.C.Star}
e_1\to x^4\,,\,\, e_2\to x^5\,,\,\,e_3\to x^6\,,\,\, e_4\to -x^1\,,\,\, e_5\to -x^2\,, \,\,e_6\to -x^3\,.
\end{equation}
\begin{remark}\label{remDictionary}
If we regard $\CC$ as $\CC_{\f}$, then $e_i\leftrightarrow \left.D_i^{(1)}\right|_{\f}$ and $\epsilon^i\leftrightarrow\left.\partial_{u_i}\right|_{\f}$. Therefore, each $x^i$ of \eqref{eqCoordinatesXonC} corresponds precisely to  $\left. d y^i\right|_{\CC_{\f}}$ that appears in \eqref{eq:Phi.per.MAE}, $i=1,\ldots,6$: it follows that $\epsilon^i=z^i$ and $e_i=q_i$, $i=1,2,3$, where $z^i$ and $q_i$ are as in Notation \ref{notZetaQu}. 
\end{remark}
\begin{remark}\label{rem:legendre.and.partial.point}
The assignment 
\begin{equation}\label{eq:legendre.in.un.punto}
(x^1,x^2,x^3,x^4,x^5,x^6)\to (x^4,x^5,x^6,-x^1,-x^2,-x^3)
\end{equation}
defines a transformation of $\CC$ that preserves the symplectic form \eqref{eq:symplectic.1}. By borrowing the terminology from Example~\ref{ex:Legendre}, we call \eqref{eq:legendre.in.un.punto} a \emph{total Legendre transformation}: indeed, \emph{partial} Legendre transformations can be defined as well,  along the lines of   \eqref{eq:legendre.partial}. 
\end{remark}

\subsection{The moment map identifying   $\sp(\CC)$ with $S^2(\CC^*)$}\label{subsIdentSpQuadraticForms}
The Lie group $\Sp(\CC)$ of symplectomorphisms of $\CC$ is defined as usual:
\begin{equation*}
    \Sp(\CC):=\{g\in\GL(\CC)\mid g^*(\omega)=\omega\}\, .
\end{equation*}
For any $X\in\gl(\CC)=\End(\CC)$ the following contraction of $X$ with $\omega$, namely 
\begin{equation}\label{eqFormaQuIcs}
    Q_X(a,b):=\omega(X(a),b)\, ,\quad \forall a,b\in\CC\, ,
\end{equation}
defines a quadratic form $Q_X$ on $\CC$; this allows for a transparent description of the Lie algebra
\begin{equation*}
    \sp(\CC):=\{X\in\gl(\CC) \mid Q_X\ \textrm{is symmetric}\}\, 
\end{equation*}
of the group $\Sp(\CC)$.\par 
By regarding $\gl(\CC)=\CC^*\otimes \CC$ as the linear part of the (graded) algebra $\X(\CC)$ of polynomial vector fields on the (linear) symplectic manifold $\CC$, the natural  embedding
\begin{equation}\label{eqEmbedJay}
    j:\sp(\CC)\longrightarrow\X(\CC)
\end{equation}
realizes an element $X\in \sp(\CC)$ as a (linear) vector field $j(X)$ on $\CC$; the fact that $X\in \sp(\CC)$ translates into $j(X)$ being a symplectic,  even Hamiltonian,  vector field. It makes then sense to consider the associated \emph{moment map}, that is the $\Sp(\CC)$--equivariant map
\begin{equation*}
    \mu:\CC\longrightarrow \sp(\CC)^*
\end{equation*}
unambiguously defined by
\begin{equation}\label{eqDefMomMapBanal}
    d\langle \mu, X\rangle = j(X)\ \ins\  \omega \, ,\quad\forall X\in\sp(\CC)\,  .
\end{equation}
It follows immediately that
\begin{eqnarray*}
\mu(a): \sp(\CC) &\longrightarrow & \C\, ,\\
X &\longmapsto & Q_X(a,a)\, ,
\end{eqnarray*}
for any $a\in\CC$. Formula \eqref{eqDefMomMapBanal} tells precisely that the \emph{linear} map $j$ associating with any element of $\sp(\CC)$ its Hamiltonian vector field $j(X)$ arises as the differential of the \emph{quadratic} map $\mu$; as such, the latter factors  through the Veronese embedding $v_2$, i.e., diagram
\begin{equation}\label{eqIdS2CSpC}
    \xymatrix{
    S^2(\CC)\ar[r]^{\phi^*} & \sp(\CC)^*\\
    \CC\ar[ur]_\mu\ar[u]^{v_2} &
    }
\end{equation}
commutes. 
We labeled the upper arrow by $\phi^*$, because it is precisely the dual isomorphism to
\begin{eqnarray}
\phi: \sp(\CC) &\longrightarrow & S^2(\CC^*)\, ,\label{eqIdentPhi}\\
X &\longmapsto & Q_X \, .\nonumber
\end{eqnarray}
This is just another way to prove that $\sp(\CC)$ is naturally identified with  $S^2(\CC^*)$: to obtain this well--known identification, we employed  the moment map of the natural $\Sp(\CC)$--action on $\CC$, whose quadratic character is responsible for the appearance of the second symmetric power of $\CC^*$; 
we stressed this elementary phenomenon here, because it will reappear later on in Section \ref{secMomentMap} when we will be performing an analogous construction on the space of Monge--Amp\`ere equations.
\subsubsection{The identification in matrix form}
In view of the obvious decomposition
\begin{equation*}
    \gl(\CC)=\gl(V)\oplus V^{\otimes 2}\oplus V^{\ast\otimes 2}\oplus \gl(V^*)\, ,
\end{equation*}
any element $X\in \gl(\CC)$ can be presented, by employing bases \eqref{eqBasiViViDuale},  as
\begin{equation*}
    X=S_i^j\epsilon^i\otimes e_j+R^{ij}e_i\otimes e_j+T_{ij}\epsilon^i\otimes\epsilon^j+U_j^ie_i\otimes\epsilon^j\, ,
\end{equation*}
that is
\begin{equation*}
    X=\left( \begin{array}{cc}
        S & R \\
        T & U
    \end{array} \right)\, .
\end{equation*}
Easy computations shows that
\begin{equation*}
    Q_{\epsilon^i\otimes e_j}=e_j\otimes\epsilon^i\, ,\quad
      Q_{e_i\otimes e_j}=-e_j\otimes e_i\, ,\quad 
     Q_{\epsilon^i\otimes \epsilon^j}=\epsilon^j\otimes\epsilon^i\, ,\quad
       Q_{e_i\otimes \epsilon^j}=-\epsilon^j\otimes e_i\, ,
\end{equation*}
cf. \eqref{eqFormaQuIcs}, whence
\begin{equation}\label{eqQXNotSymmetric}
    Q_X=S_i^je_j\otimes\epsilon^i-R^{ij}e_j\otimes e_i+T_{ij}\epsilon^j\otimes\epsilon^i-U_j^i\epsilon^j\otimes e_i\, 
\end{equation}
is symmetric if and only if both $R$ and $T$ are symmetric and, moreover, $U=-S^t$; these are the conditions that  single out the     21--dimensional Lie sub--algebra 
\begin{equation*}
    \sp(\CC)=\left\{  \left( \begin{array}{cc}
        S & R \\
        T & -S^t
    \end{array} \right)\mid R=R^t\,, T=T^t   \right\} \, 
\end{equation*}
of $\gl(\CC)$.\par 
Therefore, if $X\in\sp(\CC)$, then \eqref{eqQXNotSymmetric} reads
\begin{equation*}
    Q_X= T_{ij}\epsilon^i\epsilon^j+S_i^j\epsilon^i e_j-R^{ij}e_i  e_j \in S^2(\CC^*)
\end{equation*}
and  \eqref{eqIdentPhi} reads
\begin{equation*}
     \phi:\left( \begin{array}{cc}
        S & R \\
        T & -S^t
    \end{array} \right)\longmapsto  \left( \begin{array}{cc}
        T & S \\
        S^t & -R
    \end{array} \right)\, .
\end{equation*}
It is easy to see that $\phi$ is a $\Sp(\CC)$--module isomorphism: indeed, for all $g\in\Sp(\CC)$ we have
\begin{equation*}
    g^*(Q_X)(a,b)=Q_X(g\cdot a,g\cdot b))=\omega(X(g\cdot a),g\cdot b))=\omega(g\cdot X( a),g\cdot b))=g^*(\omega)(X(a),b)=\omega(X(a),b)=Q_X(a,b)\, .
\end{equation*}
\subsection{The space $\Lambda^3(\CC)$ and its subspace $\Lambda_0^3(\CC)$}\label{subsPluckerSpace}
It is now convenient to introduce the notation
\begin{equation*}
    e_{i_1i_2\cdots i_k} \coloneq e_{i_1} \wedge \cdots \wedge e_{i_k}\, ,\quad\forall i_1,\ldots,i_k=1,2,\ldots,6\, ,\quad k=2,3,\ldots, 6\, ,
\end{equation*}
with the obvious identifications
\begin{equation*}
    e_{i_1i_2\cdots i_k}=\sign(\sigma)e_{\sigma(i_1)\sigma(i_2)\cdots \sigma(i_k)}\, .
\end{equation*}
By defining the dual symbols in an analogous way, i.e.,
\begin{equation*}
    x^{i_1i_2\cdots i_k}=x^{i_1}\wedge\cdots\wedge x^{i_k}\, ,
\end{equation*}
we see that, for example,   the symplectic form \eqref{eq:symplectic.1}
\begin{equation*}
\omega=x^{14}+x^{25}+x^{36}\,.
\end{equation*}
We warn the reader that the justapoxition of symbols, e.g., $e_{123}e_{456}$  denotes the \emph{symmetric} product and not the anti--symmetric one, i.e., $e_{123456}$.\par
Modulo these identifications, we have exactly 20 symbols, that correspond to as many generators of the 20--dimensional space $\Lambda^3(\CC)$;  the latter is equipped with a naturally defined $\Lambda^6(\CC)$--valued symplectic form $\Omega$:
\begin{eqnarray}
    \Omega: \Lambda^3(\CC)\times\Lambda^3(\CC) & \longrightarrow & \Lambda^6(\CC)\, ,\label{eqOmegaBig}\\
    (\alpha,\beta)&\longrightarrow & \alpha\wedge\beta\, .\nonumber
\end{eqnarray}
Skipping the twisting factor $e_{123456}$, the symplectic form $\Omega$ looks like  
\begin{equation*}
    \Omega = x^{123}\wedge x^{456}-\begin{pmatrix}
x^{423} & x^{143} & x^{124} \\
x^{523} & x^{153} & x^{125} \\
x^{623} & x^{163} & x^{126}
\end{pmatrix}\wedge \begin{pmatrix}
x^{156} & x^{256} & x^{356} \\
x^{416} & x^{426} & x^{436} \\
x^{451} & x^{452} & x^{453}
\end{pmatrix}\in \Lambda^2(\Lambda^3(\CC^*))\, ,
\end{equation*}
having understood (by borrowing the notiation from  \cite[Section 2.2]{Iliev2005}) the wedge product of the matrices above in the following way: 
\begin{equation*}
    \|a^{ij}\|\wedge\|b^{ij}\|:=\sum_{i,j}a^{ij}\wedge b^{ij}\, .
\end{equation*}
By linearity with respect to the wedge product, the natural $\Sp(\CC)$--action extends  to the whole exterior algebra $\Lambda^\bullet(\CC)$: the resulting $\Sp(\CC)$--module $\Lambda^3(\CC)$ is not, however, irreducible since it contains  the space of 3--forms that are multiple of   $\omega$, which is   a copy of the 6--dimensional fundamental representation. The remaining 14--dimensional  constituent, henceforth denoted by $\Lambda_0^3(\CC)$, is irreducible and  can be described as follows.\par
Let $i_\omega:\Lambda^{3}(\CC)\longrightarrow \CC$ denote the insertion of $\omega$ and $m_{\omega^{-1}}$ the right multiplication by $\omega^{-1}
\in\Lambda^2(\CC)$, i.e., the embedding
\begin{eqnarray*}
    m_{\omega^{-1}}:\CC& \longrightarrow & \Lambda^3(\CC)\, ,\\
    e & \longmapsto & e\wedge\omega^{-1}\, .
\end{eqnarray*}
Since $i_{\omega}(\omega^{-1})=1$, we have a commutative diagram: 
\begin{equation*}
\xymatrix{
\CC\ar@{=}[dr]\ar[r]^{m_{\omega^{-1}}}  & \Lambda^3(\CC)\ar[d]^{i_\omega}\\
& \CC
}
\end{equation*}
From now on, our main concern will be the 14--dimensional space of 3--forms
\begin{equation*}
    \Lambda_0^3(\CC):=\ker i_\omega\, ,
\end{equation*}
which, in the real--differentiable setting and up to the natural identification $\Lambda^3(\CC)\simeq\Lambda^3(\CC^*)$ via $\omega$,  is  the space of ``effective 3--forms'' mentioned  in Section \ref{sub3DSMAE} above.\par 
Using the above--defined   coordinates   $(x^{123},  X, Y, x^{456})$ on $\Lambda^3(\CC)$, where
\begin{equation}\label{eqDefCoordXY}
X = \begin{pmatrix}
x^{423} & x^{143} & x^{124} \\
x^{523} & x^{153} & x^{125} \\
x^{623} & x^{163} & x^{126}
\end{pmatrix}, \,\,
Y =  \begin{pmatrix}
x^{156} & x^{416} & x^{451} \\
x^{256} & x^{426} & x^{452} \\
x^{356} & x^{436} & x^{453}
\end{pmatrix},
\end{equation}
we see that  an element of $\Lambda^3(\CC)$ belongs to $\Lambda_0^3(\CC)$ if and only if equalities  $X = X^t$ and $Y = Y^t$ hold on that element; in a similar way, the symplectic form $\Omega$ descends to $\Lambda_0^3(\CC)$ (therefore, we keep using the same symbol).
\begin{definition}\label{defSpaceMAE}
 The projectivization $\p(\Lambda_0^3(\CC^*))$ of the 14--dimensional irreducible representation $\Lambda_0^3(\CC^*)$ of $\Sp(\CC)$  is the \emph{space parametrizing 3D symplectic Monge--Amp\`ere equations}.
\end{definition}

\begin{remark}\label{remParamSpaceMAE}
This sudden switch from $\CC$ to $\CC^*$ in  Definition \ref{defSpaceMAE} will simplify matching it with the previously given one, as we shall see below; from a representation--theoretical standpoint, however, there is no difference, since the $\Sp(\CC)$--modules $\Lambda_0^3(\CC)$ and $\Lambda_0^3(\CC^*)$ are isomorphic: a distinguished isomorphism, which in coordinates is given by  
\eqref{eq:C.C.Star}, descends from the symplectic form \eqref{eqOmegaBig}. In other words, both the spaces $\p(\Lambda_0^3(\CC))$ and $\p(\Lambda_0^3(\CC^*))$ can be taken as the parametrizing space of 3D symplectic Monge--Amp\`ere equations.
\end{remark}

To see how above Definition \ref{defSpaceMAE} matches with the definition given above of a Monge--Amp\`ere equation (see Definition \ref{defMAEq}) it is enough to take an element
\begin{equation*}
   [\eta] \in \p(\Lambda_0^3(\CC^*))
\end{equation*}
and associate with it the \emph{hyperplane section}
\begin{equation}\label{eqEeta}
    \E_\eta:=\p(\ker\eta)\cap\LL(3,\CC)\, .
\end{equation}
Then it is enough to recall that $\LL(3,\CC)$ is identified with $J^2_{\f}$ and that a symplectic $2\Nd$ order PDE is a trivial bundle over $J^1$, see Section \ref{sub2ordPDEhypSurf}; to see how $\E_\eta$ looks like in coordinates, let us write down $\eta$ as a linear combination
\begin{equation}\label{eqFormulaEtaConMatrici}
    \eta=\eta_{123}x^{123}+\tr(B_\eta X)+\tr(C_\eta Y)+ \eta_{456}x^{456}\, ,
\end{equation}
where $B_\eta$ and  $C_\eta$ are  $3\times 3$ matrices, cf. also \eqref{eq.MAE}: then, 
$\E_\eta=\{F_\eta=0\}$, with
\begin{equation}\label{eqFormulaEffeEtaConMatrici}
    F_\eta(u_{ij})=\eta_{123}+\tr(B_\eta U)+\tr(C_\eta U^\sharp)+\eta_{456}\det U\, ,
\end{equation}
and $U=\|u_{ij}\|$. Observe that $F_\eta$ depends upon $\eta$, whereas $\E_\eta$ only upon $[\eta]$.\par

\smallskip
\noindent
Below we give an example that will be useful later on.

\begin{example}\label{ex:gianni.co.co}
The element 
$$
e_{423}=e_4\wedge e_2\wedge e_3\in \Lambda_0^3(\CC)\, 
$$ 
can be regarded, via  
\eqref{eq:C.C.Star}, as the 3--form
$$
\eta=-x^{156}=-x^1\wedge x^5\wedge x^6\in \Lambda_0^3(\CC^*)\,
$$
on $\CC$, see also Remark \ref{remParamSpaceMAE}:  in the coordinate representation \eqref{eqFormulaEtaConMatrici}    all the coefficients of $\eta$   are  equal to zero, save for the $(1,1)$--entry of the matrix $C_\eta$, which is equal to $-1$: therefore, dropping the negligible sign, formula \eqref{eqFormulaEffeEtaConMatrici} leads to  the   Monge-Amp\`ere equation
\begin{equation}\label{eq:MAE.minore.1.1}
F_{\eta}(u_{ij})=u_{11}^\sharp=u_{22}u_{33}-u_{23}^2=0\, .
\end{equation}
It is worth stressing that transformation \eqref{eq:legendre.in.un.punto} sends $-x^{156}$ into $x^{423}$ and that  
the Monge-Amp\`ere equation associated with $\eta=x^{423}$ is
$$
F_{\eta}(u_{ij})=u_{11}=0\, ,
$$
which turns out to be equivalent to Monge-Amp\`ere equation \eqref{eq:MAE.minore.1.1}.
\end{example}
Having interpreted symplectic Monge--Amp\`ere equations as hyperplane sections of the Lagrangian Grassmannian $\LL(3,\CC)$, it is natural to expect that the cocharacteristic variety (see above Definition \ref{def.cochar.var}) be a geometric feature of the hyperplane section itself that can be computed by means of algebraic manipulations on $\eta$: this will be shown in the last Section \ref{secLinSec}, after a two--sections iatus. In the next Section \ref{SecNormFormQadr} we find a list of normal forms of quadratic forms on $\CC$ with respect to $\Sp(\CC)$ and then, in Section \ref{secMomentMap}, we show which of these forms come, via the Hitchin moment map, from the four isomorphism classes of Monge--Amp\`ere equations. At the very end of  Section \ref{secMomentMap} it is shown that the KLR contraction map, see \eqref{eq.JGG}, is equivalent to the Hitchin moment map (Theorem \ref{thFirstOriginalResult}).  

\section{Normal forms of quadratic forms on $\CC$ with respect to $\Sp(\CC)$}\label{SecNormFormQadr}

In this section we work out the classification of all $\Sp(\CC)$--orbits in $S^2(\CC^*)$. The main result, i.e., a list of normal forms, given in the basis \eqref{eqCoordinatesOnC} of $\CC$, can immediately be seen in Section \ref{subsCompClass} below: as the table shows,  there are three qualitative different types of quadratic forms, which we called \emph{nilpotent} (discussed in Section \ref{subsNilpotent}), \emph{semisimple} (see Section \ref{subsSemisimple}) and \emph{mixed} (see Section \ref{subsMixed}); basic facts about the root structure of $\Sp(6)$ are collected in Section \ref{subsrepTheoPrelims}

\subsection{The complete classification}\label{subsCompClass}
The representatives of all nonzero $\Sp(\CC)$--orbits in $S^2(\CC^*)$ are listed below: they are all non--equivalent, up to a sign change of the coefficients $(\lambda,\mu,\nu)\in\C^3$ and, if possible, a their permutation.\par

\begin{equation*}
    \begin{array}{l|l||l|l}
    \textrm{representative} & \textrm{coordinate expression} & \textrm{orbit type}& \textrm{orbit dimension}\\
    \hline
       q_{[6]}  & \epsilon^1e_2+\epsilon^2e_3+e_3^2 &\textrm{nilpotent}& 18  \\
       
               q^{(111)} & \lambda \epsilon^1e_1 +\mu \epsilon^2e_2+\nu \epsilon^3e_3 &\textrm{semisimple}& 18  \\

       q^{(21)}+\phi(X_{h_1-h_2})&\mu(\epsilon^1e_1+\epsilon^2e_2)+\nu \epsilon^3e_3+ \epsilon^1e_2&\textrm{mixed}& 18\\

        q^{(11)}+\phi(X_{-2h_1})&\mu \epsilon^2e_2   +\nu \epsilon^3e_3+(\epsilon^1)^2&\textrm{mixed}& 18\\
       
              q^{(2)}+ \phi(X_{h_2-h_3}+X_{-2h_1})&\nu (\epsilon^2e_2+\epsilon^3e_3)+   \epsilon^2e_3+    (\epsilon^1)^2&\textrm{mixed}& 18\\
       
           q^{(3)}+\phi(X_{h_1-h_2}+X_{h_2-h_3})&\nu (\epsilon^1e_1+\epsilon^2e_2+\epsilon^3e_3 )+\epsilon^1e_2 + \epsilon^2e_3&\textrm{mixed}& 18\\

       q^{(1)}+\phi(X_{h_1-h_2}-X_{2h_2})&\nu    \epsilon^3e_3+ \epsilon^1e_2+e_2^2&\textrm{mixed}& 18\\

       q_{[4,2]}  & \epsilon^1e_3+e_2^2+e_3^2 &\textrm{nilpotent}& 16  \\

      q^{(21)}  & \mu( \epsilon^1e_1+   \epsilon^2e_2)+\nu\epsilon^3e_3 &\textrm{semisimple}& 16  \\
          q^{(11)}  & \mu \epsilon^2e_2   +\nu \epsilon^3e_3 &\textrm{semisimple}& 16  \\

       q^{(2)}+ \phi(X_{h_2-h_3}) &\nu (\epsilon^2e_2+\epsilon^3e_3)+   \epsilon^2e_3  &\textrm{mixed}& 16\\
       
       q^{(2)}+   \phi(X_{-2h_1})&\nu (\epsilon^2e_2+\epsilon^3e_3)+     (\epsilon^1)^2&\textrm{mixed}& 16\\

        q^{(3)}+\phi(X_{h_1-h_2})&\nu (\epsilon^1e_1+\epsilon^2e_2+\epsilon^3e_3 )+\epsilon^1e_2&\textrm{mixed}& 16\\
       
          q^{(1)}+\phi(-\tfrac{1}{2}X_{h_1+h_2})&\nu    \epsilon^3e_3+ e_1e_2&\textrm{mixed}& 16\\

       q_{[4,1^2]}  & \epsilon^1e_2+e_2^2  &\textrm{nilpotent}& 14  \\
       q_{[3^2]}  & \epsilon^1e_3+e_2e_3 &\textrm{nilpotent}& 14  \\

             q^{(2)}   &  \nu (\epsilon^2e_2+\epsilon^3e_3) &\textrm{semisimple}& 14  \\
          q^{(1)}+\phi(-X_{2h_1})&\nu    \epsilon^3e_3+ e_1^2&\textrm{mixed}& 14\\
       
       q_{[2^3]}  &e_1^2+e_2^2+e_3^2 &\textrm{nilpotent}& 12  \\
       
                  q^{(3)}   & \nu ( \epsilon^1e_1+\epsilon^2e_2+\epsilon^3e_3 )&\textrm{semisimple}& 12  \\
       
       q_{[2^2,1^2]}  & e_1^2+e_2^2  &\textrm{nilpotent}& 10  \\
       
               q^{(1)}  & \nu \epsilon^3e_3 &\textrm{semisimple}& 10 \\

       q_{[2,1^4]}  & e_1^2  &\textrm{nilpotent}& 6  \\

    \end{array}
\end{equation*}
We stress that the basis elements appearing in the column labeled ``coordinate expression" are to be regarded as elements of $\CC^*$: a more homogeneous, though less  explanatory coordinate representation, involving only the $x^i$'s, may be obtained   by means of the  substitution \eqref{eq:C.C.Star}.

\subsection{Representation--theoretic preliminaries}\label{subsrepTheoPrelims}
With the standard choice of a Cartan subalgebra, 
\begin{equation}\label{eqCartSub}
    \gh:=\{\diag(\lambda,\mu,\nu,-\lambda,-\mu,-\nu)\mid \lambda,\mu,\nu\in\C\}\, ,
\end{equation}
%
%
\begin{figure}
    \centering
    \epsfig{file=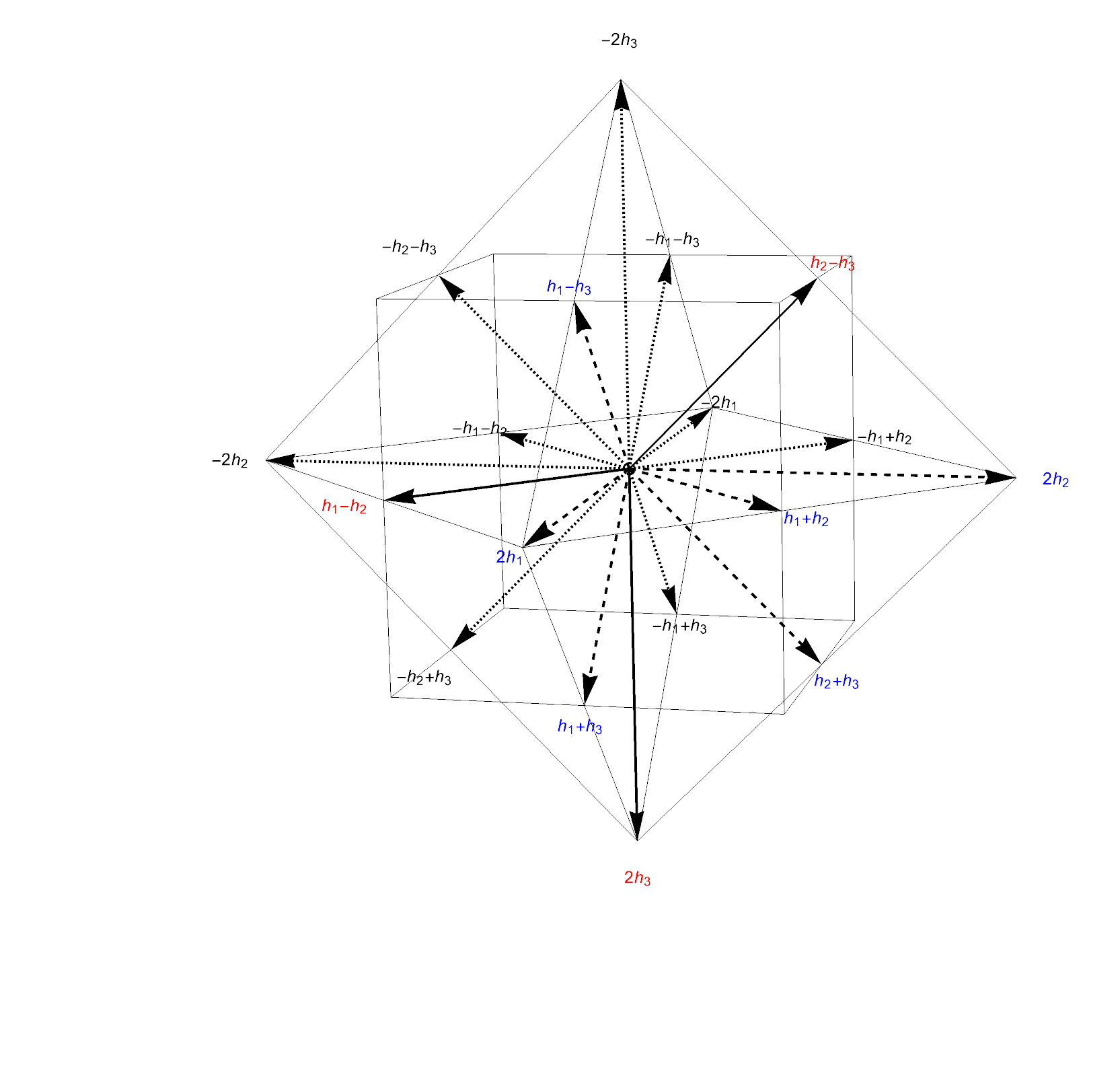,width=0.8\textwidth}
    \caption{The root system of type $\mathsf{C}_3$.}
    \label{fig:rootsystemC3}
\end{figure}
the 3--dimensional, eighteen--elements root system $\Phi$ of $\sp(\CC)$ is generated by the simple roots
\begin{equation*}
    \Delta:=\{h_1-h_2, h_2-h_3, 2h_3\}\, ,
\end{equation*}
where $2h_3$ is the long one: by $h_i\in\gh^*$ we mean the linear operator reading off the $i\Th$ entry of a diagonal matrix $H\in\gh$, i.e.,
\begin{equation*}
    h_1(H):=\lambda\, ,h_2(H):=\mu\, ,h_3(H):=\nu\, .
\end{equation*}

The corresponding set of positive roots will be then
\begin{equation*}
    \Phi^+=\Delta\cup\{ h_1-h_3, 2h_1, 2h_2, h_1+h_2, h_2+h_3, h_1+h_3  \}\, ,
\end{equation*}
where $\Phi=\Phi^+\cup \Phi^-$, with $\Phi^-=-\Phi^+$, see Figure \ref{fig:rootsystemC3}; hence, the root space decomposition of $\sp(\CC)$ is
\begin{equation*}
    \sp(\CC)=\gh\bigoplus_{\alpha\in\Phi}\g_\alpha\, .
\end{equation*}
\begin{remark}\label{remEIJ}
We let $E(i,j)\in\gl(\CC)$ be the $(i,j)$--elementary matrix and we pick a generator  $X_\alpha\in\sp(\CC)$   of the (one--dimensional) root space $\g_\alpha$, for all $\alpha\in\Phi$; it is then easy to see that: 
\begin{itemize}
    \item matrices $E(1,4), E(2,5)$ and $E(3,6)$ will be the root vectors $X_{2h_1}, X_{2h_2}$ and  $X_{2h_3}$ corresponding to the long positive roots $2h_1$, $2h_2$ and $2h_3$,  respectively;
    \item matrices $E(1, 2) - E(5,4), E(2, 3) - E(6, 5)$ and $E(1, 3) - E(6, 4)$ will be the root vectors $X_{h_1-h_2}, X_{h_2-h_3}$ and  $X_{h_1-h_3}$ corresponding to the short positive roots $h_1-h_2$, $h_2-h_3$ and $h_1-h_3$,  respectively;
    \item matrices $E(1, 5) + E(2, 4), E(2, 6) + E(3, 5)$ and $E(1, 6) + E(3, 4)$ will be the root vectors $X_{h_1+h_2}, X_{h_2+h_3}$ and  $X_{h_1+h_3}$ corresponding to the short positive roots $h_1+h_2$, $h_2+h_3$ and $h_1+h_3$,    respectively;
    \item matrix $X_\alpha:=X_{-\alpha}^t$ will be the root vector corresponding to the negative root  $\alpha\in\Phi^-$.
\end{itemize}
\end{remark}
\begin{remark}\label{remFundWeights}
Recalling that the fundamental weights of $\sp(\CC)$ are
\begin{equation*}
    h_1\, ,\quad h_1+h_2\, ,\quad h_1+h_2+h_3\, 
\end{equation*}
(see, e.g., \cite[2.2.13]{MR2532439}), we denote by $W_{(a,b,c)}$ the irreducible $\sp(\CC)$--representation whose highest weight is
\begin{equation*}
    (a+b+c)h_1+(b+c)h_2+ch_3\, ,
\end{equation*}
for any non--negative integers $a,b,c$. In particular, $W_{(1,0,0)}=\CC$ (resp., $W_{(2,0,0)}=\sp(\CC)$) and the highest weight vector is $e_1$ (resp., $X_{2h_1}$); accordingly, $W_{(0,0,1)}=\Lambda^3_0(\CC)$ with highest weight $h_1+h_2+h_3$ and highest weight vector $e_{123}$: the  irreducible representation $W_{(0,0,2)}$, whose highest weight vector $e_{123}^2$ is the square of $e_{123}$, turns out to be 84--dimensional and will play some role in the sequel.
\end{remark}

\subsection{Nilpotent orbits in $\sp(\CC)$}\label{subsNilpotent}
The Hasse diagram of nonzero nilpotent orbits in $\sp(\CC)$ is well-known (see, e.g., \cite[Example 6.2.6]{Collingwood2017}):
\begin{equation*}
    \xymatrix{
     & \OO_{[6]}\ar@{-}[d]&  & \dim=18\\
      & \OO_{[4,2]}\ar@{-}[dr]\ar@{-}[dl]& & \dim=16\\ 
      \OO_{[4,1^2]}\ar@{-}[dr]&&  \OO_{[3^2]}\ar@{-}[dl] & \dim=14\\
     & \OO_{[2^3]}\ar@{-}[d]&  & \dim=12\\
     & \OO_{[2^2,1^2]}\ar@{-}[d]& & \dim=10\\
     & \OO_{[2,1^4]} & & \dim=6
    }
\end{equation*}
Diagram above is a particular example of the Dynkin--Kostant classification \cite[Chapter 3]{Collingwood2017}, which ultimately associates with any (nonzero)  
\emph{nilpotent} orbit $\OO:=\Sp(\CC)\cdot X$,  the orbit of  a distinguished \emph{semisimple} element $H\in\gh$: these distinguished semisimple elements can be then labeled by \emph{weighted Dynkin diagrams}, the admissible weights being $0,1$ and $2$.\par 
Indeed, thanks to Jacobson--Morozov theorem, it is always possible to find, beside $H$, an   appropriate $ Y\in\sp(\CC)$, such that  the three--elements subset
\begin{equation*}
    \{X,H,Y\}\subset\sp(\CC)
\end{equation*}
constitutes a so--called \emph{standard $\sll_2$--triple}, $X$, $H$ and $Y$ being its \emph{nilpositive}, \emph{neutral} and \emph{nilnegative} element, respectively. Then a theorem by Kostant guarantees that any two   standard $\sll_2$--triple sharing the same nilpositive element are conjugate: this means that the orbit of the neutral element
\begin{equation*}
    \OO_H:=\Sp(\CC)\cdot X\, ,
\end{equation*}
is well defined and canonically asssociated with the nilpotent orbit $\OO$. The Dynkin--Kostant classification give us the (finite) list of all conjugacy classes of the so--obtained neutral elements: a representative for each class can be easily read off from the corresponding weighted Dynkin diagram.\par
Below we start from a weighted Dynkin diagram, we construct the corresponding neutral element $H$, and then we compute all the possible standard $\sll_2$--triples containing $H$: we finally pick a particular  nilpositive element $X$, such that the corresponding quadratic form on $\CC$ takes a particularly easy expression. 
  \subsubsection{The principal orbit $\OO_{\textrm{prin}}=\OO_{[6]}$} The weighted Dynkin diagram is
  \begin{equation*}
      \xymatrix{
      \overset{2}{\bullet} \ar@{-}[r]&\overset{2}{\bullet} &\overset{2}{\bullet}  \ar@{=}[l]|{\scalebox{1.6}{$<$}}
      }\, ,
  \end{equation*}
to which it corresponds the neutral element $H=\diag(5,3,1,-5,-3,-1)$: easy computations show that, if $\{X,H,Y\}$ is a  standard $\sll_2$--triple containing the aforementioned neutral element $H$, then its nilpositive and nilnegative elements must necessarily be 
\begin{equation}\label{eqStSL2TrpMaxOrb}
    X=\left(
\begin{array}{cccccc}
 0 & \text{$\alpha_1 $} & 0 & 0 & 0 & 0 \\
 0 & 0 & \text{$\alpha_2 $} & 0 & 0 & 0 \\
 0 & 0 & 0 & 0 & 0 & \text{$\alpha_3 $} \\
 0 & 0 & 0 & 0 & 0 & 0 \\
 0 & 0 & 0 & -\text{$\alpha_1 $} & 0 & 0 \\
 0 & 0 & 0 & 0 & -\text{$\alpha_2 $} & 0 \\
\end{array}
\right)\, ,\quad Y=\left(
\begin{array}{cccccc}
 0 & 0 & 0 & 0 & 0 & 0 \\
 \frac{5}{\text{$\alpha_1 $}} & 0 & 0 & 0 & 0 & 0 \\
 0 & \frac{8}{\text{$\alpha_2 $}} & 0 & 0 & 0 & 0 \\
 0 & 0 & 0 & 0 & -\frac{5}{\text{$\alpha_1 $}} & 0 \\
 0 & 0 & 0 & 0 & 0 & -\frac{8}{\text{$\alpha_2 $}} \\
 0 & 0 & \frac{9}{\text{$\alpha_3 $}} & 0 & 0 & 0 \\
\end{array}
\right)\, ,
\end{equation}
respectively. 
In particular, none of the three parameters $\alpha_1$, $\alpha_2$ and $\alpha_3$ can be zero; therefore we can set them to be $1$, $1$ and $-1$, respectively: this choice leads to a particularly simple expression for the associated quadratic form on $\CC$, viz.
\begin{equation*}
    q_{[6]}:=\phi(X)=\epsilon^1e_2+\epsilon^2e_3+e_3^2\, .
\end{equation*}
In the next cases we will refrain from showing the explicit form of the nilpositive and nilnegative elements forming a standard $\sll_2$--triple together with $H$; indeed, on a deeper level, what we have done above in \eqref{eqStSL2TrpMaxOrb}, can be reformulated as the definition of a $\sll_2$--gradation of $\sp(\CC)$, that is
\begin{equation}\label{eqSL2Grad}
    \sp(\CC)=\bigoplus_{i\in \Z}\g_i\, ,
\end{equation}
where $g_i:=\{Z\in \g\mid [H,Z]=i\}$: the chosen  nilpositive element $X$ is then an arbitrary nonzero element of the 3--dimensional constituent $\g_2$. In the  case under consideration  of the  principal orbit $\OO_{[6]}$, the nontrivial pieces of the $\sll_2$--gradation are:
\begin{equation*}
    \g_{-10}\oplus\g_{-8}\oplus\g_{-6}\oplus\g_{-4}\oplus \g_{-2}\oplus \gh\oplus \underset{\g_2}{\underbrace{ \Span{X_{h_1-h_2},X_{h_2-h_3},X_{2h_3}}}}
    \oplus \underset{\g_4}{\underbrace{ \Span{X_{h_1-h_3},X_{h_2+h_3}}}}
    \oplus \underset{\g_6}{\underbrace{ \Span{X_{h_1+h_3},X_{2h_2}}}}
    \oplus \underset{\g_8}{\underbrace{ \Span{X_{h_1+h_2}}}}
    \oplus \underset{\g_{10}}{\underbrace{ \Span{X_{2h_1}}}}\, .
\end{equation*}
In general, such a gradation is very useful in computing the Lie algebra $\g^X:=\stab_{\sp(\CC)}(X)$ of the subgroup $\Stab_{\Sp(\CC)}(X)$ of $\Sp(\CC)$ stabilizing $X$: indeed,   $\g^X$ is compatible with the  $\sll_2$--gradation:
\begin{equation*}
    \g^X=\bigoplus_{i\in \Z}\g_i^X\,,\quad  \g_i^X:=\g_i\cap \g^X\, ,
\end{equation*}
see \cite[(3.4.2)]{Collingwood2017}. Then one realizes that    solving the equation $[Z,X]=0$ in $\g^i$ is computationally easier that solving it in the whole of $\g$: in the case of $\OO_{[6]}$ we find out that only $\g^X_2$, $\g^X_6$ and $\g^X_{10}$ do not vanish and, in fact, are one--dimensional. It follows that $\dim\g^X=1+1+1=3$ and then $\dim \OO_{[6]}=\dim\sp(\CC)-\dim\g^X=21-3=18$; on this concern it is worth observing that the dimension of the principal nilpotent orbit of a semisimple Lie group $G$ is \emph{always} equal to $\dim(G)-\rk(G)$, see \cite[Lemma~4.1.3]{Collingwood2017}.\par
We will not insist here on the topology of nilpotent orbits of $G$, which   is highly nontrivial \cite[Chapter 6]{Collingwood2017}; we only recall that there exists a unique orbit of codimension $\rk(G)+2$, denoted by $\OO_{\textrm{subreg}}$ and called \emph{subregular}, which is open and dense in     $\overline{\OO_{\textrm{prin}}}\smallsetminus\OO_{\textrm{prin}}$, according to Steinber theorem  \cite[Theorem 4.2.1]{Collingwood2017}.
  \subsubsection{The subregular orbit $\OO_{\textrm{\normalfont subreg}}=\OO_{[4,2]}$}
  In the case $G=\Sp(\CC)$ we find   a $16$--dimensional orbit,   corresponding to the   weighted Dynkin diagram  
  \begin{equation*}
      \xymatrix{
      \overset{2}{\bullet} \ar@{-}[r]&\overset{0}{\bullet} &\overset{2}{\bullet}  \ar@{=}[l]|{\scalebox{1.6}{$<$}}
      }\, .
  \end{equation*}
The  $\sll_2$--gradation \eqref{eqSL2Grad} above now reads
\begin{equation*}
    \g_{-6}\oplus\g_{-4}\oplus\g_{-2} \oplus \underset{\g_0}{\underbrace{ \gh\oplus\Span{X_{h_2-h_3},X_{-h_2+h_3}}}}
    \oplus \underset{\g_2}{\underbrace{ \Span{X_{h_1-h_2},X_{h_1-h_3},X_{h_2+h_3},X_{2h_2},X_{2h_3}}}}
    \oplus \underset{\g_4}{\underbrace{ \Span{X_{h_1+h_2},X_{h_1+h_3}}}}
    \oplus \underset{\g_6}{\underbrace{ \Span{X_{2h_1}}}}
\end{equation*}
and $\OO_{\textrm{subreg}}$ will be generated by any nonzero element of $\g_2$: a particular choice of the nilpositive element $X\in \g_2$ gives us
\begin{equation*}
    q_{[4,2]}:=\epsilon^1e_3+e_2^2+e_3^2\, .
\end{equation*}
It is then easy to check that the only   nontrivial constituents of $\g^X$ are $\g^X_2$, $\g^X_4$ and $\g^X_6$,  and have dimensions 3, 1 and 1, respectively: indeed, $\dim \OO_{\textrm{subreg}}=21-(3+1+1)=16=\dim(\Sp(\CC))-\rk(\Sp(\CC))-2$. 

\subsubsection{The 14--dimensional strata: $\OO_{[4,1^2]}$ and $\OO_{[3,3]}$}
The singular locus   $\overline{\OO_{\textrm{subreg}}}\smallsetminus \OO_{\textrm{subreg}}$ consists of the closure of two 14--dimensional orbits:
\begin{equation*}
 \overline{\OO_{\textrm{subreg}}}\smallsetminus \OO_{\textrm{subreg}}=  \overline{\OO_{[4,1^2]}}\cup\overline{\OO_{[3,3]}}\, ,
\end{equation*}
one of which, namely $\OO_{[4,1^2]}$, will be of a paramount importance for us, as it contains the  cocharacteristic variety of a ``generic'' Monge--Amp\`ere equation.\par
The    weighted Dynkin diagrams of $\OO_{[4,1^2]}$ and $\OO_{[3,3]}$ are 
  \begin{equation*}
      \xymatrix{
      \overset{2}{\bullet} \ar@{-}[r]&\overset{1}{\bullet} &\overset{0}{\bullet}   \ar@{=}[l]|{\scalebox{1.6}{$<$}}
    }\, ,\quad 
      \xymatrix{
      \overset{0}{\bullet} \ar@{-}[r]&\overset{2}{\bullet} &\overset{0}{\bullet}  \ar@{=}[l]|{\scalebox{1.6}{$<$}}
      }\, ;
  \end{equation*}
 the $\sll_2$--gradations are:
\begin{equation*}
    \g_{-6}\oplus\g_{-4}\oplus\g_{-3} \oplus\g_{-2} \oplus\g_{-1} \oplus  \underset{\g_0}{\underbrace{ \gh\oplus\Span{X_{2h_3},X_{-2h_3}}}}
    \oplus \underset{\g_1}{\underbrace{ \Span{X_{h_2-h_3},X_{h_2+h_3}}}}
    \oplus \underset{\g_2}{\underbrace{ \Span{X_{h_1-h_2},X_{2h_2}}}}
    \oplus \underset{\g_3}{\underbrace{ \Span{X_{h_1-h_3},X_{h_1+h_3}}}}
    \oplus \underset{\g_4}{\underbrace{ \Span{X_{h_1+h_2}}}}
    \oplus \underset{\g_6}{\underbrace{ \Span{X_{2h_1}}}}
\end{equation*}
and
\begin{equation*}
    \g_{-4}\oplus\g_{-2} \oplus   \underset{\g_0}{\underbrace{ \gh\oplus\Span{X_{h_1-h_2},X_{2h_3},X_{-h_1+h_2},X_{-2h_3}}}}
    \oplus \underset{\g_2}{\underbrace{ \Span{X_{h_1-h_3},X_{h_2-h_3},X_{h_1+h_3},X_{h_2+h_3}}}}
    \oplus \underset{\g_4}{\underbrace{ \Span{X_{h_1+h_2},X_{2h_1},X_{2h_2}}}}\, ,
\end{equation*}
respectively. We finally obtain
\begin{equation*}
    q_{[4,1^2]}:=\epsilon_1e_2+e_2^2\, ,\quad q_{[3^2]}:=\epsilon^1e_3+e_2e_3\, .
\end{equation*}

\subsubsection{The 12--dimensional orbit  $\OO_{[2^3]}$}
The common singular locus 
\begin{equation*}
  (\overline{\OO_{[4,1^2]}}\smallsetminus \OO_{[4,1^2]}) \cap(\overline{\OO_{[3,3]}}\smallsetminus \OO_{[3,3]}) 
\end{equation*}
is the closure of the 12--dimensional orbit   $\OO_{[2^3]}$; its   weighted Dynkin diagram is 
  \begin{equation*}
      \xymatrix{
      \overset{0}{\bullet} \ar@{-}[r]&\overset{0}{\bullet} &\overset{2}{\bullet}  \ar@{=}[l]|{\scalebox{1.6}{$<$}}
      }\, ,
  \end{equation*}
 and the corresponding  $\sll_2$--gradation:
\begin{equation*}
    \g_{-2}  \oplus  \underset{\g_0}{\underbrace{ \gh\oplus\Span{X_{h_1-h_2},X_{h_1-h_3},X_{h_2-h_3},X_{-h_1+h_2},X_{-h_1+h_3},X_{-h_2+h_3}}}}
    \oplus \underset{\g_2}{\underbrace{ \Span{X_{h_1+h_2},X_{h_1+h_3},X_{h_2+h_3},X_{2h_1},X_{2h_2},X_{2h_3}}}}\, .
\end{equation*}
The quadratic form is
\begin{equation*}
    q_{[2^3]}:=e_1^2+e_2^2+e_3^2\,   .
\end{equation*}

\subsubsection{The 10--dimensional orbit  $\OO_{[2^2,1^2]}$}
The 10--dimensional stratum
\begin{equation*}
  \overline{\OO_{[2^2,1^2]}}=\overline{\OO_{[3^2]}}\smallsetminus \OO_{[3^2]} 
\end{equation*}
is last non--smooth one and,  in our further analysis, it 
will be parametrizing the Goursat--type symplectic  Monge--Amp\`ere equations, see Section \ref{sec:Goursat}. The weighted Dynkin diagram of $\OO_{[2^2,1^2]}$ is 
  \begin{equation*}
      \xymatrix{
      \overset{0}{\bullet} \ar@{-}[r]&\overset{1}{\bullet} &\overset{0}{\bullet}  \ar@{=}[l]|{\scalebox{1.6}{$<$}}
      }\, ,
  \end{equation*}
 and the corresponding  $\sll_2$--gradation:
\begin{equation*}
    \g_{-2}\oplus  {\g_{-1}}  \oplus  \underset{\g_0}{\underbrace{ \gh\oplus\Span{X_{h_1-h_2},X_{2h_3},X_{-h_1+h_2},X_{-2h_3}}}}
    \oplus \underset{\g_1}{\underbrace{ \Span{X_{h_1-h_3},X_{h_2-h_3},X_{h_1+h_3},X_{h_2+h_3}}}}
    \oplus \underset{\g_2}{\underbrace{ \Span{X_{h_1+h_2},X_{2h_1},X_{2h_2}}}}\, .
\end{equation*}
The quadratic form is
\begin{equation*}
    q_{[2^2,1^2]}:=e_1^2+e_2^2\,   .
\end{equation*}

\subsubsection{The cone over the adjoint variety   $\OO_{[2,1^4]}$}\label{subAdjVar}
The smallest nonzero nilpotent orbit  is 
\begin{equation*}
  \OO_{[2,1^4]}=\overline{\OO_{[2,1^4]}}=\overline{\OO_{[2^2,1^2]}}\smallsetminus \OO_{[2^2,1^2]}
\end{equation*}
and has dimension 6: unique, amongst all adjoint orbits, for being closed, its projectivization $\p( \OO_{[2,1^4]})$ is a remarkable 5--dimensional homogeneous contact manifold, known as the \emph{adjoint variety} of $\Sp(\CC)$; its weighted Dynkin diagram  is
  \begin{equation*}
      \xymatrix{
      \overset{1}{\bullet} \ar@{-}[r]&\overset{0}{\bullet} &\overset{0}{\bullet}  \ar@{=}[l]|{\scalebox{1.6}{$<$}}
      }\, .
  \end{equation*}
The  $\sll_2$--gradation of the subadjoint variety, namely
\begin{equation*}
    \g_{-2}\oplus  {\g_{-1}}  \oplus  \underset{\g_0}{\underbrace{ \gh\oplus\Span{X_{h_2-h_3},X_{h_2+h_3},X_{2h_2},X_{2h_3},X_{-h_2+h_3},X_{-h_2-h_3},X_{-2h_2},X_{-2h_3}}}}
    \oplus \underset{\g_1}{\underbrace{ \Span{X_{h_1-h_2},X_{h_1-h_3},X_{h_1+h_2},X_{h_1+h_3}}}}
    \oplus \underset{\g_2}{\underbrace{ \Span{X_{2h_1}}}}\, ,
\end{equation*}
is the only one, for which $\dim\g_2=1$; moreover, the $\g_2$--valued two--form
\begin{eqnarray*}
\Lambda^2(\g^*_1)&\longrightarrow & g_2\, ,\\
(Z,W) &\longmapsto & [Z,W]
\end{eqnarray*}
is non--degenerate and this motivates calling such a gradation a \emph{contact grading}, see \cite{MR3904635} on this concern.\par
The following   quadratic form ends our list of normal forms of quadratic forms on $\CC$ corresponding to (nonzero) nilpotent elements of $\sp(\CC)$:
\begin{equation*}
    q_{[2,1^4]}:=e_1^2\, .
\end{equation*}

 
\subsection{Semisimple orbits in $\sp(\CC)$}\label{subsSemisimple}
The theory of semisimple orbits in semisimple Lie algebras is considerably simpler than that of nilpotent ones: their classification boils down to taking the quotient of a Cartan subalgebra $\gh$ with respect to the natural action of the Weyl group.\par
In the present case the elements of $\gh$ are unambiguously labeled by a vector $(\lambda,\mu,\nu)\in\C^3$, cf. \eqref{eqCartSub}, and the Weyl group is 
\begin{equation*}
    W=S_3\ltimes \Z_2^3\, ,
\end{equation*}
acting naturally on $\C^3$: this means that we identify 
\begin{equation*}
    (\lambda,\mu,\nu)\sim (\epsilon_1\sigma(\lambda),\epsilon_2\sigma(\mu),\epsilon_3\sigma(\nu))\, ,
\end{equation*}
where $\epsilon_i=\pm 1$, for $i=1,2,3$, and $\sigma$ is a permutation of the set $\{\lambda, \mu,\nu\}$, and the corresponding equivalence class unambiguously identifies a semisimple orbit.\par
In what follows, symbol $H$ denotes an element of $\gh$ labeled by the equivalence class of $(\lambda,\mu,\nu)$, while $\g^H$ denotes its stabilizing subalgebra, i.e., the Lie algebra of the subgroup
\begin{equation*}
    G^H:=\{x\in\Sp(\CC)\mid x H x^{-1}=H\}= \{x\in\Sp(\CC)\mid x H =H x\}\, 
\end{equation*}
of elements of $\Sp(\CC)$ commuting with $H$. Letting %
\begin{equation*}
    \CC=(\ker H)\oplus\bigoplus_{\underset{\alpha\neq 0}{\alpha=\lambda,\mu,\nu}} (V_\alpha\oplus V_\alpha^*)\, ,
\end{equation*}
where $V_\alpha$ denotes the eigenspace relative to the eigenvector $\alpha\in\C\smallsetminus\{0\}$, 
we immediately see that
\begin{equation}\label{eqDecompGH}
    G^H=\Sp(\ker H)\times\prod_{\underset{\alpha\neq 0}{\alpha=\lambda,\mu,\nu}}\GL(V_\alpha)\, .
\end{equation}
Therefore, there are only six possible isomorphism types of stabilizers, collected below:
\begin{equation*}
    \begin{array}{l||l|l|l||l|l||l}
    \textrm{type of }\OO&
    \dim V_\lambda &
    \dim V_\mu& \dim V_\nu&\dim\ker H &\textrm{stabilizer} & \dim\OO\\
    \hline
    (111) & 1 &1&1&0 &\GL(V_\lambda)\times\GL(V_\mu)\times\GL(V_\nu)=\GL_1^3 & 18\\
    (21) & 0&2&1&0& \GL(V_\mu)\times\GL(V_\nu)=\GL_2\times\GL_1 & 16\\
    (11) & 0&1&1&2 &\Sp(\ker H)\times \GL(V_\mu)\times\GL(V_\nu)=\SL_2\times\GL_1^2 & 16\\
    (2) & 0&0&2&2&\Sp(\ker H)\times \GL(V_\nu)=\SL_2\times\GL_2 &14\\
    (3) & 0&0&3&0&  \GL(V_\nu)=\GL_3 &12\\
    (1) & 0&0&1&4& \Sp(\ker H)\times \GL(V_\nu)=\Sp_4\times\GL_1 &10
    \end{array}
\end{equation*}
Observe that, by the \emph{type} of the orbit $\OO$ passing through $H$ we meant the collection of   the nonzero  dimensions of $V_\lambda$, $V_\mu$ and $V_\nu$.

\subsubsection{18--dimensional semisimple orbits}   The type of $H$ is $(111)$ if   $\lambda,\mu,\nu$ are different from zero and different form each other; this is the ``non--degenerate'' case, that is, when $H$ commutes only with other elements of $\gh$: this means that 
\begin{equation*}
    G^H=\GL_1^3\, ,\quad \g^H=\gh\,  ,
\end{equation*}
and then $\dim (\Sp(\CC)\cdot H)=21-3=18$.\par
The quadratic form on $\CC$ corresponding to $H$ will be denoted by
\begin{equation*}
   q^{(111)}:=  \lambda  \epsilon^1e_1 + \mu  \epsilon^2e_2+ \nu  \epsilon^3e_3\, .
\end{equation*}

\subsubsection{16--dimensional semisimple orbits} The type of $H$ is $(21)$ if  $\lambda=\pm \mu$, and $\nu$ are different from zero and different form $\pm$ each other: in this case
\begin{equation}\label{eqStabH21}
   G^H=\GL_2\times\GL_1\, ,\quad \g^H=\gh\oplus\Span{X_{h_1-h_2},X_{-h_1+h_2}}\, ,
\end{equation}
whence $\dim\g^H=5$ and then the orbit is 16--dimensional.\par
The type of $H$ is $(11)$ if $\lambda=0$ and that $\mu$ and $\nu$ are different from zero and different form $\pm$ each other: we obtain
\begin{equation}\label{eqStabH11}
    G^H=\Sp_2\times\GL_1^2=\SL_2\times\GL_1^2\, ,\quad \g^H=\gh\oplus\Span{X_{2h_1},X_{-2h_1}}\, ,
\end{equation}
that is another 16--dimensional orbit. The corresponding quadratic forms are
\begin{equation*}
    q^{(21)}:= \mu(  \epsilon^1e_1+  \epsilon^2e_2)+ \nu \epsilon^3e_3\, ,\quad q^{(11)}:=  \mu  \epsilon^2e_2+ \nu \epsilon^3e_3\, ,
\end{equation*}
respectively.

\subsubsection{14--dimensional semisimple orbits} The type of $H$ is $(2)$ if $\lambda=0$ and   $\mu=\pm \nu$ is different than zero: we obtain   7--dimensional stabilizers
\begin{equation}\label{eqStabH2}
    G^H=\Sp_2\times\GL_2=\SL_2\times\GL_2\, ,\quad \g^H=\gh\oplus\Span{X_{h_2-h_3},X_{2h_1},X_{-h_2+h_3},X_{-2h_1}}\, ,
\end{equation}
whence a 14--dimensional orbit and the corresponding  quadratic form is
\begin{equation*}
   q^{(2)}:=   \nu  (\epsilon^2e_2    + \epsilon^3e_3)\, .
\end{equation*}

\subsubsection{12--dimensional semisimple orbits} The type of $H$ is $(3)$ if $\lambda=\pm\mu=\pm\nu$ is different from zero; in this case the stabilizers are  9--dimensional: 
\begin{equation}\label{eqStabH3}
   G^H=\GL_3\, ,\quad \g^H=\gh\oplus\Span{X_{h_1-h_2},X_{h_1-h_3},X_{h_2-h_3},X_{-h_1+h_2},X_{-h_1+h_3},X_{-h_2+h_3}}\, ,
\end{equation}
whence a 12--dimensional orbit and the corresponding quadratic form is
\begin{equation*}
   q^{(3)}:=  \nu  (\epsilon^1e_1+\epsilon^2e_2+\epsilon^3e_3)\, .
\end{equation*}

\subsubsection{10--dimensional semisimple orbits} The smallest nontrivial semisimple orbits have dimension 10 and correspond to an $H$ of type $(1)$, i.e., with $\lambda=\mu=0$ and $\nu\neq 0$: the stabilizers 
\begin{equation}\label{eqStabH1}
G^H=\Sp_4\times\GL_1 \, ,\quad
    \g^H=\gh\oplus\Span{X_{h_1-h_2},X_{h_1+h_2},X_{2h_1},X_{2h_2},X_{-h_1+h_2},X_{-h_1-h_2},X_{-2h_1},X_{-2h_2}}\, 
\end{equation}
are  11--dimensional, 
whence a 10--dimensional orbit and the corresponding quadratic form is
\begin{equation*}
  q^{(1)}:=   \nu \epsilon^3e_3\, .
\end{equation*}


\subsection{Mixed orbits in $\sp(\CC)$}\label{subsMixed}
We pass now to the generic case and we study orbits passing through (nonzero) elements of the form
\begin{equation*}
    Z=H+X\, ,
\end{equation*}
where $H\in\g$ is semisimple and $X\in\NN$, where $\NN(\g)\subset\g$ denotes the \emph{nilpotent cone}, that is, the subset of nilpotent elements of the Lie algebra $\g$, and, moreover, $[H,X]=0$.\par
The strategy consists in bringing $H$ to one of the normal forms above (corresponding to the quadratic forms labeled $q^{(111)},q^{(21)},q^{(11)},q^{(2)},q^{(3)},q^{(1)}$): then, thanks to the Jordan's decomposition theorem, the nilpotent part $X$ of $Z$ must be coming from the set
\begin{equation*}
    \g^H\cap\NN =\stab_{\sp(\CC)}(H)\cap\NN=\{X\in\sp(\CC)\mid [H, X]=0\}\cap\NN\, ,
\end{equation*}
whence the problem is reduced to studying the $G^H$--orbits in $\g^H\cap\NN$, where $G^H$ has the structure shown in \eqref{eqDecompGH} and acts naturally and componentwise on 
\begin{equation*}
    g^H=\sp(\ker H)\oplus\bigoplus_{\underset{\alpha\neq 0}{\alpha=\lambda,\mu,\nu}}\gl(V_\alpha)\, .
\end{equation*}
Such a particular structure of $g^H$ is mirrored by an analogous  decomposition of the nilpotent cone:
\begin{equation*}
    \g^H\cap\NN=\NN(\sp(\ker H))+\sum_{\underset{\alpha\neq 0}{\alpha=\lambda,\mu,\nu}}\NN(\gl(V_\alpha))\, ,
\end{equation*}
whence it sufficies to classify the orbit of 
\begin{equation*}
    G^H_{\textrm{ss}}:=\Sp(\ker H)\times\prod_{\underset{\alpha\neq 0}{\alpha=\lambda,\mu,\nu}}\SL(V_\alpha)\, 
\end{equation*}
acting naturally and component--wise on
\begin{equation*}
    \g^H\cap\NN=\NN(\sp(\ker H))+\sum_{\underset{\alpha\neq 0}{\alpha=\lambda,\mu,\nu}}\NN(\sll(V_\alpha))\, .
\end{equation*}
\subsubsection{Mixed orbits with semisimple part of type $ (111)$}
In this case, $\g^H=\gh$, whence $\g^H\cap\NN=0$: this means that the only mixed orbit with semisimple part of type $q^{(111)}$ is the orbit of $q^{(111)}$ itself.
\subsubsection{Mixed orbits with semisimple part of type $(21)$} In this case, $H$ corresponds to $(\mu,\mu,\nu)$: the space $V_\mu$ has dimension two and 
\begin{equation*}
    G^H_{\textrm{ss}}= \SL(V_\mu)\simeq \SL_2
\end{equation*}
acts naturally on 
\begin{equation*}
    \g^H\cap\NN= \NN(\sll(V_\mu))\subseteq\Sp(\CC)\cdot\Span{X_{h_1-h_2},X_{-h_1+h_2}}\, ,
\end{equation*}
cf. \eqref{eqStabH21}. Since in $\sll(V_\mu)\simeq \sll_2$ there is only one nontrivial  nilpotent orbit, namely the one passing through $ X_{h_1-h_2}$, we see that, besides $q^{(21)}$ itself, the only other orbit with semisimple  part $q^{(21)}$ is
\begin{equation*}
    q^{(21)}+\phi(X_{h_1-h_2})=\mu(\epsilon^1e_1+\epsilon^2e_2)+\nu \epsilon^3e_3+ \epsilon^1e_2\, .
\end{equation*}
The aforementioned nilpotent orbit in $\sll_2$ has dimension 2, whereas the orbit in $\sp(\CC)$ of the semisimple element $q^{(21)}$ has dimension 16: it follows that the mixed orbit has dimension 18.

\subsubsection{Mixed orbits with semisimple part of type $ (11)$} In this case, $H$ corresponds to $(0,\mu,\nu)$: the spaces $V_\mu$ and $V_\nu$ have both dimension one and 
\begin{equation*}
    G^H_{\textrm{ss}}= \Sp(\ker H)=\SL(\ker H)\simeq\SL_2
\end{equation*}
acts naturally on 
\begin{equation*}
    \g^H\cap\NN= \NN(\sll(\ker H))\subseteq\Sp(\CC)\cdot\Span{X_{2h_1},X_{-2h_1}}\, ,
\end{equation*}
cf. \eqref{eqStabH11}. 
Again, in $\sll(\ker H)\simeq\sll_2$ there is only one nontrivial  nilpotent orbit: the one passing, e.g.,  through $ X_{-2h_1}$, whence, besides $q^{(11)}$ itself, the only other orbit with semisimple part $q^{(11)}$ is
\begin{equation*}
    q^{(11)}+\phi(X_{-2h_1})=\mu \epsilon^2e_2   +\nu \epsilon^3e_3+(\epsilon^1)^2\, .
\end{equation*}
Recalling that also the semisimple element $q^{(11)}$ of $\sp(\CC)$ has a 16--dimensional orbit, we get another mixed orbit of dimension 18.

\subsubsection{Mixed orbits with semisimple part of type $ (2)$}  In this case, $H$ corresponds to $(0,\nu,\nu)$: the space $ V_\nu$ has   dimension two and 
\begin{equation*}
    G^H_{\textrm{ss}}= \Sp(\ker H)\times\SL(V_\nu)=\SL(\ker H)\times\SL(V_\nu)\simeq \SL_2\times \SL_2
\end{equation*}
acts naturally and componentwise on 
\begin{equation*}
    \g^H\cap\NN= \NN(\sll(\ker H))+\NN(\sll(V_\nu))\subseteq\Sp(\CC)\cdot(\Span{X_{h_2-h_3},X_{-h_2+h_3}}+\Span{X_{2h_1},X_{-2h_1}})\, ,
\end{equation*}
cf. \eqref{eqStabH2}. 
Now it is convenient to introduce two coefficients $\delta_1$ and $\delta_2$, that can be either 0 or 1: all mixed orbits with semisimple part $q^{(2)}$ will then pass trough one and only one of the four following elements:
\begin{equation*}
    q^{(2)}+\delta_1\phi(X_{h_2-h_3})+\delta_2\phi(X_{-2h_1})=\nu (\epsilon^2e_2+\epsilon^3e_3)+ \delta_1 \epsilon^2e_3+  \delta_2 (\epsilon^1)^2\, .
\end{equation*}
The semisimple orbit in $\sp(\CC)$ passing through $ q^{(2)}$ has dimension 14; since both $\ker H$ and $\sll(V_\nu)$ are isomorphic to $\sll_2$, we easily get the following formula for the dimension of the mixed orbits:
\begin{equation*}
   \dim= 14+2(\delta_1+\delta_2)\, .
\end{equation*}

\subsubsection{Mixed orbits with semisimple part of type $(3)$} In this case, $H$ corresponds to $(\nu,\nu,\nu)$: the space   $V_\nu$ has   dimension three and 
\begin{equation*}
    G^H_{\textrm{ss}}=  \SL(V_\nu)\simeq\SL_3
\end{equation*}
acts naturally on 
\begin{equation*}
    \g^H\cap\NN= \NN(\sll(V_\nu))\subseteq\Sp(\CC)\cdot\Span{X_{h_1-h_2},X_{h_1-h_3},X_{h_2-h_3},X_{-h_1+h_2},X_{-h_1+h_3},X_{-h_2+h_3}}\, ,
\end{equation*}
cf. \eqref{eqStabH3}. 
Now, in $\sll(V_\nu)\simeq\sll_3$ there are two nontrivial  nilpotent orbits:  one passing through $ X_{h_1-h_2}$ and the other passing through $X_{h_1-h_2}+X_{h_2-h_3}$, whence, besides $q^{(3)}$ itself, the only other orbits with semisimple part $q^{(3)}$ are
\begin{eqnarray}
    q^{(3)}+\phi(X_{h_1-h_2})&=&\nu (\epsilon^1e_1+\epsilon^2e_2+\epsilon^3e_3 )+\epsilon^1e_2\, ,\label{eqMixQ3six}\\
    q^{(3)}+\phi(X_{h_1-h_2}+X_{h_2-h_3})&=&\nu (\epsilon^1e_1+\epsilon^2e_2+\epsilon^3e_3 )+\epsilon^1e_2 + \epsilon^2e_3\, .\label{eqMixQ3four}
\end{eqnarray}
The aforementioned nilpotent orbits of $\sll_3$ correspond, respectively, to the submaximal (which is the same as the minimal) and to the maximal   nilpotent orbit and, therefore, have dimension 4 and 6, respectively; recalling that in $\sp(\CC)$, the semisimple element $q^{(3)}$ generates a 12--dimensional orbit, the corresponding mixed orbits \eqref{eqMixQ3six} and \eqref{eqMixQ3four} have dimension 16 and 18, respectively.

\subsubsection{Mixed orbits with semisimple part of type $(1)$} In this case, $H$ corresponds to $(0,0,\nu)$: the space   $V_\nu$ has   dimension one, whereas $\ker H$ is four--dimensional; therefore
\begin{equation*}
    G^H_{\textrm{ss}}=  \Sp(\ker H)\simeq\Sp_4
\end{equation*}
acts naturally on 
\begin{equation*}
    \g^H\cap\NN= \NN(\sp(\ker H))\subseteq \Sp(\CC)\cdot\Span{X_{h_1-h_2},X_{h_1+h_2},X_{2h_1},X_{2h_2},X_{-h_1+h_2},X_{-h_1-h_2},X_{-2h_1},X_{-2h_2}}\, ,
\end{equation*}
cf. \eqref{eqStabH1}. 
Now, in $\sp(\ker H)\simeq\sp_4$ there are three nontrivial  nilpotent orbits, passing through, e.g.,  $  X_{h_1-h_2}-X_{2h_2}$, $-\tfrac{1}{2} X_{h_1+h_2}$ and $-X_{2h_1}$, whence, besides $q^{(1)}$ itself, the only other orbits with semisimple part $q^{(1)}$ are
\begin{eqnarray}
    q^{(1)}+\phi(X_{h_1-h_2}-X_{2h_2})&=&\nu    \epsilon^3e_3+ \epsilon^1e_2+e_2^2\, ,\label{eqMixQ1dim18}\\
    q^{(1)}+\phi(-\tfrac{1}{2}X_{h_1+h_2})&=&\nu    \epsilon^3e_3+ e_1e_2\, ,\label{eqMixQ1dim16}\\   q^{(1)}+\phi(-X_{2h_1})&=&\nu    \epsilon^3e_3+ e_1^2\, ,\label{eqMixQ1dim14}
\end{eqnarray}
respectively.\par
The aforementioned nilpotent orbits of $\sp_4$ correspond, respectively, to the  maximal, the submaximal, and to   the minimal   nilpotent orbit and, therefore, have dimension 8, 6 and 4, respectively; since the semisimple element $q^{(1)}$ generates a 10--dimensional orbit in  $\sp(\CC)$, the corresponding mixed orbits \eqref{eqMixQ1dim18}, \eqref{eqMixQ1dim16} and \eqref{eqMixQ1dim14} have dimension 18, 16 and 14, respectively.

\begin{remark}
For any semisimple element $H\in\sp(\CC)$ there is a nilpotent element $X$, commuting with $H$, such that the mixed orbit passing throug $Z=H+X$ has dimension 18; therefore, such a dimension is the highest attainable by all the adjoint orbits in $\sp(\CC)\simeq\sp_6$.
\end{remark}

\section{The moment map $\varpi$ on the space of Monge--Amp\`ere equations}\label{secMomentMap}

We resume here the study of the   space $\p(\Lambda_0^3(\CC))$, whose dual $\p(\Lambda_0^3(\CC^*))$ parametrizes Monge--Amp\`ere equations, that  we have introduced earlier in Section \ref{subsPluckerSpace}. In Section \ref{subsMomMap} we sketch the construction of a  quadratic $\Sp(\CC)$--equivariant moment map $\varpi$ on   the symplectic space $\Lambda_0^3(\CC)$, that has been observed in the first place by N.~Hitchin in \cite{Hitchin3Forms}, in a similar fashion as we did for the canonical identification discussed in Section \ref{subsIdentSpQuadraticForms} above. In the next Section \ref{subsVarpiAndOrbits} we compute the image via $\varpi$ of each of the four orbits of $\p(\Lambda_0^3(\CC))$, as well as all the isomorphism types  of the fibers of $\varpi$. In the last Section \ref{subsRussianContractio} we show that the Hitchin moment map is equivalent to the KLR contraction, cf. \eqref{eq.JGG}, up to the natural identification $\CC\simeq\CC^*$ via the symplectic form.

\subsection{The Hitchin moment map on $\Lambda_0^3(\CC)$}\label{subsMomMap}
By Leibniz's rule, the embedding $j$ given in \eqref{eqEmbedJay} extends to the algebra of polynomial vector fields on $\Lambda_0^3(\CC)$ and it is not hard to see that these are Hamiltonian with respect to $\Omega$, cf. \eqref{eqOmegaBig}: therefore, it must exist, similarly  as before,  a \emph{quadratic} moment map $\widetilde{\mu}$ closing the diagram
\begin{equation}\label{eqDiagS2Lamdba3C}
     \xymatrix{
    S^2(\Lambda_0^3(\CC))\ar[r]^{\mathrm{pr}} & \sp(\CC)^*\\
    \Lambda_0^3(\CC)\ar[ur]_{\widetilde{\mu}}\ar[u]^{v_2} &
    }
\end{equation}
of $\Sp(\CC)$--equivariant polynomial maps. It is worth stressing that only the upper arrow is linear and it is in fact the projection of the 105--dimensional representation $S^2(\Lambda_0^3(\CC))$ onto its unique 21--dimensional irreducible constituent:
\begin{equation*}
S^2(\Lambda_0^3(\CC))= W_{(2,0,0)}\oplus W_{(0,0,2)} =\sp(\CC)^*\oplus W_{(0,0,2)}\, .
\end{equation*}

By identifying $\sp(\CC)^*$ with $S^2(\CC)$, cf. \eqref{eqIdS2CSpC}, we finally obtain the $\Sp(\CC)$--equivariant quadratic map
\begin{equation}\label{eqPhiLambda3CS2C}
    \varphi:=\phi^{\ast\, -1}\circ \widetilde{\mu}: \Lambda_0^3(\CC) \longrightarrow S^2(\CC)
\end{equation}
the whole paper hinges around.\par
Observe that the  map $\varphi$, which is clearly non--surjective, fails also to be injective: in particular, the smallest nonzero $\Sp(\CC)$--orbit in $\Lambda_0^3(\CC)$, that is the cone over the Lagrangian Grassmanian $\LG(3,\CC)$, is mapped to zero; therefore, by projectivizing the above diagram \eqref{eqDiagS2Lamdba3C} and taking into account \eqref{eqPhiLambda3CS2C}, we obtain a commuting diagram of rational maps
 \begin{equation}\label{eqDiagVarpi}
\xymatrix{
\p(\Lambda^3_0(\CC))\ar[r]^{v_2}\ar@{.>}[dr]_{\varpi}& \p(S^2(\Lambda_0^3(\CC)))\ar@{.>}[d]^{[\mathrm{pr}]}\\
&  \p(S^2(\CC)) \,. 
}
\end{equation}
\begin{definition}
The rational map $\varpi$, defined as the projectivization of the composition of the Veronese embedding and the projection onto the first factor or, equivalently, as the projectivization of the moment map on the  space of symplectic 3D Monge--Amp\`ere equations, will be called simply the \emph{moment map}.
\end{definition}


\subsection{Images and fibers of the moment map across the four orbits of $\p(\Lambda^3_0(\CC))$}\label{subsVarpiAndOrbits}
An obvious and yet crucial remark is that the above defined map $\varpi$ is $\Sp(\CC)$--equivariant, as it is a composition of two maps having such property. It is also well known that the natural action of $\Sp(\CC)$ on $\p(\Lambda^3_0(\CC))$ has   precisely four orbits, that  are linearly arranged  with respect to the closure order.\par
Therefore, in order to fully describe $\varpi$, it will suffice to compute the image of chosen representatives of the four orbits of $\p(\Lambda^3_0(\CC))$, and $\varpi(\p(\Lambda^3_0(\CC)))$ is going to be a sum of orbits of $\p(S^2(\CC))$: our main tool to spell out the structure of such sums will be the study of weights.  
%
Concerning the study of fibers, it is worth noticing that, if $[q]=\varpi([\eta])$, then $\varpi^{-1}([q])$ is naturally acted upon by the stabilizer $\Stab_{\Sp(\CC)}([q])$ in a transitive way; in other words, $\varpi^{-1}([q])$ has the structure of a $\Stab_{\Sp(\CC)}([q])$--homogeneous space.

\subsubsection{The orbit structure of $\p(\Lambda^3_0(\CC))$}

The details of the four orbits of $\p(\Lambda^3_0(\CC))$, that we are going to need, are summarized below (recall that symbol $X^\vee$ denotes the projective dual of the variety $X$):
\begin{equation*}
\begin{array}{l|l|l||l}
   \textrm{acronym}  & \textrm{dimension} & \textrm{structure} & \textrm{representative} \\
   \hline
    \textrm{O (as in ``open'')} & 13 &  \p(\Lambda^3_0(\CC))\smallsetminus \LL(3, \CC)^\vee & [e_{123} + e_{456}]\\
    \textrm{L (as in ``linearizable'')} & 12 & \LL(3, \CC)^\vee\smallsetminus\Sing(\LL(3, \CC)^\vee) &  [e_{423} + e_{126} + e_{153} + e_{123}]\\
     \textrm{G (as in ``Goursat'')} & 9 & \Sing(\LL(3, \CC)^\vee)\smallsetminus \LL(3, \CC)\simeq \Gr(3,\CC)/\Z_2&  [e_{163} + e_{125}]\\
     \textrm{P (as in ``parabolic'')} & 6 &   \LL(3, \CC) &  [e_{123}]\\
\end{array}
\end{equation*}
The smallest one (P) is 6--dimensional, smooth and closed: it is  the Lagrangian Grassmanian $\LG(3,\CC)$; it is natural to pick as a representative (the projective class of) the highest weight vector  $[e_{123}]$ of the irreducible representation $W_{(0,0,1)}=\Lambda^3_0(\CC)$, see Remark \ref{remFundWeights}.\par
The projective dual of $\LG(3,\CC)$ is the quartic hypersurface $ \LG(3,\CC)^\vee=\{f=0\}$, with  
\begin{equation}\label{eqEQdefLGDual}
f = (x^{123}x^{456} - \tr(XY))^2 +4x^{123} \det(Y) +4x^{456}\det(X) - 4\sum_{1 \le i,j \le 3} \det \|X_{i,j}\| \det \|Y_{i,j}\|\, ,
\end{equation}
where $X$ and $Y$ are  defined by \eqref{eqDefCoordXY} and the symbol  $A_{i,j}$ denotes the $2\times 2$ block that is complementary to the $i\Th$ row and the $j\Th$ column of the $3\times 3$ matrix $A$.\par 
Its smooth locus $ \LG(3,\CC)^\vee \smallsetminus \Sing(\LG(3,\CC)^\vee)$ is the second biggest orbit (L) and as an its representative we take the projective class $[e_{423} + e_{126} + e_{153} + e_{123}]$. 
%
The biggest orbit is the open one (O), that is $\p(\Lambda^3_0(\CC))\smallsetminus \LG(3,\CC)^\vee$, which is the orbit passing through, e.g., $[e_{123} + e_{456}]$. 
Finally, the singular locus $\Sing(\LG(3,\CC)^\vee)$, that is the variety cut out by the Jacobian ideal of \eqref{eqEQdefLGDual}, is 9--dimensional and $\Sing(\LG(3,\CC)^\vee) \smallsetminus \LG(3,\CC)$ is the orbit (G) represented, e.g., by $[e_{163} + e_{125}]$. 


\subsubsection{(P) The singular locus}
\begin{proposition}\label{propSingLoc}
	The singular locus of $\varpi$ is precisely $\LG(3,\CC)$.
\end{proposition}
\begin{proof}
	 Recall that we have chosen, as a  representative for the smallest orbit, the element $[e_{123}]$: the vector $e_{123}$ has weight $h_1 + h_2 +h_3$ and the image of $[e_{123}]$ under the Veronese embedding corresponds to a vector of weight $2h_1 +2h_2 + 2h_3$; the latter  is nothing but the  highest weight vector of the 84--dimensional constitutent  $W_{(0,0,2)}$ that is annihilated by the affine projection $\mathrm{pr}$ in \eqref{eqDiagS2Lamdba3C}. This shows that the map $\varpi$ is not defined on $\LG(3,6)$ and since in the next propositions we will find the images of remaining  orbits, it suffices to prove the claim.
\end{proof}

\subsubsection{(G) Image of the 9--dimensional orbit}
It will be useful to recall that the minimal projective embedding  $\LG(2,V^\prime)\subset\p(\Lambda_0^2(V^\prime))$, where $V^\prime$ is a 4--dimensional symplectic space, has only two orbits: the 3--dimensional (smooth, closed) Lagrangian Grassmanian  $\LG(2,V^\prime)=\LG(2,4)$ itself and its (open) complement in $\p(\Lambda_0^2(V^\prime))=\p^4$; one way to see this is the Dynkin diagram identification $\mathsf{C}_2\equiv\mathsf{B}_2$:   it follows that the $\Sp(4)$--action on $\p^4$ coincides with the $\SO(5)$--action and then there is only one invariant quadric.
\begin{proposition}\label{propOrbGour}
	The image of $\Sing(\LG(3,\CC)^\vee) \setminus \LG(3,\CC)$ via $\varpi$ is the (5--dimensional, contact) adjoint variety $\p(\OO_{[2,1^4]})$ of $\Sp(\CC)$, see Section \ref{subAdjVar}. The fiber of $\varpi$ restricted to $\Sing(\LG(3,\CC)^\vee) \setminus \LG(3,\CC)$ is isomorphic to $\p^4 \setminus \LG(2,4)$. The ``boundary'' of the compactification of the fiber, i.e.,  $\LG(2,4)$, is a subset of $\LG(3,\CC)$.
\end{proposition}
\begin{proof}
	Since the vector corresponding to the chosen representative $[e_{163} +e_{125}]$ has weight $h_1$, its square has weight $2h_1$. There are two other elements in $S^2\big(\Lambda^3_0(\CC)\big)$ with this weight, namely  $e_{123}\cdot e_{156}$ and $e_{153}\cdot e_{126}$ and some combination of these three vectors is the highest weight vector $w_{(2,0,0)}$ of $W_{(2,0,0)}=S^2(\CC) \subset S^2 \big(\Lambda^3_0(\CC)\big)$. 
	To determine $w_{(2,0,0)}$ recall that it is annihilated by all primitive positive root vectors. On the other hand, the images of elements with weights $2h_1 + 2h_3$ and $2h_1 +h_2 -h_3$ via negative root vectors belong to $W_{(0,0,2)} \subset S^2 \big(\Lambda^3_0(\CC)\big)$ (i.e., to the complement). Direct computations show that   $w_{(2,0,0)}=(e_{163}+e_{125})^2 -4e_{123}e_{156} +4e_{153}e_{126}$  and then the 3--dimensional weight space $\large( S^2 \big(\Lambda^3_0(\CC)\big) \large)_{2h_1}$  splits as
	\begin{equation*}
	    \large( S^2 \big(\Lambda^3_0(\CC)\big) \large)_{2h_1}=\large(W_{(2,0,0)} \large)_{2h_1}\oplus \large(W_{(0,0,2)} \large)_{2h_1} =\Span{w_{(2,0,0)}}\oplus\Span{e_{123}e_{156} +e_{153}e_{126}, (e_{163}+e_{125})^2 + 2e_{126}e_{153}}\, .
	\end{equation*}
    Since $(e_{163} +e_{125})^2$ has nonzero component along $w_{(2,0,0)}$, it projects nontrivially onto $(W_{(2,0,0)})_{2h_1}$ and this subspace is spanned by the matrix $E(1,4)$, that is $X_{2h_1}$, which  clearly belongs to $\OO_{[2,1^4]}$, see Remark \ref{remEIJ}.\par
    
	As for the second statement, recall that the  fiber of $\varpi$  that passes  through $[e_{163} + e_{125}]$ is a homogeneous space of the subgroup $\Stab_{\Sp(\CC)}([E(1,4)])$ acting on $\p(\Lambda^3_0(\CC))$.
	
	It is now convenient to introduce the symplectic space    $V' = \Span{e_2, e_3, e_5, e_6}$: indeed,   every element of $\Stab_{\Sp(\CC)}([E(1,4)])$ leaves   $e_1$ fixed  and only $e_4$ gets sent to $e_4$, so that   the stabilizer is isomorphic to  $  \Sp(V^{\prime})\simeq\Sp(4)$. More precisely, the wedge--multiplication by $e_1$ realizes an embedding of $\p(\Lambda^2_0(V'))$ into $\p(\Lambda^3_0(\CC))$ that commutes with the action of $\Stab_{\Sp(\CC)}([E(1,4)])$, understood as a subgroup of $\Sp(\CC)$: this reduces   our problem  to that of finding the $\Sp(V^\prime)$--orbit of the class of the two--form $e_{63} + e_{25}$ living in $\p(\Lambda^2_0(V'))$. The second claim follows then from the fact that $e_{63} + e_{25}$ is it not a decomposable two--form, i.e.,  $[e_{63} + e_{25}] \not\in \LG(2,V')$. Finally, the wedge multiplication by $e_1$ maps the Lagrangian subspaces of $V'$ to Lagrangian subspaces of $\CC$, and they belong to $\LG(3,\CC)$. 
\end{proof}

\subsubsection{(L) Image of the twelve--dimensional orbit}
We introduce now a special projective line in $\p(\Lambda^3_0(\CC))$:
\begin{equation*}
    \p^1:=\overline{\{[e_{423} + e_{126} + e_{153} + k \cdot e_{123}]\mid k\in\C\}}\ ,
\end{equation*}
understood as the compactification of its (one--dimensional) affine neighborhoud
\begin{equation*}
    \C^1:= \{[e_{423} + e_{126} + e_{153} + k \cdot e_{123}]\mid k\in\C\}\, ,
\end{equation*}
by means of the ``point at infinity'' $[e_{123}]$.
\begin{proposition}\label{propOrbitLin}
	The image of  $\LG(3,\CC)^\vee \setminus \Sing(\LG(3,\CC)^\vee)$  via $\varpi$ is $\p(\OO_{[2^3]})$, which has dimension 11. The fiber of $\varpi$ over $\LG(3,\CC)^\vee \setminus \Sing(\LG(3,\CC)^\vee)$ is isomorphic to  $\C^1$, and can be compactified to $\p^1$ by a point lying in $\LG(3,\CC)$.
\end{proposition}
\begin{proof}
	Consider the square of $e_{423} + e_{126} + e_{153} + k \cdot e_{123}$. It can be decomposed into the sum of elements with weights $\pm 2h_1 \pm 2h_2 \pm 2h_3$ (pure squares), $2h_i +2h_j$ (elements with $k\cdot e_{123}$) and $2e_{126}e_{153} + 2e_{423}e_{126} + 2e_{423}e_{153}$. The projection onto $W_{(2,0,0)}$ kills the first two sets  of elements, and as in  Proposition \ref{propOrbGour}, we can check that  the three last elements project nontrivially onto the corresponding weight subspaces in $\sp(\CC)$: they are spanned by $(e_{623}+e_{124})^2 -4 e_{123}e_{426} +4e_{126}e_{423}$ for $2h_2$ and $(e_{523}+e_{143})^2 -4e_{123}e_{453} +4e_{423}e_{153}$ for $2h_3$. Therefore, the image of the chosen representative does not depend on $k$ and is equal to the class of $E(1,4) + E(2,5)+ E(3,6)$ and such matrices live in $\p(\OO_{[2^3]})$. \par
	In other words,  we have just shown that the fiber contains $\C^1$; since the ``point at infinity'' $[e_{123}]$ lies in  $\LG(3,\CC)$, the   closure $\p^1$ of $\C^1$ is no longer a subset of the fiber.\par
	Now, to show that there are no other components than $\C^1$, we consider the stabilizer subgroup \linebreak $\Stab_{\Sp(\CC)}([E(1,4) + E(2,5) + E(3,6)]) \subset  \Sp(\CC)$, and recall that it acts transitively on the fiber: we will prove the connectedness of the fiber by showing that the stabilizer subgroup is connected. To this end, observe that  
	\begin{equation*}
    \Stab([E(1,4) + E(2,5) + E(3,6)])=\left\{  \left( \begin{array}{cc}
        A & B \\
        0 & a\cdot A
    \end{array} \right)\mid a \in \C^*, a\cdot A \cdot A^t = I_3, A^t \cdot B = B^t \cdot A   \right\} \,  .
\end{equation*}
    Since we want to stabilize the projective class of an element we gain an additional parameter: the scalar $a$. Thanks to it, the determinant of $A$ can be any nonzero complex number, therefore the stabilizer subgroup is connected, and that finishes the proof.
\end{proof}

\subsubsection{(O) Image of the open orbit}\label{subsecImageOpenOrbit}

As before, we introduce  a special projective line in $\p(\Lambda^3_0(\CC))$:
\begin{equation*}
    \p^1:=\overline{\{[e_{123}+ k\cdot e_{456}]\mid k\in\C\}}\ ,
\end{equation*}
understood this time as the compactification of 
\begin{equation}\label{eqCiStar}
    \C^\times:= \{[e_{123}+ k\cdot e_{456}]\mid k\in\C\smallsetminus\{0\}\}\, ,
\end{equation}
by means of the ``point at infinity'' $[e_{456}]$, as well as well as the ``zero point'' $[e_{123}]$.
\begin{proposition}\label{propOpenOrbit}
	The image of  $\p(\Lambda^3_0(\CC))\smallsetminus \LG(3,\CC)^\vee$ via $\varpi$ is the projectivization
	\begin{equation*}
	   \p\left( \bigcup_{\nu\in\C\smallsetminus\{0\}}\Sp(\CC)\cdot \nu \cdot (\epsilon^1e_1+\epsilon^2e_2+\epsilon^3e_3)\right)
	\end{equation*}
	of the sum of a 1--parameter family of 12--dimensional orbits of type $q^{(3)}$. The fiber of $\varpi$, restricted to \linebreak $\p(\Lambda^3_0(\CC))\smallsetminus \LG(3,\CC)^\vee$, is equal to $\C^\times$ and can be compactified by two points that belong to $\LG(3,\CC)$.
\end{proposition}

\begin{proof}
	Consider the element $e_{123}+ k\cdot e_{456}$, whose projective class belongs to the open orbit for all $k \ne 0$. Its square is equal to $e_{123}^2 +k^2 e_{456}^2 +2k e_{123}e_{456}$ and the weights of its two first summands clearly indicate that they belong to $W_{(0,0,2)}$. The mixed term has weight 0, so if it projects nontrivially onto $\sp(\CC)$ then it lives in the Cartan subalgebra. We claim that in fact the image of $[e_{123} + ke_{456}]$ is precisely $[h_1 +h_2 +h_3] = [\diag(1,1,1,-1,-1,-1)]$. To  this end, we can compute how the weight vectors $X_{2h_i}$ act upon $2k e_{123}e_{456}$, namely $X_{2h_1}\cdot 2k e_{123}e_{456} = 2ke_{123}e_{156}$, $X_{2h_2}\cdot 2k e_{123}e_{456} = 2ke_{123}e_{426}$ and $X_{2h_3}\cdot 2k e_{123}e_{456} = 2ke_{123}e_{453}$. Since all three are nonzero, it follows that $[2ke_{123}e_{456}] \mapsto [ah_1 +bh_2 +ch_3]$ for $a,b,c \ne 0$, and moreover $a=b$ because $X_{h_2 -h_1} \cdot X_{2h_1}\cdot 2k e_{123}e_{456} = X_{h_1 -h_2} \cdot X_{2h_2}\cdot 2k e_{123}e_{456}$. In the same manner we can show that $b=c$. \par
	As before, what we have shown so far ensures that the fiber contains $\C^\times$ and we see that the points $0$ and $\infty$ lie in $\LG(3,6)$. Again, let us consider the stabilizer acting transitively on the fiber over $[h_1 + h_2 + h_3]$:
\begin{equation*}
    \Stab([h_1 +h_2 +h_3])=\left\{  \left( \begin{array}{cc}
        A & B \\
        C & D
    \end{array} \right) \in \Sp(6) \mid  \left( \begin{array}{cc}
        A & -B \\
        C & -D
    \end{array} \right) = \left( \begin{array}{cc}
        aA & aB \\
        -aC & -aD
    \end{array} \right), a \in \C^* \right\} \,  .
\end{equation*}
    Since now $a$ can only take values  $\pm 1$,  the stabilizer has two connected components: $B=C=0$ and $A=D=0$, so we cannot conclude as before. However, it is straightforward that for any $g$ belonging to the stabilizer, $g \cdot [e_{123} +e_{456}] = [(\det(A) + \det(B)) e_{123} + (\det(C) +\det(D))e_{456}]$, so that there are no other components. 
\end{proof}


\subsubsection{Summary of the results}\label{subListNormalForms}

\begin{equation*}
\begin{array}{l|l||l|l|l||l}
   \textrm{orbit}  & \textrm{dimension}   & \textrm{representative} & \textrm{image of the representative}& \textrm{fiber} & \textrm{PDE}\\
   \hline
    \textrm{O} & 13   & [e_{123} + e_{456}] & [\epsilon^1e_1+\epsilon^2e_2+\epsilon^3e_3]& \C^\times&\det\|u_{ij}\|=1\\
    \textrm{L} & 12   &  [e_{423} + e_{126} + e_{153} + e_{123}]& [e_1^2+e_2^2+e_3^2]&\C^1&  u_{11}+u_{22}+u_{33}=0 \\
     \textrm{G} & 9  &  [e_{163} + e_{125}]& [e_1^2]& \p^4\smallsetminus\LL(2,4)& u_{23}=0\\
     \textrm{P} & 6  &  [e_{123}]& \emptyset& \textrm{-- --}& u_{11}=0\\
\end{array}
\end{equation*}

It is worth stressing that, by regarding each element in the ``representative'' column as a 3--form \eqref{eq:Phi.per.MAE} on $\CC$, see Remark \ref{remDictionary}, and then by applying \eqref{eq:gen.MAE} to the so--obtained  3--from, one \emph{does not} obtain exactly the corresponding element in the ``PDE" column, but rather an $\Sp(\CC)$--equivalent to it. The following example clarifies this aspect.
\begin{example}\label{ex:chiarificatore}
By employing the same procedure we used in Example \ref{ex:gianni.co.co}, we will associate a Monge-Amp\`ere equation to all of the four 3--forms that appear in the ``representative'' column of the above table. 
\begin{itemize}
\item The form $e_{123} + e_{456}$ gives the Monge-Amp\`ere equation $\det\|u_{ij}\|=1$.
\item The form $e_{423} + e_{126} + e_{153} + e_{123}$ gives the Monge-Amp\`ere equation 
\begin{equation*}
\det\|u_{ij}\|=u^\sharp_{11}+u^\sharp_{22}+u^\sharp_{22}\,.
\end{equation*}
The last equation, however, thanks to a transformation \eqref{eq:legendre.in.un.punto}, is $\Sp(\CC)$--equivalent  to $u_{11}+u_{22}+u_{33}+1=0$, which, in turn, is $\Sp(\CC)$--equivalent to $u_{11}+u_{22}+u_{33}=0$, by using, for instance, the additional transformation  $x^4\to x^4-x^1$.
\item The $e_{163} + e_{125}$ gives the Monge-Amp\`ere equation $u^\sharp_{23}=0$, which is $\Sp(\CC)$--equivalent to $u_{23}=0$, again, by means of  \eqref{eq:legendre.in.un.punto}.
\item The form $e_{123}$  gives the Monge-Amp\`ere equation $\det\|u_{ij}\|=0$, which  is $\Sp(\CC)$--equivalent to $u_{11}=0$: to see this it suffices to use a  partial Legendre transformation 
$$
(x^1,x^2,x^3,x^4,x^5,x^6)\to (x^1,x^5,x^6,x^4,-x^2,-x^3)\, ,
$$ 
see  Remark \ref{rem:legendre.and.partial.point}.
\end{itemize}
\end{example}

\subsection{Equivalence of the KLR invariant with the Hitchin moment map}\label{subsRussianContractio} We recast now, in the complex setting, the definition of the KLR invariant, introduced earlier in Section \ref{secKLRform}; the underlying idea is the same:  a 3--form $\eta$ on a 6--dimensional symplectic space $(\CC,\omega)$ can be contracted with  (the inverse $\omega^{-1}$ of) the symplectic form $\omega$, thus obtaining a quadratic form on $\CC$. In the present coordinates, we have
\begin{equation*}
    \eta=\eta_{ijk}x^{ijk}\in\Lambda^3(\CC^*)\, ,
\end{equation*}
so that the output $q(\eta)$ of the   contraction, i.e., formula \eqref{eq.JGG}, will now look like
\begin{equation*}
    q(\eta)_{ab}=\eta_{aij}\eta_{bhk}\omega^{ih}\omega^{jk}\in S^2(\CC^*)\, .
\end{equation*}
Direct computations show that $g^*(q(\eta))=q(g^*(\eta))$, for all $g\in\Sp(\CC)$; moreover, if $\eta=\alpha\wedge\omega$, then $q(\eta)=3\alpha^2$: in other words,
\begin{equation*}
    \xymatrix{
    \CC^*\ar[r]^{m_\omega}\ar[dr]_{3v_2} & \Lambda^3(\CC^*)\ar[d]^q\\
    & S^2(\CC^*)
    }
\end{equation*}
is a commutative diagram of $\Sp(\CC)$--invariant polynomial  maps. We will denote by the same symbol $q$ the restriction of $q$ to the subspace $\Lambda^3_0(\CC^*)$.  Let us stress that $\varpi:\p(\Lambda^3_0(\CC))\dashedrightarrow\p(S^2(\CC))$, cf. \eqref{eqDiagVarpi}, whereas $[q]:\p(\Lambda^3_0(\CC^*))\dashedrightarrow\p(S^2(\CC^*))$, so that the next statement makes sense, provided that one identifies $\CC$ with $\CC^*$ via $\omega$, see Section \ref{subCCstar}.
\begin{theorem}\label{thFirstOriginalResult}
The rational quadratic map $\varpi$ coincides with the projectivization of $q$.
\end{theorem}
\begin{proof}
    In view of the invariance of both $\varpi$ and $q$ it suffices to check the claim on the orbit representatives. 
    We also observe that, in our coordinates, the inverse of $\omega$ (which is nothing but $-\omega$) reads $x^{41}+x^{52}+x^{63}$.\par Let us start from the open orbit (O):   if we take $\eta=e_{123}+e_{456}$ we have that
    \begin{eqnarray*}
        q(\eta)&=& (e_{123}+e_{456})(e_{123}+e_{456})(x^{41}+x^{52}+x^{63})(x^{41}+x^{52}+x^{63}) \\
        &=& (e_{123}+e_{456}) ( e_{23}x^4 +e_{13}x^5 + e_{12}x^6 + e_{56}x^1 +e_{46}x^2 +e_{45}x^3 ) (x^{41}+x^{52}+x^{63})\\
        &=& (e_{123}+e_{456}) \,\cdot\,2 \,\cdot\, (e_1x^5x^6+e_2x^4x^6+e_3x^4x^5+e_4x^2x^3+e_5x^1x^3+e_6x^1x^2)\\
        &=&4  \,\cdot\,(e_1e_4+e_2e_5+e_3e_6)\\
        &=&4  \,\cdot\,q^{(3)}\, .
    \end{eqnarray*}
This means that $\varpi([\eta])=[g(\eta)]$ and the claim is valid on the orbit O, see Proposition \ref{propOpenOrbit}.\par
In order to study the orbit L we take $\eta=e_{423} + e_{126} + e_{153} + e_{123}$ and we compute
    \begin{eqnarray*}
        q(\eta)&=& (e_{423} + e_{126} + e_{153} + e_{123})(e_{423} + e_{126} + e_{153} + e_{123})(x^{41}+x^{52}+x^{63})(x^{41}+x^{52}+x^{63}) \\
        &=& (e_{423} + e_{126} + e_{153} + e_{123})\,\cdot\\
        &&\cdot\, \big( e_{23}x^1+(e_{26} +e_{53}+e_{23})x^4+(e_{43}+e_{16}+e_{13})x^5+e_{13}x^2+(e_{42}+e_{15}+e_{12})x^6+e_{12}x^3\big)\,\cdot\\
        &&\cdot\, (x^{41}+x^{52}+x^{63})\\
        &=& (e_{423} + e_{126} + e_{153} + e_{123})\,\cdot\\
                &&\cdot\, 2 \,\cdot\,\big( e_1(x^2x^6+x^5x^6+x^3x^5)+e_2(x^1x^6+x^4x^6+x^3x^4)+e_3(x^1x^5+x^2x^4+x^4x^5)+e_4x^5x^6+e_5x^4x^6+e_6x^4x^5\big)\\
         &=&2 \,\cdot\, (e_1(e_1+e_1)+e_2(e_2+e_2)+e_3(e_3+e_3))\\
        &=&4  \,\cdot\,(e_1^2+e_2^2+e_3^2)\\
        &=&4 \,\cdot\, q_{[2^3]}\, ,
    \end{eqnarray*}
    and the claim follows now from Proposition \ref{propOrbitLin}.\par
    We pass   to the last orbit on which $\varpi$ acts nontrivially, that is the orbit G: we take $\eta=e_{163} + e_{125} $ and  we compute
     \begin{eqnarray*}
        q(\eta)&=& (e_{163} + e_{125})(e_{163} + e_{125})(x^{41}+x^{52}+x^{63})(x^{41}+x^{52}+x^{63}) \\
        &=& (e_{163} + e_{125}) ( e_{63}x^4 +e_{25}x^4 + e_{12}x^2 + e_{15}x^5 +e_{13}x^2 +e_{16}x^3 ) (x^{41}+x^{52}+x^{63})\\
        &=& (e_{163} + e_{125})\big(2 (e_1x^2x^5+e_2x^2x^4)+e_3(x^4x^6+x^3x^4)+e_5(x^4x^5+x^2x^4)+e_6(x^3x^4+x^4x^6)+e_1((x^6)^2+(x^3)^2)\big)\\
        &=&2 \,\cdot\,  e_1^2 \\
        &=&2  \,\cdot\,q_{[2,1^4]}\, ,
    \end{eqnarray*}
    see then Proposition \ref{propOrbGour}.\par
    To complete the proof, we recall Proposition \ref{propSingLoc} and observe that, on the representative of the orbit G, we have
     \begin{eqnarray*}
        q(e_{123})&=& e_{123}e_{123}(x^{41}+x^{52}+x^{63})(x^{41}+x^{52}+x^{63}) \\
        &=& e_{123}( e_{23}x^4 +e_{13}x^5 + e_{12}x^6  ) (x^{41}+x^{52}+x^{63})\\
        &=& e_{123} \,\cdot\, 2  \,\cdot\,(e_1x^5x^6+e_2x^4x^6+e_3x^4x^5 )\\
        &=&0\, .
    \end{eqnarray*}
\end{proof}

\subsection{Fiber compactification via singular limits (over $\R$)}

In this subsection we momentarily go back to the real--differentiable setting.\par

Indeed, the two--points compactification of the fiber through a generic Monge--Amp\`ere equation, discussed in Section \ref{subsecImageOpenOrbit} above, becomes particularly evident if we go back to the real case and employ a total Legendre transform  
(for reasons of clarity, below we shall adopt the notation of Example \ref{ex:Legendre}). Then, instead of \eqref{eqCiStar}, we shall have a real one--parametric family of generic, that is, non--linearizable, Monge--Amp\`ere equations, corresponding to the following family of 3--forms:
\begin{equation*} 
      \{[e_{123}+ k\cdot e_{456}]\mid k\in\R\smallsetminus\{0\}\}\, .
\end{equation*}

By computing the Monge--Amp\`ere equation associated with each member of the family (for the procedure, see also Example \ref{ex:chiarificatore}), we obtain a one--parameter family of PDEs:
\begin{equation}\label{eqEqLimit1}
    \det\|u_{ij}\|=k
\end{equation}
While the limit for 
$k\to 0$ is the parabolic  Monge--Amp\`ere equation 
\begin{equation}\label{eq:PMAE.appoggio}
\det\|u_{ij}\|=0\,,
\end{equation}
in order to take the limit of \eqref{eqEqLimit1}  for 
$k\to\infty$
we can, for instance, perform a total Legendre transform 
(i.e., \eqref{eq:legendre.partial} with $m=n=3$), thus obtaining (see also Example \ref{ex:oriz.and.equival})
\begin{equation*}
    \det\|\widetilde{u}_{ij}\|=-\frac{1}{k}
\end{equation*}
in the new (tilded) coordinates. Now, by taking 
$k\to \infty$, we obtain
\begin{equation}\label{eqEqLimit2}
    \det\|\widetilde{u}_{ij}\|=0\, .
\end{equation}
The two equations \eqref{eq:PMAE.appoggio} and \eqref{eqEqLimit2}, though equivalent, are not the same, as they correspond to two mutually transversal 3D subdistributions, namely $\mathcal{D}=\Span{D^{(1)}_1,D^{(1)}_2,D^{(1)}_3}$ and $\widetilde{\mathcal{D}}=\Span{\widetilde{D}^{(1)}_1,\widetilde{D}^{(1)}_2,\widetilde{D}^{(1)}_3}=\Span{\partial_{u_1},\partial_{u_2},\partial_{u_3}}$: see Section \ref{sec:Goursat}, in particular equation \eqref{eq:D} with $b_{ij}=0$. Recall also that $D^{(1)}_i=\partial_{x^i}+u_i\partial_u$, cf. \eqref{eq:total.derivative}.\par

The proposed framework has thus allowed visualizing the ``singular limit" procedure mentioned by E.~Ferapontov in the context of integrable PDEs as a two--points compactification of an affine line whose points parametrize a family of  non--linearizable Monge--Amp\`ere equations. Also, note that the two parabolic Monge--Amp\`ere equations  \eqref{eq:PMAE.appoggio} and \eqref{eqEqLimit2}, related by a total Legendre transformation, or, in a matter of speaking, by a ``flip" of the corresponding 3D subdistribution, are actually the limit points of a family of non--linearizable equations. \par

We would have obtained similar results, had we considered the partial Legendre transformation discussed in Remark \ref{rem:partial.Legendre} above. It transforms equation \eqref{eqEqLimit1} into (cf. also Example \ref{ex:oriz.and.equival})
\begin{equation*}
k\tilde{u}_{11}+\tilde{u}_{22}\tilde{u}_{33}-\tilde{u}_{23}^2=0\,.
\end{equation*}
In this case, the limits 
$k\to \infty$ and $k\to 0$ give, respectively, the equations $\tilde{u}_{11}=0$ and $\tilde{u}_{22}\tilde{u}_{33}-\tilde{u}_{23}^2=0$.

\section{Hyperplane sections of the Lagrangian Grassmannian $\LG(3,\CC)$}\label{secLinSec}

We can finally pick up the thread we left at the end of Section \ref{secMAESpace} and provide a proof of Theorem \ref{thMain1} over the field of the complex numbers, see Corollary \ref{corEquivalencyMomentCoChar} below. The reader must be warned that the definition of a Monge--Amp\`ere equation, as well as that of its cocharacteristic variety, were already given above, see  Definition \ref{defMAEq} and Definition \ref{def.cochar.var}, respectively: we repeat them below (see Definition \ref{defMAEComplex} and Definition \ref{defCoCarComp}, respectively) to stress that, despite their formal similarity, the two versions   of the same definition pertain to different categories (real differentiable versus complex analytic). Moreover, the study carried out in the second part of the paper is strictly point--wise, i.e., it pertains a particular fiber of the bundle $J^2\to J^1$, whereas in the first part the formalism is global, even though we have always considered \emph{symplectic} (i.e., not depending on the point of $J^1$) Monge--Amp\`ere equations: compare, for instance, the earlier Definition \ref{def:Goursat.MAE} of a Goursat type Monge--Amp\`ere equation with the new Definition \ref{def:Goursat.MAE.Complex} given below.

\subsection{The Lagrangian Grassmannian and its tangent geometry}\label{secFinalSubsection}
In what follows, $X$ will be denoting the 6--dimensional regular variety $\LG(3,\CC)$ of 3--dimensional isotropic linear subspaces of $\CC$, embedded into $\p(\Lambda_0^3(\CC))$ via the Pl\"ucker embedding
\begin{eqnarray*}
X & \longrightarrow & \p(\Lambda_0^3(\CC))\, ,\\
L=\Span{l_1,l_2,l_3} & \longmapsto & \vol(L):=[l_1\wedge l_2\wedge l_3]\, .
\end{eqnarray*}
Let us fix $L\in X$: a nice description of the tangent geometry of $X$ at $L$ goes as follows \cite{SamSteven2009}: we ``perturb'' the Lagrangian subspace $L$ by means of a linear map $h\in\Hom(L,\CC)$, that is, we employ $h$ to define a curve
\begin{equation}\label{eqCurvaGammaTiInX}
    \gamma_h(t):=[(l_1+th(l_1))\wedge (l_2+th(l_2))\wedge (l_3+th(l_3))]=[l_1\wedge l_2\wedge l_3+t(h(l_1)\wedge l_2\wedge l_3+l_1\wedge h(l_2)\wedge l_3+l_1\wedge l_2\wedge h(l_3))+o(t^2)]
\end{equation}
passing through $L$ at $t=0$. Above formula \eqref{eqCurvaGammaTiInX} shows that, if      $h$ takes its values in $L$, then  part that is   linear in $t$ is absorbed by the free term and, hence, the  velocity $\Dot{\gamma}_h(0)$ of $\gamma_h$ at 0 vanishes: in order to obtain all tangent vectors to $X$ at $L$ it is then safe to assume that $h$ be an element of
\begin{equation*}
    \Hom(L,\tfrac{\CC}{L})=L^*\otimes \tfrac{\CC}{L} =L^*\otimes L^*\, ,
\end{equation*}
where the identification $\tfrac{\CC}{L}=L^*$ is given by the isomorphism $\widetilde{\omega}_L$ closing the commutative diagram
\begin{equation*}
    \xymatrix{
    \CC\ar[d]\ar[r]^{\omega} & \CC^*\ar[d]\\
    \tfrac{\CC}{L}\ar[r]^{\widetilde{\omega}_L} & L^*\, .
    }
\end{equation*}
Less evident is that $\gamma_h(t)$ is a curve in $X$ if and only if $h\in S^2(L^*)$ (it can be proved by a direct coordinate approach); since the   contraction
\begin{eqnarray*}
S^2(L^*) & \longrightarrow &  \Lambda^3(\CC)\\
h &\longmapsto & h\ins(l_1\wedge l_2\wedge l_3):=
     h(l_1)\wedge l_2\wedge l_3+l_1\wedge h(l_2)\wedge l_3+l_1\wedge l_2\wedge h(l_3)
\end{eqnarray*}
is injective, it can be concluded that $h\longmapsto\Dot{\gamma}_h(0)$ is a monomorphism of $S^2(L^*)$ into $T_LX$ and then, for obvious dimensional reasons, an isomorphism.

\subsection{Hyperplane sections and their tangent geometry}
The moduli space of hyperplane sections of $X$ is the 13--dimensional projective space $\p(\Lambda_0^3(\CC^*))$: in view of the applications to the theory of $2\Nd$ order PDEs, let us recall, cf. \eqref{eqEeta}, that $\E_\eta$ denotes 
the hyperplane section determined by $[\eta]\in \p(\Lambda_0^3(\CC^*))$: in other words, from now on, the name ``(symplectic) Monge--Amp\`ere equation" will be a synonym of ``hyperplane section".
\begin{definition}\label{defMAEComplex}
The five--fold $ \E_\eta\subset X$ is the \emph{Monge--Amp\`ere equation}  associated with $\eta\in\Lambda^3_0(\CC^*)$ (in the sense of Lychagin). 
\end{definition}
Let us now fix a smooth point $L\in(\E_\eta)_{\textrm{sm}}$ and consider   the curve \eqref{eqCurvaGammaTiInX} determined by $h\in S^2(L^*)$: it will be tangent to the MAE $\E_\eta$ if and only if \begin{equation*}
    \eta(h\ins(l_1\wedge l_2\wedge l_3))=0\, .
\end{equation*}
Put differently, having defined the linear map
\begin{eqnarray*}
    T_LX=S^2(L^*) & \stackrel{\widetilde{\eta}_L}{\longrightarrow} & \Lambda^3L^*\, ,\\
    h&\longmapsto & \eta(h\ins\, \cdot\, )\, ,
\end{eqnarray*}
we obtain 
\begin{equation}\label{defTangSpaceEL}
    T_L(\E_\eta)=\ker\widetilde{\eta}_L\, .
\end{equation}
\begin{definition}
The quadric $\sigma_L\subset\p(L^*)$ cut out by the equation $\widetilde{\eta}_L=0$ is the \emph{characteristic variety} of the MAE $\E_\eta$ at the point $L$.
\end{definition}
We can globalize the above reasoning by introducing the tautological bundle $\sL$ on $X$ (whose fiber at $L\in X$ is $L$ itself): we have then a global identification
\begin{equation*}
    TX=S^2(\sL^*)
\end{equation*}
and $\widetilde{\eta}_L$ turn out to be the value at $L$ of a section
\begin{equation*}
    \widetilde{\eta}\in\Gamma((\E_\eta)_{\textrm{sm}},S^2(\sL)\otimes\Lambda^3(\sL^*))\, ,
\end{equation*}
such that
\begin{equation*}
    \left.\ker \widetilde{\eta}\right|_{(\E_\eta)_\textrm{sm}}=T (\E_\eta)_\textrm{sm}\, .
\end{equation*}
\begin{definition}
The subvariety $\sigma\subset\p(\sL^*|_{(\E_\eta)_{\textrm{sm}}})$ cut out by the equation $\widetilde{\eta}=0$ is the \emph{characteristic variety} of the MAE $\E_\eta$.
\end{definition}
Observe that
\begin{equation*}
    \sigma=\bigcup_{L\in (\E_\eta)_{\textrm{sm}}}\sigma_L\, .
\end{equation*}
\subsection{Projective duality and the cocharacteristic variety}\label{secSubSecCoCar}
We would like now to pass to the projective dual $\sigma_L^\vee\subset\p(L)$ of the characteristic varieties $\sigma_L$, because each $\p(L)$ embeds naturally into $\p(\CC)$, whereas  $\p(L^*)$  does not posses any special embedding into $\p(\CC^*)$: the passage from $\p(L^*)$ to $\p(L)$ will allow us to regard the   sum of all the $\sigma_L^\vee$'s as a subset of $\p(\CC)$, viz.
\begin{equation}\label{defCocVar}
    \sigma^\vee:=\overline{\bigcup_{L\in(\E_\eta)_\textrm{sm}}\sigma_L^\vee}\subset\p(\CC)\, .
\end{equation}
\begin{definition}\label{defCoCarComp}
The subset $\sigma^\vee$ of $\p(\CC)$ is called the \emph{cocharacteristic variety} of the MAE $\E_\eta$.
\end{definition}
The equation cutting out $\sigma_L^\vee$ is easily obtained from $\widetilde{\eta}$, regarded as a map
\begin{equation*}
\widetilde{\eta}_L: L^*\longrightarrow L\otimes\Lambda^3(L^*)
\end{equation*}
and then passing to its second exterior power:
\begin{equation*}
\Lambda^2(\widetilde{\eta}_L): \Lambda^2(L^*)\longrightarrow \Lambda^2(L)\otimes(\Lambda^3(L^*))^2\, .
\end{equation*}
Then, in virtue of the identifications $\Lambda^2(L)=L^*\otimes \Lambda^3(L)$ and $\Lambda^2(L^*)=L\otimes\Lambda^3(L^*)$, we can regard $\Lambda^2(\widetilde{\eta}_L)$ as
\begin{equation*}
\Lambda^2(\widetilde{\eta}_L): L\otimes \Lambda^3(L^*)\longrightarrow  L^* \otimes\Lambda^3(L)\otimes(\Lambda^3(L^*))^2=L^*\otimes \Lambda^3(L^*) \, ,
\end{equation*}
and finally
\begin{equation*}
\Lambda^2(\widetilde{\eta}_L)\in L^*\otimes L^* \otimes\Lambda^3(L^*)\otimes\Lambda^3(L)= L^*\otimes L^* \, .
\end{equation*}
In view of the usual symmetry consideration, we have in fact that
\begin{equation*}
\Lambda^2(\widetilde{\eta}_L)\in S^2(L^*)\, .
\end{equation*}
Since $\sigma_L^\vee$ is cut out by the $2\times 2$ minors of the ($3\times 3$) matrix of the quadratic form $\widetilde{\eta}_L$, and the matrix of such minors is precisely $\Lambda^2(\widetilde{\eta}_L)$, we have that 
\begin{equation}\label{eqCoCarPointL}
   \sigma_L^\vee=\{ \Lambda^2(\widetilde{\eta}_L)=0\}\subset\p(L)\, .
\end{equation}
It is worth observing that 
\begin{equation}\label{eqLambdaTwoEta}
    \Lambda^2(\widetilde{\eta})\in\Gamma((\E_\eta)_\textrm{sm}, S^2(\sL^*))\, ,
\end{equation}
that is, $\Lambda^2(\widetilde{\eta})$ is a section of honest (untwisted) quadratic forms on the tautological bundle $\sL$. In the last Section~\ref{secFinalSubsection} we will prove that the  cocharacteristic variety is, in the appropriate sense, cut out by   section \eqref{eqLambdaTwoEta}. Before passing to that, we study the rank of the characteristic variety across the four isomorphism types of   Monge--Amp\`ere equations.

\subsection{Singular loci of Monge--Amp\`ere equations}

The Lagrangian Grassmanian $X\subset\p(\Lambda_0^3(\CC))$ allows us to recast the stratification of $\p(\Lambda_0^3(\CC^*))$ in terms of dual varieties: in this perspective, the open (13--dimensional) orbit is the complement $\p(\Lambda_0^3(\CC^*))\smallsetminus X^\vee$, whereas the closed (6--dimensional) one is isomorphic to $X$ itself via the   isomorphism
\begin{equation*}
    \p(\Lambda_0^3(\CC))\longrightarrow \p(\Lambda_0^3(\CC^*))
\end{equation*}
induced by the symplectic form $\Omega$, see Remark \ref{remParamSpaceMAE}. In order to stress that $X$ in embedded into $\p(\Lambda_0^3(\CC^*))$ we will use the symbol $X^*$ for it (not to be confused with the projective dual $X^\vee$).\par
The remaining strata, of dimension 12 and 9, are $X^\vee\smallsetminus\Sing(X^\vee)$ and $\Sing(X^\vee)\smallsetminus X^*$, respectively.\par

As observed in \cite[Proposition 2.5.1 (iii)]{Iliev2005} there is a finite map
\begin{eqnarray}\label{eqMappaMappinaGr3LGr}
    \Gr(3,\CC)&\stackrel{}{\longrightarrow} & \Sing(X^\vee)\, ,\\
    D & \longmapsto & [{\eta(D)}]\, ,\nonumber
\end{eqnarray}
where $[{\eta(D)}]$ is the only element, such that the hyperplane section $\E_{\eta(D)}$ is the Schubert  cycle
\begin{equation}\label{eqSchubert}
   \E_{\eta(D)}= \E_D:=\{L\in X\mid\dim(L\cap D)\geq 1\}
\end{equation}
determined by $D$ in $X$; see also \cite[Section 3.3]{MR2985508} for the real--differentiable setting.\par
\begin{definition}\label{def:Goursat.MAE.Complex}
The Monge--Amp\`ere equation $\E_D$ is called the \emph{Goursat type Monge--Amp\`ere equation} determined by the 3--dimensional subspace $D\subset\CC$. If $D$ is Lagrangian, then $\E_D$ is called \emph{parabolic}.
\end{definition}
The map \eqref{eqMappaMappinaGr3LGr}, restricted to the Lagrangian Grassmannian $X\subset  \Gr(3,\CC)$, realizes an isomorphism between $X$ and $X^*$, whereas the fiber over a point $[{\eta(D)}]\in\Sing(X^\vee)\smallsetminus X^*$ is the pair $\{D, D^\perp\}$: all of this reflects what happens in the real--differentiable setting, see Section \ref{sec:Goursat} above.

The  results of   Theorem \ref{thMain2} below are partially contained in the above--quoted \cite[Proposition 2.5.1]{Iliev2005}.
\begin{theorem}\label{thMain2}
Let $[\eta]\in\p(\Lambda^3_0(\CC^*))$. Up to $\Sp(\CC)$-action, we have four possibilities:
\begin{itemize}
    \item[(O)] if $[\eta]\not\in X^\vee$, then $\E_\eta$ is nonsingular and $\rk\widetilde{\eta}=3$;
    \item[(L)] if $[\eta]\in X^\vee\smallsetminus\Sing(X^\vee)$, then $\E_\eta$ has a unique quadratic singularity and $\rk\widetilde{\eta}=3$ on $(\E_\eta)_{\textrm{sm}}$;
    \item[(G)] if $[\eta]\in \Sing(X^\vee)\smallsetminus X^*$, then $\Sing(\E_\eta)$ is isomorphic to the Schubert cycle determined by a non--Lagrangian two--dimensional subspace in $\LG(2,4)$ and $\rk\widetilde{\eta}=2$ on $(\E_\eta)_{\textrm{sm}}$;
    \item[(P)] if $[\eta]\in  X^*$, then $\Sing(\E_\eta)$ is the projective cone over the Veronese surface, and $\rk\widetilde{\eta}=1$ on $(\E_\eta)_{\textrm{sm}}$.
\end{itemize}
\end{theorem}
Before giving a proof, we recall 
the $\Sp(\CC)$--equivariant double fibration
 \begin{equation}\label{eqDiagChow}
     \xymatrix{
      & \LFl(1,3;\CC)\ar[dl]_{\overline{p}}\ar[dr]^{\overline{q}}&\\
      \p\CC && X\, ,
     }
 \end{equation}
 where $\LFl(1,3;\CC)$ is the smooth, eight--dimensional $\Sp(\CC)$--homogeneous variety of \emph{Lagrangian} (or, according to some authors, \emph{$\omega$--isotropic})  \emph{flags} of $\CC$. 
The fibers are easily described:  $\overline{p}^{-1}(\ell)=\LG\left(2,\tfrac{\ell^\perp}{\ell}\right)$, that is, a copy of the 3--dimensional quadric $\LG(2,4)\subset\p^4$, for all $\ell\in\p\CC$,  whereas $\overline{q}$ is nothing but the projectivized tautological bundle $\p\sL$, i.e.,   $\overline{q}^{-1}(L)=\p L$, that is, a copy of $\p^2$, for all $L\in X$.\par
 The projectivization $\p D$ of a 3--dimensional linear subspace $D\in\Gr(3,\CC)$ of $\CC$ is  an irreducible subvariety of $\p\CC$ of pure codimension three and degree one: by regarding the Schubert cycle \eqref{eqSchubert} as the \emph{Lagrangian Chow transform} 
 \begin{equation}\label{eqLagChowTrED}
     \E_D=\overline{q}(\widetilde{\E}_D)\, ,\quad \widetilde{\E}_D:=\overline{p}^{-1}(\p D)=\{(\ell, L)\mid \ell\subseteq L\cap D\}\, 
 \end{equation}
 of $\p D$, we come to the conclusion that $\E_D$ is     a hypersurface of X of degree one, that is, a hyperplane section, see \cite[Lemma 23]{MR3904635}. This hypersurface, however, is not smooth.  \par
 Diagram \eqref{eqDiagChow} provides  a singularity resolution of $\E_D$:  the subset $\widetilde{\E}_D\subset \LFl(1,3;\CC)$, being the restriction of the bundle $\overline{p}$ to a smooth subvariety of its base $\p\CC$, is smooth as well and has dimension $5=2+3=\dim(\p D)+\dim(\LG(2,4))$, that is the same  dimension of $\E_D\subset X$: by restricting \eqref{eqDiagChow} to $\p D$ and $\E_D$ we obtain then a double fibration
 \begin{equation}\label{eqDiagChowRestr}
     \xymatrix{
      & \widetilde{\E}_D\ar[dl]_{p}\ar[dr]^{q}&\\
      \p D && \E_D\, ,
     }
 \end{equation}
 where $p$ has the same fibers as $\overline{p}$, but the restriction $q$  of $\overline{q}$ is a surjective morphism between varieties of the same dimension five.\par
Notation 
 \begin{equation}\label{eqDefEDiI}
        \E_{D,i}:=\{ L\in X\mid \dim(L\cap D) =i\}\, ,\quad  i=1,2,3\, ,
 \end{equation}
 will come in handy in the proof below.
 
\begin{proof}[Proof of Theorem \ref{thMain2}]
    In the case (O), the hyperplane $\p(\ker\eta)$ is nowhere tangent to $X$, therefore $\E_\eta=X\cap \p(\ker\eta)$ must be nonsingular. \par
    In the case (L), there exists a unique line through $[\eta]$ that is tangent to $X^*$. The point that corresponds  to the tangency point in $X$, via the isomorphism induced by the symplectic form,  is precisely the unique singular point of $\E_\eta$. If we choose some local coordinates around this singular point, then the Hessian will be nonzero in it, so the singularity will be quadratic: see \cite[Theorem 6.1]{Tevelev2003} \par
    In the case (G), there exists an element $D\in\Gr(3,\CC)\smallsetminus X$, such that $\E_\eta=\E_D$: the dimension of $L\cap D$ must then be less than 3 and the union  
    \begin{equation*}
        \E_{D}=\E_{D,1}\cup \E_{D,2}
    \end{equation*}
    is disjoint, cf. \eqref{eqDefEDiI}. 
    We also observe that, since $\omega|_D$ must be degenerate and $\ker\omega|_D$ cannot be 3, the line $\ell_D:=\ker\omega|_D$ will fulfill $\omega(\ell,D)=0$, that is $\ell_D\subset D^\perp$: summing up,
    \begin{equation}\label{edDefEllDiPerp}
        \ell_D=D\cap D^\perp\, ,\quad \dim\ell_D=1\, ,\quad D + D^\perp =\ell_D^\perp\, ,\quad \dim\ell_D^\perp=5\, .
    \end{equation}
    If $L\in \E_{D,1}$, then $\ell(D):=L\cap D$ is a line and then $q^{-1}(L)=\{\ell(D)\}\subset\p L$ is a point, cf. \eqref{eqLagChowTrED}: in other words, 
    \begin{eqnarray*}
        q' : \E_{D,1} & \longrightarrow & \p\sL|_{\E_{D,1}}\\
        L&\longmapsto & \ell(D)
    \end{eqnarray*}
    defines a section of the projectivized tautological bundle over $\E_{D,1}$, i.e., a section of   $\overline{q} $ over $\E_{D,1}$, that  takes its values in ${q}^{-1}(\E_{D,1})$ and it is clearly an inverse of  
    $q$.  Summing up,  $q$  restricts to an isomorphism between the subsets $q^{-1}(\E_{D,1})\subset \widetilde{\E}_D $ and $\E_{D,1}\subset \E_D$, which in turn implies that $\Sing(\E_D)\subseteq\E_D\smallsetminus\E_{D,1}=  \E_{D,2}$.\par
    If $L\in \E_{D,2}$, then $\pi:=L\cap D $ is a plane and, by passing to the corresponding  symplectic orthogonal subspaces, $\pi^\perp=L + D^\perp$, with $\dim\pi^\perp=4$: it follows that $\dim L \cap D^\perp =2$, i.e., $L$ has a plane in common both with $D$ and with $D^\perp$; but $L$ has dimension three, so these two planes cannot be disjoint: they must necessarily have a line in common, and such a line  must necessarily be $\ell_D$. In other words, $L$ contains $\ell_D$ and it is contained into $\ell_D^\perp$; it follows that   the image $\widetilde{L}$ of $L$ via the projection 
    \begin{equation*}
        \ell_D^\perp\longrightarrow\frac{\ell_D^\perp}{\ell_D}\, 
    \end{equation*}
    on the 4--dimensional symplectic space $\tfrac{\ell_D^\perp}{\ell_D}$ will be a 2--dimensional isotropic subspace. This means that  $\widetilde{L}\in\LG\left(2,\tfrac{\ell_D^\perp}{\ell_D}\right)$; denoting by $\widetilde{D}$ the projection of $D$, it is then obvious that the projection itself induces an isomorphism between $\E_{D,2}$ and the (two--dimensional) Schubert cycle  
\begin{equation*}
    \E_{\widetilde{D}}=\left\{\widetilde{L}\in\LG\left(2,\tfrac{\ell_D^\perp}{\ell_D}\right)\mid \dim(\widetilde{L}\cap \widetilde{D})\geq 1\right\}\, ,
\end{equation*}
which is smooth, since the dimension of $\widetilde{L}\cap \widetilde{D}$ is $\geq 1$ if and only if it is equal to one. We also observe that $q^{-1}(L)=\p (L\cap D)\subset \p L$ is a copy of $\p^1$, for all $L\in \E_{D,2} $.\par 
To prove the first part of claim (G) it then suffices to show that $\Sing(\E_D)\supseteq \E_{D,2}$; let us take $L\in \E_{D,2}$ and assume that there is a smooth neighborhood $\mathcal{U}\subset \E_D$ of $L$: then $q$ will be a surjective morphism between the  five--folds $q^{-1}(\E_{D,2})$ and $\E_{D,2}$ and, therefore, there are two possibilities: either it is an isomorphism, or the locus 
\begin{equation}\label{eqSingLoc}
  Z:= \{z\in q^{-1}(\mathcal{U})\mid \det(d_zq)=0 \}\subseteq q^{-1}(\mathcal{U})\subseteq \widetilde{\E}_D \ ,
\end{equation}
where the rank of $d q$   drops, is a hypersurface, that is, 4--dimensional. The first scenario is quickly discarded, since the fiber of $q$  over $L$ is a whole $\p^1$; in general,
\begin{equation}\label{eqPuntSingSchubCell}
    q^{-1}(\E_{D,2}) \subseteq Z\, 
\end{equation}
is a proper subset, because $\dim \E_{D,2}=2$ and $q$ has one--dimensional fibers over $\E_{D,2}$, hence $q^{-1}(\E_{D,2})$ has dimension three: at the same time, since $Z$ cannot contain the pre--images of points outside  $\E_{D,2}$, i.e., $Z\cap q^{-1}(\E_D\smallsetminus \E_{D,2})=\emptyset$, inclusion $Z\subseteq  q^{-1}(\E_{D,2})$ should hold as well---a contradiction.\par
It was then forbidden to assume that  $L\in \E_{D,2}$ admit a smooth neighborhood $\mathcal{U}\subset \E_D$: all points of $\E_{D,2}$ are singular, whence the sought--for inclusion $\Sing(\E_D)\supseteq \E_{D,2}$.
\par
In the last case (P) the element $D$, such that $\E_\eta=\E_D$, is Lagrangian, i.e., $D\in  X$: the dimension of $L\cap D$ can then attain the value 3 in the case when $L$ coincides with  $D$,  so that in the disjoint union  
    \begin{equation*}
        \E_{D}=\E_{D,1}\cup \E_{D,2}\cup \E_{D,3}\, ,
    \end{equation*}
 cf. \eqref{eqDefEDiI}, we have $\E_{D,3}=\{D\}$. 
As before, the restriction of $q$  defines an isomorphism  between   $q^{-1}(\E_{D,1})  $ and $\E_{D,1} $, so that     $\Sing(\E_D)\subseteq \E_{D,2}\cup \E_{D,3}$.\par 
%
Since $D$ is Lagrangian, there always exists a $D'\in X$, such that $\CC=D\oplus D'$ is a bi--Lagrangian decomposition; the latter determines the so--called \emph{big cell} $\mathcal{V}:=\{L\in X\mid L\cap D'=0\}$, that is an open and dense subset of $X$ isomorphic to $S^2 D^*$: a quadratic form $h\in S^2 D^*$ is identified with the Lagrangian subspace $\Graph(h):=\Span{x+h(x)\mid x\in D}\subset\CC$, having understood $h\in\Hom(D, D^*)$ as an element of $\Hom(D, D')$ by means of the natural identifications $D^*\simeq\tfrac{\CC}{D}\simeq D'$, see \cite[Section 1.2]{Gutt2019}. Obviously, $\E_{D,3}\subset\mathcal{V}$, because $\E_{D,3}=\{D\}$ and $D=\Graph(0)\in\mathcal{V}$.\par 
A key remark is that $L\cap D=\ker(h)$ in the case when  $L=\Graph(h)$: it follows  that  $\E_D\cap\mathcal{V}=\{h\in S^2D^*\mid\det(h)=0  \}$; therefore,  the well--known formula  
\begin{equation*}
    d_h(\det)=h^\sharp\cdot dh\
\end{equation*}
for the differential of the determinant $\det:S^2D^*\to\C$ at the point $h\in S^2 D^*$ implies  that
\begin{equation*}
    \Sing(\E_D\cap\mathcal{V})=\Sing(\E_D)\cap\mathcal{V}=\{h\in S^2D^*\mid\rk(h)=0,1  \} \, .
\end{equation*}
On the other hand, if $L\in \E_{D,2}\cap\mathcal{V}$, then $L=\Graph(h)$ and $L\cap D=\ker(h)$ is two--dimensional, i.e., $\rk(h)=1$, and vice--versa: in other words, 
\begin{equation*}
    \E_{D,2}\cap\mathcal{V}=\{h\in S^2D^*\mid\rk(h)=1  \} \, ,
\end{equation*}
which leads to
\begin{equation}\label{eqVeronSing}
     \Sing(\E_D)\cap\mathcal{V}=(\E_{D,2}\cap\mathcal{V})\cup \E_{D,3} \, .
\end{equation}
To prove the first part of claim (P), we observe that the composition of the inclusion $\mathcal{V}\subset X$ with the Pl\"ucker embedding $X\longrightarrow\p(\Lambda_0^3\CC)$ is the map
\begin{eqnarray}
    S^2D^*\simeq\mathcal{V}&\stackrel{\iota}{\longrightarrow}&\p(\Lambda_0^3\CC)=\p(\C^*\oplus S^2D^*\oplus S^2D\oplus\C)\, ,\nonumber\\
    h&\longrightarrow& [1:h:h^\sharp:\det h]\, ,\label{eqPluckDeomposed}
\end{eqnarray}
where we decomposed $\Lambda_0^3\CC$ into irreducible $\SL(D)$--modules, see \cite[Section 1.5]{Gutt2019} (the reader will recognize in \eqref{eqPluckDeomposed} the same expression that appeared in  \eqref{eq.MAE}, after replacing $h$ with $u$).\par
Therefore, we can regard the image of   $\mathcal{V}\cap(\E_{D,2}\cup \E_{D,3})$ via  \eqref{eqPluckDeomposed} as a subset of the projective subspace $\p(\C^*\oplus S^2D^*)$:
\begin{equation}\label{eqImageIota}
    \iota(\mathcal{V}\cap(\E_{D,2}\cup \E_{D,3}))=\{ [1:h]\mid \rk(h)=0,1\}\subset \p(\C^*\oplus S^2D^*)\, .
\end{equation}
In particular, $\iota(D)=[1:0]=\p(\C^*)$ is the unique point in the image of  $ \E_{D,3}$: the \emph{projective} cone with vertex $[1:0]$ over the  
%
%
Veronese surface 
\begin{equation*}
        v_2:\p D^*\longrightarrow \p(S^2D^*)\, 
    \end{equation*}
in $\p(S^2D^*)$, that is
\begin{equation}\label{eqConeVero}
    \bigcup_{x\in v_2(\p D^*)}\p^1([1:0],x)=\underset{\textrm{three--dimensional}}{\underbrace{\big\{[1:h]\mid h\in\widehat{v_2(\p D^*)} \big\}}}\cup \underset{\textrm{two--dimensional}}{\underbrace{v_2(\p D^*)}}\, ,
\end{equation}
contains then  \eqref{eqImageIota} as the (open and dense) three--dimensional subset, where by $\widehat{v_2(\p D^*)}\subset S^2D^*$ we meant the \emph{affine} cone.\par
Since $\overline{\mathcal{V}}=X$ and $\E_{D,2}\cup \E_{D,3}$ is closed, by passing to the closures, we see that $\iota(\E_{D,2}\cup \E_{D,3})$ coincides with the whole of \eqref{eqConeVero}; similarly, from \eqref{eqVeronSing} 
we obtain $\Sing(\E_D)\supseteq \E_{D,2}\cup \E_{D,3}$ and then the equality. But $\iota(\E_{D,2}\cup \E_{D,3})$ is the projective cone over the Veronese surface and $\iota$ is an embedding, therefore we are done.\par

    %
    In order to finish the proof, it suffices to compute the rank of $\widetilde{\eta}$ for each of the four 3--forms $\eta$ listed in Section \ref{subListNormalForms}; to this end, we observe that the linear subspace \eqref{defTangSpaceEL} can be recast as $\ker dF$, where $F$ is any function, such that, locally, $\E_\eta=\{F=0\}$: a choice of such a function can be seen in the last column of the aforementioned table. Therefore,  $\widetilde{\eta}$ and $dF$ are proportional and, as such, regarded as quadratic forms, they have the same rank: it then remains to use formula \eqref{eq:symbol.intro} to compute the symbol of $\E_\eta$ at a smooth  point of the open and dense subset where equality $\E_\eta=\{F=0\}$ holds. In the cases (L), (G) and (P) one immediately obtains, retaining the same notation as in \eqref{eq:symbol.intro},   $\eta_1^2+\eta_2^2+\eta_3^2$, $\eta_1\eta_2$ and $\eta_1^2$, respectively, whose rank is manifestly 3, 2 and 1, respectively. The last case (O) can be worked out in the same fashion.
\end{proof}

\subsection{The cocharacteristic variety via the momentum map}\label{subsecfinale}
We go back to the idea  of the cocharacteristic variety of a  Monge--Amp\`ere equation that we have introduced earlier in Section \ref{secSubSecCoCar}, and we prove the next   original result of the paper.\par
The construction that has led us to the section \eqref{eqLambdaTwoEta} is manifestly $\Sp(\CC)$--invariant, so that, if we compute $\Lambda^2(\widetilde{\eta})$ for any of the four representatives $\eta$ listed in Section \ref{subListNormalForms} above, we will have obtained all four possible isomorphism classes of such sections. The very same approach has been used in  Theorem \ref{thFirstOriginalResult}, when we   established that the KLR invariant $q(\eta)$ is, up to a projective factor, the same as the Hitchin moment map $\varpi(\eta)$, both computed on the 3--form $\eta$: we finish up this paper by establishing another (projective) equality between  $q(\eta)$ (or $\varpi(\eta)$), and $\Lambda^2(\widetilde{\eta})$, showing at the same time that all the three of them cut out the cocharacteristic variety of $\E_\eta$.\par
To make such last claim rigorous, it must be underlined that there is a $\Sp(\CC)$--equivariant linear embedding
\begin{equation}\label{eqSectionMap}
    S^2(\CC^*)\stackrel{s}{\longrightarrow} \Gamma((\E_\eta)_\textrm{sm},S^2(\mathcal{L}^*))
\end{equation}
sending a quadratic form $q$ to the section $s(q)$ defined in the natural way:
\begin{equation*}
    s(q)(L):=i_L^*(q)\,, \quad \forall L\in(\E_\eta)_{\textrm{sm}}\, ,
\end{equation*}
where $i_L:L\subset\CC$ denotes the  embedding. We have then the following  diagram of $\Sp(\CC)$--equivariant   maps:
\begin{equation}\label{eqCommDiagFinal}
    \xymatrix{
    & **[r]\p(\Lambda_0^3(\CC))\ar@{-->}[dl]_{\varpi=[ q]}\ar@{-->}[d]^{[\Lambda^2{(\widetilde{\,\cdot\,})}]}\\
    \p(S^2(\CC^*))\ar[r]^s & **[r] \p\big(\Gamma((\E_\eta)_\textrm{sm},S^2(\mathcal{L}^*))\big)\, .
    }
\end{equation}
The commutativity of \eqref{eqCommDiagFinal} can be checked manually by a case--by--case technique. The departing point is always an unknown quadratic form on $\CC$:
\begin{equation*}
    q=r^{ij}e_{i+3}e_{j+3}+s^{ij}e_{i+3}e_{j}+t^{ij}e_ie_j\, ,
\end{equation*}
that is, a  $6\times 6$ symmetric matrix presented in a $3\times 3$ block form, cf. \eqref{eqBasiViViDuale}.\par
Since the other cases are formally analogous, we will be focusing only on the  orbit ``O'', i.e., on the equation $\det\|u_{ij}\|=1$ that corresponds 
to   $\eta=e_{123}+e_{456}$, cf. \eqref{eqEqLimit1}. As a first step, we choose a local parametrization of $\E_\eta$, for example, by calculating  $u_{11}$ as a function of the remaining parameters: 
\begin{equation*}
    u_{11}=f(u_{12},u_{13},u_{22},u_{23},u_{33})=\frac{ u_{12}^2  u_{33}-2  u_{12}  u_{13}
    u_{23}+ u_{13}^2  u_{22}+1}{ u_{22}
    u_{33}- u_{23}^2}\, .
\end{equation*}
Next,  for any point $L\in\E_\eta$ lying in this local parametrization, i.e., for any
\begin{equation}\label{eqPuntoParticolare}
    L=\Span{e_1+f(u_{12},u_{13},u_{22},u_{23},u_{33})e_4+u_{12}e_5+u_{13}e_6, e_2+\sum_{i=1}^3u_{2i}e_{i+3},e_3+\sum_{i=1}^3u_{3i}e_{i+3}}\, ,
\end{equation}
we compute the restriction $i_L^*(q)$, which is then a $3\times 3$ symmetric matrix depending upon  the five parameters $u_{12},u_{13},u_{22},u_{23},u_{33}$.\par
Such a matrix can be computed by hand (better if aided by a computer algebra software) though, due to its lengthy expression, we leave it aside and we pass to the computation of the symbol of the equation $\E_\eta$ at the same point $L$ given by \eqref{eqPuntoParticolare}: the result is again  a $3\times 3$ (symmetric) matrix, cf. \eqref{eq:symbol.intro}, depending on the same five parameters $u_{12},u_{13},u_{22},u_{23},u_{33}$; it is not hard to see that its adjoint matrix is given by
\begin{equation}\label{eqDualSymbOnShell}
  \left(
\begin{array}{ccc}
 \frac{ u_{12}^2  u_{33}-2  u_{12}  u_{13}
    u_{23}+ u_{13}^2  u_{22}+1}{ u_{22}
    u_{33}- u_{23}^2} &  u_{12} &  u_{13} \\
  u_{12} &  u_{22} &  u_{23} \\
  u_{13} &  u_{23} &  u_{33} \\
\end{array}
\right)
\end{equation}
The desired commuativity of \eqref{eqCommDiagFinal} (at least on the orbit ``O'') means precisely that the requirement that $i_L^*(q)$ and the matrix \eqref{eqDualSymbOnShell} above cut out the same quadric in $L$, \emph{for all the points} $L\in(\E_\eta)_{\textrm{sm}}$, singles out a unique $q$ (up to a projective class).\par
By imposing that $i_L^*(q)$ coincide with  \eqref{eqDualSymbOnShell} for all the values of the five parameters $u_{12},u_{13},u_{22},u_{23},u_{33}$, one obtains a system of 83 linear equations in the 21 entries of $q$, whose left--hand sides are listed below:
\begin{align*}
 &-t^{11}s^{11}-1,-r^{11},1-s^{11},2
   r^{11},t^{13},-s^{13},-t^{13},2
   s^{13},-t^{33},-2 t^{11},2
   t^{33},t^{12},-s^{12},\\
   &-t^{12},2 s^{12},4
   t^{11},2-2 s^{11},-t^{23},2 s^{11}-2,2 t^{23},-2
   t^{13},2 t^{13},-t^{22},2 t^{22},-2 t^{12},2
   t^{12},-4
   t^{11},\\
   &\frac{s^{21}}{2},\frac{t^{13}}{2},-\frac{r^{12}}{
   2},-\frac{s^{13}}{2},-\frac{s^{21}}{2},-\frac{t^{13}}{2},\frac{t^{12}}{2},r^{12},s^{13},-\frac{s^{12}}{2},-\frac{t^{12}}{2},s^{12},-\frac{s^{23}}{2},s^{23},-\frac{t^{23}}{2},t^{23},t^{11},-\frac{s^{11}}{2}-\frac{ {s^{22}}}{2}+1
   ,\\
   &s^{11}+ {s^{22}}-2,-s^{21},-\frac{3
   t^{13}}{2},s^{21},\frac{3 t^{12}}{2},2
   t^{11},\frac{s^{31}}{2},-\frac{r^{13}}{2},-\frac{s^{31}}
   {2},r^{13},-\frac{s^{11}}{2}-\frac{s^{33}}{2}+1,s^{11}+{s^{33}}-2,\\
   &\frac{3
   t^{13}}{2},-\frac{s^{32}}{2},s^{32},-s^{31},-\frac{3
   t^{12}}{2},s^{31},-r^{22},-s^{23},2 r^{22},2
   s^{23},1-s^{22},2 s^{22}-2,2
   s^{21},-\frac{r^{23}}{2},\\
   &\frac{1}{2}
   (-s^{22}-s^{33}+2),r^{23},s^{22}+s^{33}-2,\frac{{t^{23}}}{2},-r^{33},1-s^{33},-s^{32},2 r^{33},2
   s^{33}-2,2 s^{32},2 s^{31}\, .
\end{align*}
Miraculously, the system is compatible and its  solution is $s^{11}=s^{22}=s^{33}=1$, the other entries being zero, that is, $q$ is proportional to a quadratic form of type $q^{(3)}$, which is precisely the image via the   KLR--Hitchin map of the form $\eta$, see Section \ref{subListNormalForms} above.\par 
The computations for the representatives of the orbits ``L'' and ``G'' are analogous and we omit them; the orbit ``P'' is the singular locus of the map $\varpi$ (Proposition \ref{propSingLoc}) and, since $\widetilde{\eta}$ and rank one (Theorem \ref{thMain2}), its cofactor matrix $\Lambda^2(\widetilde{\eta})$ must  be zero: hence, the orbit ``P'' is the singular locus of the map $[\Lambda^2(\widetilde{\,\cdot\, })]$ as well, and the commutativity of \eqref{eqCommDiagFinal} has been proved.  
On a deeper level, this  surprising compatibility  is a  cohomological feature of homogeneous bundles over $X$: in the Appendix we prove that the map $s$ of diagram \eqref{eqCommDiagFinal} is actually an isomorphism, see Section \ref{secAppendix1} below.
\begin{corollary}\label{corEquivalencyMomentCoChar}
For any $[\eta]\in\p(\Lambda^3_0(\CC^*))$ the  cocharacteristic variety $\sigma^\vee\subset\p(\CC)$ of the Monge--Amp\`ere equation $\E_\eta$ is cut out by  (the projective class of) the quadratic form $\varpi([\eta])\in \p (S^2(\CC^*))$. In particular:
\begin{itemize}
    \item[(O)] $\sigma^\vee$ is an irreducible, nonsingular and non--degenerate (rank--six) quadric;
    \item[(L)] $\sigma^\vee$ is an irreducible, nonsingular and degenerate rank--three quadric;
    \item[(G)] $\sigma^\vee$ is a reducible, nonsingular and degenerate rank--one quadric: it coincides with the hyperplane $\p(\ell_D^\perp)$, where $D\in\Gr(3,\CC)$ is such that $\E_D=\E_\eta$; 
    \item[(P)] $\sigma^\vee$ is trivial, i.e., it has rank zero and  $\sigma^\vee=\p(\CC)$.
\end{itemize}
\end{corollary}
\begin{proof}
    Commutativity   of diagram \eqref{eqCommDiagFinal} tell us that, up to a projective factor,  $\Lambda^2(\widetilde{\eta})$ is the section $s(q(\eta))$, for any of the four representatives $\eta$ and, in virtue of its $\Sp(\CC)$--equivariancy, for all the elements of $\Lambda^3_0(\CC^*)$.\par
    The cocharacterisitic variety $\sigma^\vee$, that has been  defined above by \eqref{defCocVar}, can be equivalently given as   the (closure of the) image of the natural projection  
    \begin{equation*}
        \xymatrix{
        **[l](\E_\eta)_\textrm{sm}\times\CC\supseteq \mathcal{L}|_{(\E_\eta)_{\textrm{sm}}}\ar[d]\ar[r]&\CC\\
        **[l](\E_\eta)_\textrm{sm}\, 
        }
    \end{equation*}
    over $\CC$ 
    of the (conic) sub--bundle of $\mathcal{L}|_{(\E_\eta)_{\textrm{sm}}}$ cut out by the section $s(q(\eta))$, that is the same as the section $\Lambda^2(\widetilde{\eta})$: indeed, at each nonsingular  point $L\in\E_\eta$, the conic subset of $L$ cut out by the equation $\Lambda^2(\widetilde{\eta})=0$ is the cone over the projective dual of the characteristic variety $\sigma_L$ of $\E_\eta$ at $L$, see \eqref{eqCoCarPointL}.\par
    Since the section $s(q(\eta))$ comes from the quadratic form $q(\eta)$, the latter is precisely the quadratic form cutting out $\sigma^\vee$ in $\CC$: this proves the first part of the claim.\par
    The second part can be carried out by working out all the four cases listed in Section \ref{subListNormalForms} above: for the orbit ``P'' the claim is evident; in the cases ``O'' the quadric has manifestly rank six: as such, it is nonsingular and non--degenerate and, therefore, it is irreducible; in the case ``L'' we see that $\sigma^\vee$ is the pre--image of a rank--three (and, as such, non--degenerate, nonsingular and irreducible) quadric on the first summand $V$ of $\CC=V\oplus V^*$.\par
    In the last case ``G'' reducibility it obvious, since $\{e_1^2=0\}$ is a linear hyperplane; it remains to prove that such a hyperplane is the symplectic orthogonal to the line $\ell_D$, where $D$ is one of the only two elements of $\Gr(3,\CC)$ that determine the same equation $\E_\eta$, see \eqref{eqMappaMappinaGr3LGr}. Direct computations, analogous to those carried out over $\R$ in Section~\ref{sec:cone.structure.u11} above, show that these elements are precisely
    \begin{equation*}
        D=\Span{e_1,e_3,\epsilon^3}\, ,\quad D^\perp=\Span{e_1,e_2,\epsilon^2}\, .
    \end{equation*}
    One readily sees that $\E_\eta=\E_D=\E_{D^\perp}=\{u_{23}=0 \}$. 
    Recalling the definition of $\ell_D$, cf.  \eqref{edDefEllDiPerp}, we see that $\ell_D=\Span{e_1}$, whence 
    \begin{equation*}
        \ell_D^\perp=\Span{e_1,e_2,e_3,\epsilon^2,\epsilon^3}=\ker e_1=\{e_1^2=0\}=\sigma^\vee\, ,
    \end{equation*}
    and the proof is complete.
\end{proof}

Observe that the Theorem \ref{thMain1} we have formulated before  is  the real--differentiable counterpart of the Corollary \ref{corEquivalencyMomentCoChar} we have just proved in the complex--analytic setting: claims (1), (2) and (3) of the former correspond to claims  (O+L), (G) and (P) of the latter. As the (omitted) proof of Theorem \ref{thMain1} is formally analogous to that of Corollary \ref{corEquivalencyMomentCoChar}, we can draw the following conclusion: for a Monge--Amp\`ere equation $\E_\eta$ the KLR invariant and the Hitchin moment map of the 3--form $\eta$, equated to zero, gives the same quadric in $\CC$, that we have called \emph{cocharacteristic variety}, since it consists of the projective duals of the fibers of the well--known characteristic variety; in the case when $\E_\eta$ has \emph{non--denegerate symbol}, the cocharacteristic variety   coincides also with the \emph{contact cone structure} associated with the PDE itself, which can be thought of as a non--linear generalization of Goursat's idea of describing a Monge--Amp\`ere equation via a sub--distribution of the contact distribution on $J^1$; in case of a degenerated symbol, the cocharacteristic variety is the smallest liner subspace containing both $D$ and $D^\perp$, when $\E_\eta=\E_D$, or it even becomes trivial, when $D=D^\perp$.

\section{Conclusions and perspectives}

We have seen that to any Monge--Amp\`ere equation in three independent variables, that we understood as a hypersurface of the $2\Nd$ order  jet space $J^2$,   is associated, at a fixed point of $J^1$, a quadric cone inside the contact space of $J^1$; the latter turns out to be a $6$--dimensional symplectic vector space. We showed that, in the case when the symbol of the Monge--Amp\`ere equation is non--degenerate, the aforementioned   cone coincides with the cocharacteristic variety of the considered equation. This motivated the study of quadratic forms on a $6$--dimensional symplectic vector space, up to symplectic equivalence. \par

It turned out that quadric cones that are contact cone structures of Monge--Amp\`ere equations fill up a narrow subclass. Thus, it is natural to ask where the remaining   cones come from, that is, whether there exists another class of PDEs, necessarily more general than those of Monge--Amp\`ere type, that accounts for the remaining  normal forms of the table of Section \ref{subsCompClass}; in practise, this means solving the following problem: given \emph{any} quadratic form from the aforementioned table, one can consider it zero locus and wonder whether the latter can be realized as the contact cone structure  of some PDE and, moreover, to what extent such a PDE is characterized by its contact cone structure. Among the candidates for a larger class of PDEs there are the equations of Monge--Amp\`ere type of ``higher algebraic degree", that is, PDEs given by an analogous formula to \eqref{eq.MAE}, where, instead of a linear polynomial of the minors of the Hessian matrix, we find a polynomial of higher degree (see also \cite[Section 5]{MR3603758} for more details about this class of PDEs).
For example, in Section \ref{sec:ritorno.u11.minore} we have shown how to recover the equation \eqref{eq.u11.minore} starting from the quadric cone \eqref{eq:cone.structure.u11.minore}: performing similar computations, starting from the cone variety corresponding to the normal form $q^{(2)}$ of the aforementioned table, that is
\begin{equation}\label{eqFormaMozzata}
    z^2q_2+ z^3q_3=0\,,
\end{equation}
we obtain the equation
\begin{equation}\label{eqEquazioneFormaMozzata}
    u_{12}(u_{13}u_{23}-u_{12} u_{33}) + u_{13}(u_{12}u_{23}-u_{13}u_{22})=0
\end{equation}
i.e., the variety \eqref{eqFormaMozzata} is the contact cone structure  of the PDE \eqref{eqEquazioneFormaMozzata}.
As it happens, though,  PDE \eqref{eqEquazioneFormaMozzata} is no longer a Monge--Amp\`ere equation (i.e., at any point of $J^1$, it is not a hyperplane section of the Lagrangian Grassmanian $\LL(3,\CC)$), but a \emph{quadratic} Monge--Amp\`ere equation instead, that is a PDE of Monge--Amp\`ere type of algebraic degree equal to 2 (or, in other words, at a point of $J^1$, it is a \emph{hyperquadric} section of the Lagrangian Grassmanian $\LL(3,6)$), being the left hand side of \eqref{eqEquazioneFormaMozzata} a polynomial of second degree in the minors of the Hessian matrix $\|u_{ij}\|$.
%
It is then natural to ask, whether there exists a minimal algebraic degree $d$, such that the class of  PDEs of Monge--Amp\`ere type of  algebraic degree equal to $d$ contains \emph{all} PDEs whose conic structure is given by one of the quadratic forms listed in the table of Section \ref{subsCompClass}.\par
Another important class of $2\Nd$ order PDEs is the class of so--called \emph{hydrodynamically integrable} ($2\Nd$ order) PDEs, that have been instensively studied, from the point of view of hypersurfaces in a Lagrangian Grassmannian, by  E.~Ferapontov and his collaborators \cite{MR2587572}: in particular, they have discovered  that the 21--dimensional group $\Sp(6,\R)$ has an open orbit in the space that parametrises  these PDEs, which in turns implies that such a space has dimension 21. In spite of the fact that the space of integrable PDEs and the space of quadrics on $\CC$ have the same dimension, the former cannot be the larger class of PDEs we are looking for, since it does not contain the class of Monge--Amp\`ere equations: actually they intersect along the orbit made of the linearizable Monge--Amp\`ere equations, that is, the orbit that we have labeled by ``L'' in Section \ref{secLinSec} above.  \par
%
%
One may then also wonder, whether it is possible to characterise the class of integrable $2\Nd$ order PDEs in terms of their contact cone structures.

\section{Appendix: cohomology of homogeneous bundles over $\LL(3,\CC)$}\label{secAppendix1}
 
 To prove the following result we employ the Borel--Weil--Bott theorem, as formulated in \cite{10.1215/S0012-7094-06-13233-7}, and we will follow the notation therein. 
 \begin{proposition}\label{propFinale}
     The natural sheaf morphism
\begin{equation}
    S^2(\OO\otimes\CC^*)|_{\E_\eta}\stackrel{p}{\longrightarrow} S^2(\sL^*)|_{\E_\eta}
\end{equation}
induces an isomorphism between the corresponding spaces of sections, and they can be futher identified with the irreducible $\Sp(\CC)$--representation $W_{(2,0,0)}$, that is, $S^2(\CC)$.
 \end{proposition}
Let us observe that, as a consequence of such a result, the arrow $s$ of diagram \eqref{eqCommDiagFinal} above must be an isomorphism: indeed, Proposition \ref{propFinale} was intended to provide a solid theoretical background to the computation performed throughout Section \ref{subsecfinale}.
 \begin{proof}[Proof of the Proposition \ref{propFinale}]
 We will need the hyperplane exact sequence
 \begin{equation}\label{eqSeqSheaf1}
     0\longrightarrow \OO_X(-1) \longrightarrow \OO_X \longrightarrow \OO_{\E_\eta}\longrightarrow 0\, ,
 \end{equation}
 together with the exact sequence
 \begin{equation}\label{eqSeqSheaf2}
     0\longrightarrow Q^*\cdot\CC^*\longrightarrow S^2(\OO\otimes\CC^*)\longrightarrow S^2(\sL^*)\longrightarrow 0\, ,
 \end{equation}
 coming from the dual of the tautological sequence
  \begin{equation}
     0\longrightarrow \sL 
     \longrightarrow  \OO\otimes\CC \longrightarrow Q\longrightarrow 0\,  .
 \end{equation}
 By combining sequence \eqref{eqSeqSheaf1} and \eqref{eqSeqSheaf2} we obtain a commutative diagram  of coherent sheaves on $X$ with exact rows and columns:
 \begin{equation}\label{eqBigDiag}
     \xymatrix{
      & 0\ar[d]&0\ar[d]&0\ar[d]& \\
      0\ar[r] & Q^*\cdot \CC^*(-1) \ar[r]\ar[d] & Q^*\cdot \CC^* \ar[r]\ar[d] & Q^*\cdot \CC^*|_{\E_\eta} \ar[r]\ar[d] & 0
    \\
    0\ar[r] & S^2(\OO\otimes \CC^*)(-1) \ar[r]\ar[d] & S^2(\OO\otimes \CC^* )\ar[r]\ar[d] & S^2(\OO\otimes \CC^*)|_{\E_\eta} \ar[r]\ar[d] & 0\\
    0\ar[r] & S^2(\sL^*)(-1) \ar[r]\ar[d] & S^2(\sL^*) \ar[r]\ar[d] & S^2(\sL^*)|_{\E_\eta} \ar[r]\ar[d] & 0\\
    & 0 &0 &0 & \\
     }
 \end{equation}
The symbol $Q^*\cdot \CC^*$ stands for the sheaf that fits into the exact sequence
\begin{equation}
    0\longrightarrow \Lambda^2(Q^*)\longrightarrow Q^*\otimes\CC^*\longrightarrow Q^*\cdot \CC^*\longrightarrow 0\, .
\end{equation}
 To each short exact sequence that appears in diagram \eqref{eqBigDiag} above we can associate the long exact sequence in cohomology: we will show that  the only nonzero cohomology groups are the zeroth groups of right lower $2\times 2$ square of the diagram. \par

 Since $\sL$, $Q$, $\OO_X(-1)$ and their various products are $\Sp(\CC)$--homogeneous,  Borel--Weil--Bott theorem allows us to compute their cohomologies by weight considerations.   We recall that the Lagrangian Grassmanian $X$ is a $\Sp(\CC)$--homogeneous space, whose parabolic subgroup $P$ corresponds to the marked Dynkin diagram
  \begin{equation*}
      \xymatrix{
      \overset{}{\bullet} \ar@{-}[r]&\overset{}{\bullet} &\overset{}{\circ}  \ar@{=}[l]|{\scalebox{1.6}{$<$}}
      }\, .
  \end{equation*}
 We recall that the fundamental Weyl chamber $D\subset\gh^*$ is the positive cone spanned by the fundamental weights, see Remark \ref{remFundWeights}, and we define $\delta$ as their sum, i.e.,
 \begin{equation*}
     \delta=3h_1+2h_2+h_3\, .
 \end{equation*}
 By the \emph{index} $\ind(\lambda)$ of the weight $\lambda$ we mean the smallest number of simple reflections needed to move  $\lambda$ to the fundamental Weyl chamber $D$; a representative of the orbit (with respect to the Weyl group action) of $\lambda$ belonging to $D$ will be denoted by the symbol $w(\lambda)$.\par
 Since irreducible $\Sp(\CC)$--homogeneous  bundles on $X$ correspond to irreducible representations of $P$, they are determined by their highest weights. To begin with, the tautological bundle $\sL$ has weights $-h_1, -h_2, -h_3$, the minus sign coming from the very definition of associated vector bundle, where we act on the fiber by   the inverse. In this case, the highest weight is $-h_3$.\par
 Since, in the case of $\LG(3,6)$, the bundle $\sL$ is identified with $Q^*$, the weights are the same. For $\sL^*$ the weights will be $h_1, h_2, h_3$, with $h_1$ being the highest; the line bundle $\OO_X(-1)$ has (highest) weight $-h_1-h_2-h_3$.   Therefore, we can obtain the highest weight of various products: for $S^2(\sL^*)$ we obtain $2h_1$, for $S^2\sL^*(-1)$ we obtain $h_1-h_2-h_3$, for $\Lambda^2(Q^*)$ we obtain $-h_2-h_3$ and for $\Lambda^2Q^*(-1)$ we obtain $-h_1-2h_2-2h_3$.\par
 By Borel--Weil--Bott theorem, if $w(\lambda+\delta)$ belongs to the boundary $\delta D$ of $D$, then the bundle associated to $\lambda$ is cohomologically trivial (or, as some authors say,   \emph{immaculate}), i.e., $H^i=0$ for all $i$. In our case, easy computations show that $\sL=Q^*$, $\Lambda^2(Q^*)$, $S^2\sL^*(-1)$, $S^2(\OO\otimes\CC^*)(-1)$, $Q^*(-1)$ and  $\Lambda^2Q^*(-1)$ are all immaculate; by using the cohomology long exact sequence, we obtain that $Q^*\cdot\CC^*(-1)$ and $Q^*\cdot\CC^*$ are immaculate as well. Then, by the same argument, the immaculateness $S^2(\OO\otimes\CC^*)(-1)$ and $Q^*\cdot\CC^*|_{\E_\eta}$ also follows.\par
 Therefore we have shown that only the bundles in the lower right $2\times2$ square can have nonzero cohomologies and by exactness of the cohomology sequences we have for all $i$: $ H^i (X,S^2(\sL^*) ) = H^i(X,S^2(\OO\otimes \CC^*)) = H^i(\E_\eta, S^2(\OO\otimes \CC^*)) = H^i(\E_\eta, S^2(\sL^*))$. To finish the proof it is now enough to observe that any of the nontrivial $H^0$'s is the $\Sp(\CC)$ irreducible representation $W_{(2,0,0)}$: but the claim is obvious for the trivial bundle in the center of diagram \eqref{eqBigDiag}, since the zeroth cohomology of a trivial bundle coincides with the fiber, that is, $S^2(\CC^*)$. This concludes the proof.
 \end{proof}



\bibliographystyle{abbrvnat}
\bibliography{BibUniver}
\end{document}